\documentclass{article}
\usepackage{verbatim}             
\usepackage{fancyhdr}
\usepackage{graphicx}
\usepackage{epstopdf}
\usepackage{amsmath}
\usepackage{float}
\usepackage{amsmath}
\allowdisplaybreaks[4]
\usepackage{amsthm}
\usepackage{amssymb}
\usepackage{bm}
\usepackage{eucal}%big write xila word
\usepackage{color}
\usepackage[table]{xcolor}
\usepackage[figureright]{rotating}
\usepackage[titletoc]{appendix}
\usepackage[figureright]{rotating}
\usepackage{multirow}
\usepackage{subfigure}
\usepackage{tikz}
\usepackage{pgfplots}
\usepackage{algorithmic}
\usepackage{ulem}
\usepackage[ruled,linesnumbered]{algorithm2e}
\usepgfplotslibrary{groupplots}

\graphicspath{{Figures/}}
\newtheoremstyle{plainNoItalics}{}{}{\normalfont}{}{\bfseries}{.}{ }{}

\theoremstyle{plain}
\newtheorem{thm}{Theorem}[section]

\theoremstyle{plainNoItalics}

\newtheorem{rem}[thm]{Remark}

\numberwithin{equation}{section}

\newcommand{\bu}{\mathbf{u}}
\newcommand{\bv}{\mathbf{v}}
\newcommand{\bfw}{\mathbf{w}}
\newcommand{\mf}{\mathbf{f}}
\newcommand{\T}{\text{T}}

\newcommand{\V}{V_h^p}%{\mathcal{V}_h^r}

\newcommand{\rM}{\gamma_1} %{\gamma_M}
\newcommand{\rA}{\gamma_0} %{\gamma_A}

\newcommand{\mcT}{\mathcal{T}}
\newcommand{\mesh}{\mathcal{T}_h}
\newcommand{\Th}{\mathcal{T}_h}
\newcommand{\mcF}{\mathcal{F}}
\newcommand{\Ehi}{\mathcal{E}_{h,i}}
\newcommand{\Khi}{\mathcal{K}_{h,i}}
\newcommand{\Me}{M} %I_M

\newcommand{\mcK}{\mathcal{K}}
\newcommand{\mcE}{\mathcal{E}}
\newcommand{\mcM}{\mathcal{M}}
\newcommand{\mcFM}{\mcF_{h,i}(\Me)} % interior faces to macro element M
\newcommand{\mm}{\mathtt{m}}
\newcommand{\MM}{\mathtt{M}}
\newcommand{\intp}{\pi_h^{\mathcal{M}}}

\newcommand\se[1]{\textcolor{black}{#1}}
\allowdisplaybreaks[4]

\title{A bound preserving cut discontinuous Galerkin method for one dimensional  hyperbolic conservation laws 
}
\author{Pei Fu\footnote{School of Mathematics,  Nanjing University of Aeronautics and Astronautics, 211106 Nanjing, China. E-mail: fu.pei@nuaa.edu.cn.},
Gunilla Kreiss\footnote{Division of Scientific Computing, Department of Information Technology, Uppsala University, SE-75105 Uppsala, Sweden. E-mail: gunilla.kreiss@it.uu.se.},
Sara Zahedi\footnote{Department of Mathematics, KTH Royal Institute of Technology, SE-10044 Stockholm, Sweden. Email: sara.zahedi@math.kth.se.}
}
\date{}
\begin{document}
\maketitle
\begin{abstract}
 In this paper we present a family of high order cut finite element methods with bound preserving properties for hyperbolic conservation laws in one space dimension. The methods are based on the discontinuous Galerkin framework and use a regular background mesh, where interior boundaries are allowed to cut through the mesh arbitrarily. Our methods include ghost penalty stabilization to handle small cut elements and a new reconstruction of the approximation on macro-elements, which are local patches consisting of cut and un-cut neighboring elements that are connected by stabilization. We show that the reconstructed solution retains conservation and order of convergence.  
 Our lowest order scheme results in a piecewise constant solution that satisfies a maximum principle for scalar hyperbolic conservation laws. 
 When the lowest order scheme is  applied to the Euler equations,  the scheme is positivity preserving in the sense that positivity of pressure and density are retained. For the high order schemes,   suitable bound preserving limiters are applied to the reconstructed solution on macro-elements. In the scalar case,  a maximum principle limiter is applied, which ensures that  the limited approximation satisfies the maximum principle. Correspondingly, we use a positivity preserving limiter for the Euler equations, and show that our scheme is positivity preserving.  In the presence of shocks additional limiting is needed to avoid oscillations, hence we apply a standard TVB limiter to the reconstructed solution.  The time step restrictions are of the same order as for the corresponding discontinuous Galerkin methods on the background mesh.  Numerical computations illustrate accuracy, bound preservation, and shock capturing capabilities of the proposed schemes.  
 \end{abstract}

%\subjclass{65M60, 65M12, 65M30, 35L05, 35L50}

\noindent {\bf Keywords:} Hyperbolic conservation laws, discontinuous Galerkin method,  cut elements,  satisfying maximum principle,  positivity preserving
%\maketitle

%introduction
\section{Introduction} 
\label{sec:introduction}
In this paper we  consider problems in one space dimension, both scalar hyperbolic conservation laws and the important hyperbolic system of  equations known as the Euler equations of gas dynamics.  These are of the form
\begin{equation}\label{eq:model:1d}
\left\{
  \begin{array}{ll}
    \mathbf{u}_t+\partial_x \mathbf{f(u)}=0, &x\in\Omega,t>0, \\
    \mathbf{u}(x,0)=\mathbf{u_0}(x), & x\in\Omega, 
  \end{array}
\right.
\end{equation}
with suitable boundary conditions. Here, $\Omega \subset \mathbb{R}$, $\bu=(u_1,\cdots,u_n)^\T$ is a vector of real valued functions, which represents conservative variables, and $\mathbf{f(u)}=(f_1(\bu),\cdots,f_n(\bu))^\T$ is the vector  valued flux function. We assume that the corresponding Jacobian  matrix has $n$ real eigenvalues and $n$ complete eigenvectors.   
Solutions of such systems may develop shocks or discontinuities even if the initial data is smooth. It therefore becomes necessary to consider weak solutions,  and as usual we seek the physical relevant  weak solution which satisfies an entropy condition. 

An important property of the entropy solution for a scalar equation of the form \eqref{eq:model:1d}  is the maximum principle, which for Cauchy problems means 
\begin{align}
\text{ if } \bu_0(x)\in[\mm,\MM] \text{ then }
\bu(x,t)\in[\mm,\MM].
\end{align}
The solution of the Euler equations usually does not satisfy the maximum principle, but density and pressure should remain positive. With negative density or pressure hyperbolicity is lost and the Euler system becomes ill-posed, which causes severe numerical difficulties. For initial boundary-value problems boundary conditions must satisfy corresponding restraints for the maximum principle, or positivity, to be valid.  It has been noted that high order methods without bound preserving properties often suffer from numerical instability or nonphysical features when applied to problems with low density or high Mach number, see for example \cite{ha2005numerical}.  

A finite element method (FEM) based on discontinuous piecewise polynomial spaces is commonly called a discontinuous Galerkin (DG) method. The DG methodology coupled to explicit Runge-Kutta time discretizations and limiters, has  successfully been applied to nonlinear time dependent hyperbolic conservation laws, see e.g. \cite{ShuDG2,ShuDG3,ShuDG5}. For more details on DG methods, we refer to \cite{hesthaven2007nodal,shu2009discontinuous}. 
Most finite element methods require the mesh to be aligned to boundaries and material interfaces, and to achieve the optimal accuracy the mesh quality needs to be high.

For boundaries and material interfaces with complicated shapes one approach is to start with a  Cartesian background mesh and use agglomeration to create a body fitted mesh of good quality.  An alternative approach is to allow boundaries and interfaces to cut through the  Cartesian background mesh arbitrarily. This is the approach we take. In this paper we develop an unfitted discretization for hyperbolic PDEs based on standard  Cartesian finite elements that preserves important properties.  Small cut elements may  cause problems, including ill-conditioned linear systems and severe time-step restrictions. Various techniques have been introduced to handle such difficulties. One approach often used in unfitted methods based on DG is cell merging,  where new elements, still unfitted but of sufficient size are created by merging small cut elements with their neighbours ~\cite{Johansson2013,kummer2017extended,modisette2010toward,muller2017high,Qin2013,chen2021adaptive,CHEN2023112384}. A common technique in connection with Cut Finite Element Methods (CutFEM) is to add ghost penalty stabilization terms in the weak form ~\cite{burman2010ghost,burman2012fictitious,massjung2012unfitted,schoeder2020high,gurkan2020stabilized}. In CutFEM the physical domain is embedded into a computational domain equipped with a quasi-uniform mesh. Elements that have an intersection with the domain of interest define the active mesh and associated to that is a finite dimensional function space and a weak form, which together define the numerical scheme \cite{massing2014stabilized,hansbo2014cut,frachon2019cut,sticko2016stabilized}. Interface and boundary conditions are  imposed weakly, when the mesh is unfitted.

 Recently,  some methods based on the DG framework with time-step restrictions that are independent of the sizes of smallest elements have been developed for time-dependent hyperbolic conservation laws. 
 In \cite{engwer2020stabilized,may2022dod} methods are presented that include stabilization of small elements, designed to restore proper domains of dependence. In \cite{giuliani2022two} small elements are stabilized by a redistribution step, which was first introduced in \cite{berger2021state}. In \cite{fu2021high} a family of high-order cut discontinuous Galerkin (Cut-DG) methods with ghost penalty stabilization for non-linear scalar hyperbolic problems is presented. These methods are proven to be high order accurate and $L^2$-stable under  time-step restrictions  that are independent of cut sizes. Total variation stability, which is important when considering problems with shocks,  is also established, but only for the piecewise constant scheme. In \cite{fu2022high} a similar methodology is used, but the focus is on conservation at material interfaces with discontinuous fluxes, which cut arbitrarily through elements in one and two space dimensions.  

In this paper we will develop and analyze a family of high order bound-preserving unfitted discontinuous Galerkin (DG) methods  for nonlinear hyperbolic conservation laws.  We will consider  (\ref{eq:model:1d}) with several artificial subdomain interfaces, a problem which is equivalent to (1.1) on the original domain. The computational mesh, which will be introduced later, is not fitted to these interfaces. We have chosen this model problem as a one dimensional model for problems in higher spatial dimension with a physical boundary or a physical interface that  cuts arbitrarily through the computational mesh  and creates increasingly many cut elements as the mesh is refined, see Figure \ref{cutmesh2d}.

The contribution in this paper is a method for applying limiters in connection with unfitted discretizations such as in~\cite{fu2021high} so that  bound preservation is achieved  for piecewise constant schemes as well as for high order schemes.  
We illustrate with numerical experiments (see Figure 4, 6 in Section  \ref{sec:numres}) that it is not sufficient to directly apply limiters on the numerical solution on the unfitted computational mesh.  
We show that the desired properties can be achieved on macro-elements. Macro-elements~\cite{larson2021conservative} consist of cut and un-cut neighbouring elements in the computational mesh so that each macro-element always has a large intersection with the domain of interest. We propose a reconstruction of the approximated solution on macro-elements based on polynomial extension that preserves conservation and order of accuracy. If reconstruction is applied everywhere the proposed scheme corresponds to cell merging with a particular basis on the macro-element. To minimize errors the reconstruction on a macro-element should  be  applied only when  limiting  is needed on an element belonging to the macro-element. This is possible when the ill-conditioning and time step restrictions due to small cut cells are handled by ghost penalty terms.  

 In the piecewise constant case we can show that under a suitable time step restriction the reconstructed solutions in the scalar case and for the Euler system, satisfy the maximum principle and the positivity of density and pressure, respectively.  For the higher order schemes we can show that the average values of the reconstructed solutions on the macro-elements satisfy the corresponding bounds. Thus, the bound preserving limiters developed by Zhang et. al \cite{zhang2010maximum,zhang2010positivity} can be applied to ensure that the discrete approximations satisfy the pointwise bounds. 

The bound preserving limiter is not enough to control oscillations near shocks. There we need additional limiting. In our work, we use a TVD/TVB  slope limiter commonly applied in combination with DG methods, see \cite{ShuDG2,ShuDG3,ShuDG4,ShuDG5}. Other limiters, such as WENO-limiters~\cite{qiuWenolimiter} are also possible.

The paper is organized as follows. In Section 2 we  introduce the model problem,  the mesh, the finite element function spaces, the Cut-DG discretization, the process for   reconstructing solutions on macro-elements, and some notations. In Section 3  the  scheme  is presented for scalar hyperbolic conservation laws and we show that it satisfies the maximum-principle property when a suitable bound preserving limiter is applied during time evolution.  In Section 4, we present the corresponding scheme for the Euler equations, and analyze the positivity property. 
 In Section 5 numerical computations  are presented, which  support the  theoretical results.  Finally,  a conclusion is given in Section 6.

%The definition of mesh and space
%\input{notationandspacev2}
%The DG scheme
\section{A cut discontinuous Galerkin discretization in space}
In this section we present an unfitted spatial  discretization for hyperbolic conservation laws on the form \eqref{eq:model:1d}.  After presenting the model problem we define the mesh, our finite dimensional function spaces, and the macro-element partition.
\subsection{A model geometry in one space dimension}
To mimic a problem in higher spatial dimensions with a boundary (or an interface) that cuts arbitrarily through a computational mesh, as in Figure \ref{cutmesh2d}, we consider a one dimensional problem with several artificial interfaces, see Figure \ref{fig:cutmiddle1}. In the computations in Section 5 we let the number of cut elements increase with mesh refinement. This one dimensional problem is considerably more challenging than if only cut elements at the two boundaries are present. For completeness, results with unfitted boundaries and with a physical interface  are included in subsections \ref{sec:twoblastwave} and \ref{sec:discontinuous flux}.
\begin{figure}[!htp]
\centering
\includegraphics[width=3.0in]{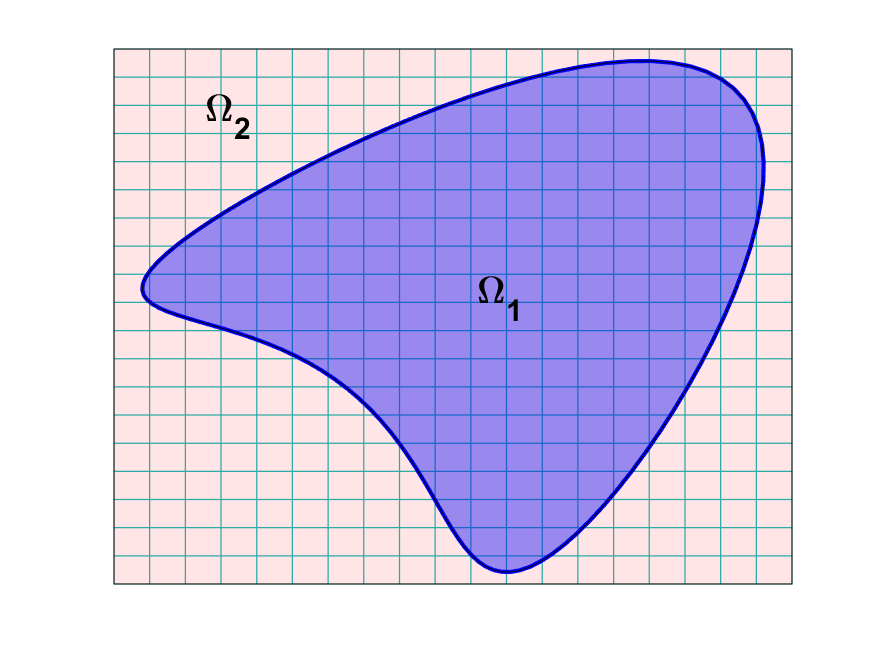}
\caption{An illustration of a regular mesh and an unfitted interface.}\label{cutmesh2d}
\end{figure}

\subsection{Mesh and spaces} 
\label{sec:notapro}
We embed  the physical domain $\Omega=[x_L,x_R]$ into a computational domain $\Omega_h=[x_L, x_R]$, which is discretized into N intervals $I_{j}=[x_{j-\frac{1}{2}},x_{j+\frac{1}{2}}]$, $j=1,\cdots, N$ such that  $x_{L}=x_{\frac{1}{2}}<x_{\frac{3}{2}}<\cdots<x_{N+\frac{1}{2}}=x_{R}$, where $x_{j+\frac{1}{2}}=x_L+jh$ with $h=\frac{x_R-x_L}{N}$. 
\begin{figure}[tbhp]
\centering
\begin{tikzpicture}[xscale=4]
\draw[-][draw=black, very thick] (0,0) -- (.4,0);
\draw[dotted][draw=black, very thick] (0.4,0) -- (0.8,0);
\draw[-][draw=black, very thick] (0.8,0) -- (1.2,0);
\draw[-][draw=red, very thick] (1.2,0) -- (1.35,0);
%\node[red][above] at (1.3,0.1) {${\alpha h}$};
\draw[-][draw=red, very thick] (1.35,0) -- (1.6,0);
%\node[red][above] at (1.5,0.3) {${(1-\alpha)h}$};
\draw[-][draw=black, very thick] (1.6,0) -- (2.0,0);
\draw[dotted][draw=black, very thick] (2.0,0) -- (2.4,0);
\draw[-][draw=red, very thick] (2.4,0) -- (2.8,0);
\draw[dotted][draw=black, very thick] (2.8,0) -- (3.2,0);
\draw[-][draw=black, very thick] (3.2,0) -- (3.6,0);
\draw [thick] (0,-.1) node[below]{$x_{\frac{1}{2}}$} -- (0,0.2);
\draw[black][thick] (0.0,-.2)-- (0.0,.5);
\node[black][above] at (0.0,.5) {$x_L$}; 
\draw [thick] (.4,-.1) node[below]{$x_{\frac{3}{2}}$} -- (0.4,0.2);
\draw [thick] (0.8,-.1) node[below]{$x_{j-\frac{5}{2}}$} -- (0.8,0.2);
\draw [dotted][draw=violet, ultra thick] (1.2,-.1) node[violet][below]{$x_{j-\frac{3}{2}}$} -- (1.2,0.2);
%\draw [dotted][draw=violet, ultra thick] (1.2,-.1) node[below]{\small{$\mcF_{h,1}\ni x_{j-\frac{3}{2}}$}} -- (1.2,0.2);
%\node[violet][below] at (1.25,-.65) {$\in \mcF_{h,1}$}; 
%\draw[coordinate label = {1 at (0.1,0.3)}]
%\draw[annotation below = {$\in \mcF_{h,1}$ at -0.3}] to (1.2,-0.1)
\draw [blue][thick] (1.3,-0.1) node[above]{$\Gamma_1$} -- (1.3,0.1);
\draw[decorate, decoration={brace, mirror},yshift=-4ex, thick]  (0.002,0) -- node[midway, below=3pt] {$\Omega_1$}  (1.3,0);
\draw[decorate, decoration={brace},yshift=2.3ex]  (0.8,0) -- node[above=0.4ex] {$\Me\in\mcM_{h,1}$}  (1.6,0);
\draw[decorate, decoration={brace, mirror},yshift=-5ex, thick]  (1.3,0) -- node[midway, below=3pt] {$\Omega_2$}  (2.58,0);
\draw[decorate, decoration={brace, mirror},yshift=-4ex, thick]  (2.59,0) -- node[midway, below=3pt] {$\Omega_3$}  (3.6,0);
\draw [thick] (1.6,-.1) node[below]{$x_{j-\frac{1}{2}}$} -- (1.6,0.2);
\draw [thick] (2.0,-.1) node[below]{$x_{j+\frac{3}{2}}$} -- (2.0,0.2);
\draw [thick] (2.4,-.1) node[below]{$x_{k-\frac{1}{2}}$} -- (2.4,0.2);
\draw [blue][thick] (2.58,-0.1) node[above]{$\Gamma_2$} -- (2.58,0.1);
\draw [thick] (2.8,-.1) node[below]{$x_{k+\frac{1}{2}}$} -- (2.8,0.2);
\draw [thick] (3.2,-.1) node[below]{$x_{N-\frac{1}{2}}$} -- (3.2,0.2);
\draw [thick] (3.6,-.1) node[below]{$x_{N+\frac{1}{2}}$} -- (3.6,0.2);
\draw [black][thick] (3.6,-.2) -- (3.6,0.5);
\node [black][above] at (3.6,0.5) {$x_R$};
\end{tikzpicture}
\caption{Discretization of the computational domain $\Omega_h=[x_L,x_R]$ into a uniform partition. This background mesh is cut by two interfaces at $x=\Gamma_1$ and $x=\Gamma_2$, dividing the computational domain into three subdomains. The dotted line at $x_{j-\frac{3}{2}}$ indicates the interior edge in $\mcM_{h,1}$ on which stabilization is applied, i.e. $x_{j-\frac{3}{2}}\in\mcF_{h,1}$.} 
\label{fig:cutmiddle1}
\end{figure}
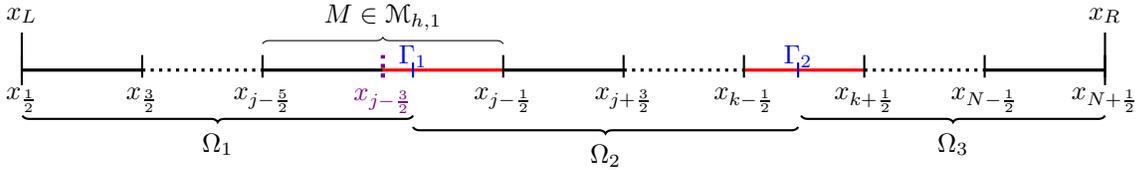
Let
\begin{equation}
  \mesh =\{I_j: I_j \cap\Omega\neq\emptyset: j=1,\cdots, N\}.
\end{equation}
We refer to $\Th$ as the background mesh of the domain $\Omega$. The background mesh may be cut by either the outer domain boundary  $\partial \Omega$ and by interior boundaries (interfaces). In this work we let
\begin{equation}
  \Omega_h=\Omega= \bigcup_{j=1}^N I_j,
\end{equation}
and hence the generated background mesh $\mesh$ aligns with the boundary $\partial \Omega$ but not with interior interfaces. Assume the interior interfaces divide the domain into $N_\Omega$ non-overlapping subdomains $\Omega_i$ such that $\Omega=\cup_{i=1}^{N_\Omega} \Omega_i$. The $N_\Omega-1$ interfaces are positioned at $x=\Gamma_j$, $j=1, \cdots, N_\Omega-1$ and non of them align with the mesh, so we have $N_\Omega-1$ cut elements in the interior of the domain $\Omega$.   For an illustration, see Figure \ref{fig:cutmiddle1}. 

Associated with the subdomains $\Omega_i$ we define active meshes
\begin{equation}
\mcT_{h,i} = \{ I_j \in \mesh : I_j \cap \Omega_i \neq \emptyset\}, \qquad i=1,\cdots, N_\Omega, 
\end{equation}
and active finite element spaces
\begin{equation}
V_{h,i}^p = \{ v \in L^2(\Omega): v|_{I_j} \in \mathbb{P}^p(I_j), \forall I_j \in \mcT_{h,i}  \},
\end{equation}
where $\mathbb{P}^p(I_j)$ denotes the space of discontinuous piecewise polynomials of degree less or equal to $p$ on $I_j$.

Let $\mcE_{h,i}$ be the set of interior edges in the active mesh $\mcT_{h,i}$. The intersections of $I_j$ with $\Omega_i$ are in the set $\mcK_{h,i}$
 \begin{equation}
\mcK_{h,i} = \{ K = I_j \cap \Omega_i : I_j \in \mcT_{h,i} \}. 
%\qquad 
%\mcE_{h,i} = \{ E \cap \Omega_i : E \in \mcE_{h,i} \}
\end{equation} 
Finally, let  
\begin{equation}
\V = \bigoplus_{i=1}^{N_\Omega-1}V_{h,i}^p.
\end{equation}

\subsubsection{Macro-elements} \label{sec:macroel}
We create a macro-element partition $\mcM_{h,i}$ of the active domain 
\begin{equation}
  \Omega_{{h,i}}=\bigcup_{I_j \in \mcT_{h,i}} I_j
\end{equation}
following \cite{larson2021conservative}. The idea is to construct elements with large $\Omega_i$-intersections. To do that each interval $I_j$ in the active mesh $\mcT_{h,i}$ is first classified as having a large or a small intersection with the domain $\Omega_i$. We classify $I_j \in \mcT_{h,i}$ as having a large $\Omega_i$-intersection if 
\begin{equation}\label{eq:largeel}
  \delta  \leq \frac{|I_j \cap \Omega_i|}{|I_j|}.
\end{equation}
Here $0 < \delta\leq 1$ is a constant which is independent of the element and of the mesh size $h$.  
We  assume that  $h$ is small enough so that in each subdomain $\Omega_i$ there is at least one element which is classified as having a large intersection with $\Omega_i$.
To each element with a  large intersection a macro-element is associated. The macro-element consists of either only this element or this element and adjacent elements (connected via a bounded number of internal edges) that are classified as having small intersections with $\Omega_i$. In one space dimension a macro-element will at most consist of three elements. In higher dimensions, if the geometry is smooth and well resolved the macro-element partitioning can be made to avoid macro-elements of very different size, please see Algorithm 1 in  \cite{larson2021conservative}.

We define by $\mcFM$ the set consisting of interior edges of macro-element $\Me\in\mcM_{h,i}$. Stabilization will be applied at these edges. Note that the set $\mcFM$ is empty when the macro-element consists of only one element $I_j$ (with a large $\Omega_i$-intersection).  An example is shown  in Figure \ref{fig:cutmiddle1}, where $M\in\mcM_{h,1}$ is a macro-element consisting of more than one element  in $\Omega_{h,1}$ and $x_{j-\frac{3}{2}}\in \mcF_{h,1}$ is the interior edge of that macro-element.  
In the original ghost-penalty approach, stabilization is applied on all interior edges (with respect to $\Omega_i$) that also belong to cut elements. This is referred to as full stabilization. Using the macro-element partitioning stabilization is applied only on interior edges in each macro-element and never between macro-elements, hence conservation properties of the underlying DG-scheme can be inherited to macro-elements. In addition, with a macro-element stabilization errors are less sensitive to stabilization parameters, see \cite{larson2021conservative}.  With smaller choices of $\delta$ less elements are connected via the stabilization, resulting in a more sparse mass-matrix.  However a small value causes a more severe time step restriction, see e.g. Theorem \ref{thm:scalar:1ord}. In the numerical experiments we have used $\delta=0.2$.

\subsection{The weak formulation}
As a first step in finding an approximate solution to \eqref{eq:model:1d}, we follow \cite{fu2021high,fu2022high} and formulate a semi-discrete weak form: For $t\in [0,T]$ find $\bu_h(\cdot,t)=(\bu_{h,1}(\cdot,t), \cdots, \bu_{h,N_\Omega}(\cdot,t)) $ $\in \V$ such that,
\begin{align}\label{scheme:cutDG2}
((\bu_{h})_t,\bv_h)_\Omega+\rM J_1(\bu_{h,t},\bv_h)+a(\bu_h,\bv_h)+\rA J_0(\bu_h,\bv_h)=0,    \quad \forall \bv_h\in\V,
\end{align}
where 
\begin{align}
(\bu_t,\bv)_\Omega &= \sum_{i=1}^{N_\Omega}  \sum_{K\in \Khi} ((\bu_{i})_t,\bv_{i})_{K}, \\
  a(\bu,\bv)=&-\sum_{i=1}^{N_\Omega}\sum_{K\in \Khi}(\mathbf{f}(\bu_i) ,\partial_x(\bv_i))_{K}
-\sum_{i=1}^{N_\Omega}\sum_{e\in \Ehi}\widehat{\mf}(\bu_i)_e[\bv_i]_{e}
-\sum_{i=1}^{N_\Omega-1} \widehat{\mf}(\bu)_{\Gamma_i}[\bv]_{\Gamma_i},\label{def:auv}\\
J_s(\bu, \bv)=&\sum_{i=1}^{N_\Omega} \sum_{\Me \in \mcM_{h,i}} \sum_{e \in \mcFM} \sum_{k=0}^{p}  w_{k}h^{2k+s}\left[\partial^{k} \bu_i \right]_e\left[\partial^{k} \bv_i \right]_{e}.
\label{stable1}
\end{align}
Here
\begin{equation}\label{eq:def:lf}
\widehat{\mf}(\bu)_\xi= \{\mf(\bu)\}_\xi-\frac{{\lambda}}{2}[\bu]_\xi 
\end{equation}
is the Lax-Friedrichs flux, where $\lambda$ is an estimate of the largest absolute eigenvalue of the  Jacobian matrix $\frac{\partial\mf}{ \partial \bu}$ in the  domain $\Omega$,  
 $[\bfw]_\xi=\bfw^{+}-\bfw^{-}$ represents the jump of the function $\bfw$ at $\xi$ and $\{\bfw\}_\xi=\frac{(\bfw^{+}+\bfw^{-})}{2}$ with $\bfw^+=\lim\limits_{\epsilon\to 0^+}\bfw(\xi+\epsilon)$, and $\bfw^-=\lim\limits_{\epsilon\to 0^+}\bfw(\xi-\epsilon) $ denoting the limit values of $\bfw\in\V$ at $\xi$ from right and left. 
  Note that at $e\in \Ehi$ we have  $[\bv]_e=\bv_{i}^{+}-\bv_{i}^{-}$ and at $\Gamma_i$ we have $[\bv]_{\Gamma_i}=\bv_{i+1}|_{\Gamma_i}-\bv_{i}|_{\Gamma_i}$ with $\bv_{i}\in V^p_{h,i}$. In this paper we apply the  global Lax-Friedrichs flux, i.e. $\lambda$ is constant in the entire domain $\Omega$.  
  The local Lax-Friedrichs flux \cite{doi:10.1137/1.9781611975109} and other positivity preserving  fluxes like HLL flux \cite{inbook}, HLLC flux \cite{doi:10.1137/S1064827593260140} can also be applied.
  
 In the stabilization term \eqref{stable1},  $\omega_k=\frac{1}{(k!)^2(2k+1)}$, we refer to \cite{sticko2019higher} for more details about this choice.  The stabilization parameters $\rA$, $\rM$ are positive constants and need to be chosen sufficiently large to avoid severe time step restrictions. In our numerical experiments $\gamma_0=0.25$ and $\gamma_1=0.75$. Assuming time and space scaled so that the wave speed is ${\mathcal O}(1)$ these values work well. However, we have not optimized these values but with the macro-element stabilization our experience is that the results are not sensitive to these precise values.
  
The initial condition is given by the stabilized $L^2$-projection:
\begin{align}\label{eq:stabilize:L2}
\sum_{i=1}^{N_\Omega} \sum_{K\in \Khi} \left(\bu_h(\cdot,0), \bv_{h}\right)_{K}+\rM J_1(\bu_h(\cdot,0), \bv_h)=\sum_{i=1}^{N_\Omega} \sum_{K\in \Khi}\left(\bu_0, \bv_h\right)_{K}, \qquad\forall \bv_h \in \V.
\end{align}

\subsection{A reconstruction on macro-elements}
In general, fully discrete solutions of \eqref{scheme:cutDG2} do not automatically satisfy the maximum principle for scalar equations (see the left panel in Figure \ref{fig:nonsmoothinitial:p1:overshoot}), or positivity of density and pressure for the Euler equations.   To guarantee that the approximate solution to \eqref{eq:model:1d} has all the desired properties we utilize limiters, but we need to apply them properly. We will need a bound preserving limiter and sometimes when shocks or discontinuities are present we will also need additional  limiters. The numerical examples in Section 5 show that it does not suffice to apply a standard limiter directly. Our main idea is to define an approximate solution $\bu_{h,l}$ from $\bu_h$ on the macro-element partition of the subdomains,  and apply the appropriate limiters on $\bu_{h,l}$.

We will in next sections show that applying a limiter to the reconstructed solution on macro elements will ensure the bound preserving properties.  
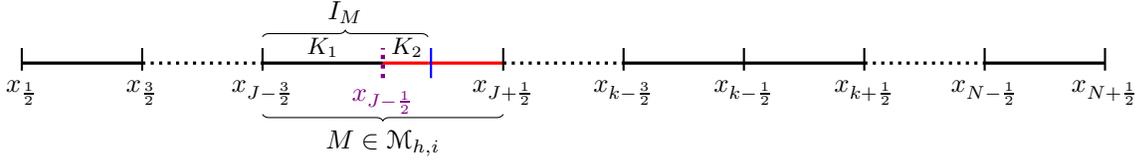
\begin{figure}[!htp]
\centering
\begin{tikzpicture}[xscale=4]
\draw[-][draw=black, very thick] (0,0) -- (.4,0);
\draw[dotted][draw=black, very thick] (0.4,0) -- (0.8,0);
\draw[-][draw=black, very thick] (0.8,0) -- (1.2,0);
\draw[-][draw=red, very thick] (1.2,0) -- (1.35,0);
\draw[-][draw=red, very thick] (1.35,0) -- (1.6,0);
\draw[dotted][draw=black, very thick] (1.6,0) -- (2.0,0);
\draw[-][draw=black, very thick] (2.0,0) -- (2.4,0);
\draw[-][draw=black, very thick] (2.4,0) -- (2.8,0);
\draw[dotted][draw=black, very thick] (2.8,0) -- (3.2,0);
\draw[-][draw=black, very thick] (3.2,0) -- (3.6,0);
\draw [thick] (0,-.1) node[below]{$x_{\frac{1}{2}}$} -- (0,0.2);
\draw [thick] (.4,-.1) node[below]{$x_{\frac{3}{2}}$} -- (0.4,0.2);
\draw [thick] (0.8,-.1) node[below]{$x_{J-\frac{3}{2}}$} -- (0.8,0.2);
%\draw [dotted][draw=violet, ultra thick] (1.2,-.1) node[below]{\small{$x_{J-\frac{1}{2}}\in\mcF_{h,i}$}} -- (1.2,0.2);
\draw [dotted][draw=violet, ultra thick] (1.2,-.2) node[violet][below]{$x_{J-\frac{1}{2}}$} -- (1.2,0.2);
\draw [blue][thick] (1.36,-0.2) -- (1.36,0.2);
\draw[decorate, decoration={brace},yshift=2.3ex]  (0.8,0) -- node[above=0.4ex] {$I_M$}  (1.35,0);
\draw[decorate, decoration={brace},yshift=-4.5ex] (1.6,0)  -- node[below=0.55ex] {$\Me\in\mcM_{h,i}$}  (0.8,0);
\draw [thick] (1.0,-0.05) node[above]{\small{$K_1$}};
\draw [thick] (1.28,-0.05) node[above]{\small{$K_2$}};
\draw [thick] (1.6,-.1) node[below]{$x_{J+\frac{1}{2}}$} -- (1.6,0.2);
\draw [thick] (2.0,-.1) node[below]{$x_{k-\frac{3}{2}}$} -- (2.0,0.2);
\draw [thick] (2.4,-.1) node[below]{$x_{k-\frac{1}{2}}$} -- (2.4,0.2);
%\draw [blue][thick] (2.45,-0.3) node[below]{$x'_{p}$} -- (2.45,0.2);
%\draw[decorate, decoration={brace},yshift=1.5ex]  (2.0,0) -- node[above=0.4ex] {$I'_M$}  (2.45,0);
\draw [thick] (2.8,-.1) node[below]{$x_{k+\frac{1}{2}}$} -- (2.8,0.2);
\draw [thick] (3.2,-.1) node[below]{$x_{N-\frac{1}{2}}$} -- (3.2,0.2);
\draw [thick] (3.6,-.1) node[below]{$x_{N+\frac{1}{2}}$} -- (3.6,0.2);
\end{tikzpicture}
\caption{Illustration of intervals $K_j\in \mathcal{K}_{h,i}$, which constitute part of the macro-element M. }\label{fig:macro-element}
\end{figure}

For each macro-element $M \in \mcM_{h,i}$, see Figure \ref{fig:macro-element},  let
\begin{align}\label{eq:def:IM}
I_M=M \cap \Omega_i= \cup_{j\in N_M} K_j \text{ with } K_j \in \Khi, 
\end{align}
where $N_M$ is an index set for those elements in $\mcK_{h,i}$ that constitute $I_M$, the part of the macro-element M that is in $\Omega_i$. Furthermore, we have $|I_M|\geq \delta h$ with the definition of $\delta$ in \eqref{eq:largeel}. For the macro-element $M$ illustrated in Figure \ref{fig:macro-element}, $N_M=\{1,2\}$.  We consider the scalar function on the macro-element here and note that we do a reconstruction for each component of $\bu_h$ if $\bu_h$ is a vector-valued function.  
Let $u_h(\cdot,t)=(u_{h,1},\cdots,u_{h,N_\Omega})$ with $u_{h,i}\in V^p_{h,i}$ be an approximation of $u(\cdot,t)$. 
and let $u_i=u|_{\Omega_i}$.  

For each macro-element $M\in\mcM_{h,i}$ and each $K_j$ with $j\in N_M$ define  
\begin{align}\label{eq:def:wh}
 w_{h,j}=u_{h,i}|_{K_j}, \quad w^e_{h,j} = E_M w_{h,j}, 
\end{align}
 where $E_M$ is the polynomial extension operator that extends $w_{h,j}$ to the entire macro-element $M$, such that  $E_M w_{h,j}|_{K_j}=w_{h,j}$. 
 At a point $x$ outside $K_j$, $w^e_{h,j}(x)$ is simply the evaluation of the polynomial at $x$. 

We propose the following reconstruction  on each $M$
\begin{align}\label{eq:def:uhm}
 u_{h,i}^{M}=& \sum_{j\in N_M} \omega_j w_{h,j}^e+c_0,\quad \text{for each } M\in \mathcal {M}_{h,i} 
\end{align}
with
\begin{align}
  \omega_j=\frac{|K_j|}{|I_M|}, \ 
  c_0=\frac{\sum_{j\in N_M} \int_{K_j}w_{h,j}dx}{|I_M|}-\frac{\sum_{j\in N_M}\omega_j \int_{I_M} w_{h,j}^e dx}{|I_M|}.
 \end{align}
Note that each $w_{h,j}$ is extended to the entire macro-element and evaluated in $I_M$ and other choices of weights are possible.  The constant $c_0$ is chosen so that $u_{h,i}^{{M}}$ has the same mean value on $I_M$ as $u_{h,i}$.  Note that $\sum_{j\in N_M} \omega_j=1$, and when $M$ consists of only one element then $|K_j|=|I_M|$ and $u_{h,i}^{{M}}=u_{h,i}|_{M}$.

This reconstruction yields a new approximate solution $u^{\mathcal{M}}_{h,i}$ on the macro-element partition, which is  in the subspace of $V_{h,i}^p$ and   has the following properties
\begin{itemize}
\item conservation:  we have 
\begin{align}\label{eq:def:conservation}
\int_{I_M} u_{h,i}^{\mathcal{M}}dx= \sum_{j\in N_M} \int_{K_j}w_{h,j}dx=\int_{I_M} u_{h,i}dx,
\end{align}
which shows that the reconstruction maintains conservation.
\item accuracy:  let $u_i=u|_{\Omega_i}\in H^{p+1}(\Omega_i) $ define the smooth solution  of \eqref{eq:model:1d} and assume $u_{h,i}$ satisfies
\begin{equation}
  ||u_{h,i}-\pi_h u_i||^2_{\Omega_i}+J_1(u_{h,i}-\pi_h u_i,u_{h,i}-\pi_h u_i)\lesssim Ch^{2p}.
  \label{eq:estphih}
\end{equation}
Here $\pi_h$ denotes the $L^2$ projection onto  $V_{h,i}^p$ and we have $\pi_h u_i=\pi_h E u_i \in V_{h,i}^p$, where $E:H^{p+1}(\Omega_i) \rightarrow H^{p+1}(\mathbb{R})$ is a continuous Sobolev extension operator \cite{stein2016singular}.
The estimate~\eqref{eq:estphih} is the same as in \cite[appendix B]{fu2022high}.
  
To consider the error estimate of the reconstruction, we write
\begin{align}
 ||u_i-u_{h,i}^{\mathcal{M}}||_{\Omega_i}=&\sum_{M\in\mcM_{h,i}} ||\sum_{j\in N_M} \omega_j(u_i-w_{h,j}^e)-c_0||_{I_M}
 = \sum_{M\in\mcM_{h,i}}||\sum_{j\in N_M} \omega_j(u_i-w_{h,j}^e)\notag\\
 & +\sum_{j\in N_M} \frac{1}{|I_M|}\int_{K_j}(u_i-w_{h,j})dx
 +\sum_{j\in N_M}\frac{\omega_j}{|I_M|}\int_{I_M}(w_{h,j}^e-u_i)dx||_{I_M}\notag\\
\leq &2\sum_{M\in\mcM_{h,i}} \sum_{j\in N_M} \omega_j ||u_i-w_{h,j}^e||_{I_M} +\sum_{M\in\mcM_{h,i}} \sum_{j\in N_M}\sqrt{\omega_j}||u_i-w_{h,j}||_{K_j}.
\label{eq:error:uhm}
 \end{align}
We have used $\sum_{j\in N_M}\omega_j=1$, added and subtracted $\frac{1}{|I_M|}\int_{I_M}u_idx=\frac{1}{|I_M|}\sum_{j\in N_M}\int_{K_j}u_idx=\frac{1}{|I_M|}\sum_{j\in N_M}\omega_j\int_{I_M}u_idx$ in the second equality in \eqref{eq:error:uhm}, and finally we appled Minkowski's inequality and H\"older's inequality. 

Define $\intp:  u_i \rightarrow \intp E u_i \in V_{h,i}^p$ to be the $L^2$ projection operator on the macro-element partition such that $\intp E u_i|_{\Me} \in  \mathbb{P}^p(\Me)$ for $\Me \in \mcM_{h,i}$. For macro-element $M$ we assume $K_l\in \Khi$ is the element with a large $\Omega_i$-intersection. Recall the definition of $w^e_{h,l}$ from \eqref{eq:def:wh}.
%Thus,   $\intp E u_i$ is continuous polynomial on each macro-element.
We rewrite the first term in \eqref{eq:error:uhm} by adding and subtracting $w^e_{h,l}$ and $\intp u_i$ to get 
\begin{align}
\sum_{M\in\mcM_{h,i}}& \sum_{j\in N_M} \omega_j ||u_i-w_{h,j}^e||_{I_M} 
\leq \sum_{M\in\mcM_{h,i}} \sum_{j\in N_M} \omega_j\left(||u_i-w^e_{h,l}||_{I_M}+||w^e_{h,l}-w_{h,j}^e||_{I_M} \right)\notag\\
&\leq \sum_{M\in\mcM_{h,i}} \sum_{j\in N_M} \omega_j\left(||u_i-\intp u_i||_{I_M}+||\intp u_i-w^e_{h,l}||_{I_M}+||w^e_{h,l}-w_{h,j}^e||_{I_M} \right)\notag\\
&\lesssim \sum_{M\in\mcM_{h,i}} \sum_{j\in N_M} \left(||u_i-\intp u_i||_{I_M}+
||\intp u_i-w_{h,l}||_{K_l}+ ||w_{h,l}-w_{h,j}^e||_{K_l}\right). 
\label{eq:error:uhm2}
\end{align}
Here  $a\lesssim b$ is the notation for  $a\leq Cb$ for some generic constant $C$, which is independent of $h$. To arrive at the third inequality in \eqref{eq:error:uhm2}  we used Lemma 2.3 in \cite{chen2021adaptive} to estimate the second and the third term (the constant in the estimate depends on $|I_M|/|{K_l}|$ and $p$). Using Taylor expansion around the edge $x_e={K_l}\cap K_j$, we have 
 \begin{align}
\sum_{M\in\mcM_{h,i}} \sum_{j\in N_M}&|| w_{h,l}(x)-w_{h,j}^e(x)||^2_{K_l}=\sum_{M\in\mcM_{h,i}} \sum_{j\in N_M}||\sum^p_{r=0}[D^r u_{h,i}]_e\cdot(x-x_e)^r/r!||^2_{K_l}\notag\\
\lesssim&  J_1(u_{h,i}, u_{h,i})=J_1(u_{h,i}-E u_i, u_{h,i}-Eu_i)\notag\\
\lesssim& J_1(u_{h,i}-\pi_h u_i,u_{h,i}- \pi_h u_i)+J_1(\pi_h u_i-Eu_{i}, \pi_h u_i-Eu_{i}) .
\label{eq:error:uhm3}
 \end{align}
 Here we utilized that $J_1(E u_i,E u_i)=0$ and added and subtracted $J_1(\pi_h u_i,\pi_h u_i)$. 
 For the first term in \eqref{eq:error:uhm} we combine the estimate \eqref{eq:error:uhm2} and \eqref{eq:error:uhm3}, 
for the second term of \eqref{eq:error:uhm} we use that $w_{h,j}|_{K_j}=u_{h,i}$ to get
\begin{align}
 ||u_i-u_{h,i}^{\mathcal{M}}||_{\Omega_i}^2
 \lesssim& ||u_i-u_{h,i}||_{\Omega_i}^2+ ||u_i-\intp u_i||_{\Omega_i}^2+ ||u_i- \pi_hu_{i}||_{\Omega_i}^2+ J_1(E_iu_{i} -\pi_h u_{i} ,Eu_{i}-\pi_hu_{i})\notag\\
 &+  ||\pi_h u_i-u_{h,i}||_{\Omega_i}^2 +J_1(u_{h,i}-\pi_h u_i,u_{h,i}-\pi_h u_i)
 \leq  Ch^{2p}. 
 \label{eq:error:uhm4}
\end{align}
In the last inequality we used \eqref{eq:estphih} and that optimal approximation properties hold for the projection operators $\pi_h$ and $\intp$, see \cite[(B.13)-(B.15)]{fu2022high}. 
Thus, we have shown that for smooth problems the reconstructed solution $u_{h,i}^{\mathcal{M}}$, defined in \eqref{eq:def:uhm}, has the same order of convergence as $u_{h,i}$.

\begin{rem}
In practice the approximate solution $\bu_{h,l}$  should be constructed only when limiting is necessary. Numerical experiments, where the reconstruction is only done for macro-elements that violate the bound property, show that this works well. See Figure \ref{figure:burgesshock2:onlymacro} for a scalar example and Figure \ref{exe:sedovblast:1d:macroonlybad} for a system. We remark that it is also possible to use the extension technique in the discretization of the PDE instead of adding stabilization terms, see for example \cite{HUANG2017439,doi:10.1137/20M137937X,burman2022cutfem}. 
\end{rem}

\end{itemize}

\section{Maximum principle preserving DG methods for scalar equations}
In this section, we propose a fully discretized scheme for the scalar hyperbolic conservation laws and study the maximum principle property of the proposed scheme.  It suffices to consider the forward Euler method.   
In the numerical experiments, we use high order time discretization methods that  are linear combinations of the forward Euler method.  In each time step we  include a  reconstruction on the macro-element partition of the mesh. We will show that our proposed scheme satisfies the maximum principle. The scheme can be summerized as follows: Given the approximate solution $u^n_h=(u_{h,1}^n, \cdots, u_{h,N_\Omega}^n) $ with $u_{h,i}^n\in V_{h,i}^p$, and satisfying  $u_{h,i}^n|_{\Omega_i}\in [\mm,\MM]$ at time $t_{n}$, we compute $u^{n+1}_h$ by the following steps:
\begin{enumerate}
\item 
Find $\tilde{u}_h^{n+1}=(\tilde u_{h,1}^{n+1}, \tilde{u}_{h,2}^{n+1}, \cdots, \tilde{u}_{h,N_\Omega}^{n+1})\in V_h^p$ such that for $\forall v_h\in\V$
\begin{align}\label{scheme:cutDG2fully}
  \sum_{i=1}^{N_\Omega}\sum_{K\in\Khi}\left(\frac{\tilde u_{h,i}^{n+1}-u_h^n}{\Delta t},v_h\right)_{K}+\rM J_1\left(\frac{\tilde u_{h}^{n+1}-u_h^n}{\Delta t},v_h \right)+a(u_h^n,v_h)
  +\rA J_0(u_h^n,v_h)
  =0, 
\end{align}
with the bilinear forms $a(\cdot,\cdot)$ and $J_s(\cdot,\cdot)$ ($s=0,1$) defined as in \eqref{def:auv} and \eqref{stable1}.  
\item  From $\tilde u^{n+1}_{h,i}$, we compute the reconstructed solution following \eqref{eq:def:uhm}. For each $M \in \mcM_{h,i},i=1,\cdots, N_\Omega$,  we compute $u^{\mathcal{M},n+1}_{h,i}$ from $\tilde u^{n+1}_{h,i}$, which is 
\begin{align}\label{eq:scalar:reconstruction}
  {u}^{\mathcal{M},n+1}_{h,i}|_{M}=  \frac{{\sum_{j\in N_M} }{|K_j|} {\tilde w}^{e,n+1}_{h,j}}{|I_M|}+c_0.
\end{align}
Here,  $\tilde w_{h,j}^{e,n+1}$  is defined as in \eqref{eq:def:wh} but $u_{h,i}$ is replaced by $\tilde u_{h,i}^{n+1}$. If needed, bound preserving limiting, which will be described in section 3.3, is applied to the reconstructed solution before the next step.
 \item Finally, for each macro-element $M \in \mcM_{h,i},i=1,\cdots, N_\Omega$ set
\begin{equation} \label{eq:schemestep3}
  u^{n+1}_{h,i}|_{M}=u^{\mathcal{M},n+1}_{h,i}|_{M}. 
\end{equation}
\end{enumerate}
At the initial time $t_0$,   given $u_0$, replace (3.1) by \eqref{eq:stabilize:L2} and then apply step 2 and 3 to compute $u_{h,i}^{M,0}$ and $u_{h,i}^0$.  
In practice, the approximate solution $\bu_{h,M}$ should be constructed only in those $M\in\mathcal{M}_{h,i}$ when limiting is needed (see Figure \ref{figure:burgesshock2:onlymacro} for Burgers' equation and Figure \ref{exe:sedovblast:1d:macroonlybad} for the Euler equations).  Note that the stabilization terms $J_i$ act only on the interior edges in macro-elements. For macro-elements where the reconstruction is done, the reconstructed solution $u_h^n$ has no jump across the interior edges and thus on such elements the stabilization terms $J_i(u_h^n, v_h)$ vanish.  In all other macro-elements the stabilization $J_0$ is needed.

In the higher order versions of our scheme the reconstructed solution may not  satisfy the maximum principle automatically at the next time level  i.e. $u^{\mathcal{M},n+1}_{h,i}\notin[\mm,\MM], i=1,\cdots,N_\Omega$, and this makes limiting necessary. The details are found in Algorithm 1.

\subsection{Piecewise constant approximation}
Let $u^n_h\in\V, \,p=0$ denote the piecewise constant function that approximates the solution at time $t_n$.   Note that in this case, we have from the conservation property \eqref{eq:def:conservation} that $c_0=0$ and thus
\begin{align}\label{eq:scalar:reconstruction:p0}
  {u}^{M,n}_{h,i}=  \frac{\sum_{j \in N_M} |K_j| {\tilde w}^{n}_{h,j} }{|I_M|},
\end{align}
where ${\tilde w}^{n}_{h,j}$ denotes the mean value on the element $K_j$. 
Note that the solution ${u}^{\mathcal{M},n}_{h,i}$ has no discontinuity at interior edges of macro-elements (for example across the edge between $K_1$ and $K_2$ in Figure \ref{fig:macro-element}), 
and with  $p=0$ there is no difference between function values and mean values.

\begin{thm}\label{thm:scalar:1ord}
  Consider the  Cut-DG scheme described by \eqref{scheme:cutDG2fully}-\eqref{eq:schemestep3} with $p=0$.  If the time step $\Delta t$ satisfies
\begin{align}\label{eq:scalar:p0:time}
\frac{\Delta t}{\delta h}\lambda\leq 1,
\end{align}
then  $u_{h,i}^n|_{\Omega_i} \in [\mm,\MM],i=1,\cdots, N_\Omega$ implies % where $m=\min_{x\in\Omega}{u_0(x)},M=\max_{x\in\Omega}{u_0(x)}$, 
 $u^{n+1}_{h,i}|_{\Omega_i}\in[\mm,\MM]$ for $i=1,\cdots, N_\Omega$.
  Here $h$ is the element size in the background mesh, and $\delta$ is a given constant, which in \eqref{eq:largeel} defines when an intersection between an element and $\Omega_i$ is classified as large.
     $\lambda$ is the maximum absolute value of the Jacobian  $\partial f(u_h^n)/\partial u_h^n$ in the domain $\Omega$.
\end{thm}
\begin{proof}
 By the  reconstruction \eqref{eq:schemestep3}, $u_h^n$ is constant in each macro element $M\in \mcM_{h,i}, i=1,\cdots,N_\Omega$,  and we denote the value by $u_{h,i}^{M,n}$. Let $u_{h,l}^n$ and $u_{h,r}^n$ denote  the values adjacent to  the left and right boundary of  $I_{M}$, respectively, see  $I_M$ in \eqref{eq:def:IM}.   We assume $u_h^n\in [m,M]$. Thus, we have $u_{h,i}^{M,n},u_{h,l}^n,u_{h,r}^n\in[\mm,\MM]$.  Further, let $\widehat{f}_{l}$ and $\widehat{f}_{r}$ denote the numerical fluxes on the left and on the right boundaries of the macro-element $I_M$, respectively. We use the global Lax-Friedrichs  flux, defined by \eqref{eq:def:lf}, which is 
\begin{align}\label{eq:def:lf:left}
\widehat{f}_l&= \frac{1}{2}\left(f(u_{h,l}^n(x_l))+f(u_{h,i}^{M,n}(x_l))\right)
 -\frac{\lambda}{2}\left(u_{h,i}^{M,n}(x_l)-u_{h,l}^n(x_l)\right),\\
 \label{eq:def:lf:right}
 \widehat{f}_r&= \frac{1}{2}\left(f(u_{h,i}^{M,n}(x_r))+f(u^n_{h,r}(x_r))\right)
 -\frac{\lambda}{2}\left(u^n_{h,r}(x_r)-u_{h,i}^{M,n}(x_r)\right).
\end{align}
Let the test function $v_h$ in \eqref{scheme:cutDG2fully} to be $v_h=1$ on each element of $M$ and $v_h=0$ otherwise. By noting that the stabilization terms $J_s$ act only on interior edges in macro-elements and that $u_{h,i}^{n}$  has no jump across such edges, we get   
\begin{align}\nonumber
\sum_{j\in N_M}|K_j|\tilde w^{n+1}_{h,j}=\sum_{j \in N_M}|K_j| u_{h,i}^{n}|_{K_j}-{\Delta t}(\widehat{f}_r-\widehat{f}_l).
\end{align}
Taking (\ref{eq:scalar:reconstruction:p0}) into account at $t=t_{n+1}$, the above equation can be written as
\begin{align}\label{eq:def:scheme:umN}
{u}_{h,i}^{M,n+1}={u}_{h,i}^{M,n}-\frac{\Delta t}{|I_{M}|}\left(\widehat{f}_r-\widehat{f}_{l}\right).
\end{align}
With the definition of  global Lax-Friedrichs flux   \eqref{eq:def:lf:left}-\eqref{eq:def:lf:right} we can rewrite (\ref{eq:def:scheme:umN}) as
\begin{align}
{u}_{h,i}^{M,n+1}&={u}_{h,i}^{M,n}-\frac{\lambda\Delta t}{|I_{M}|}u_{h,i}^{M,n}+\frac{\Delta t}{2|I_{M}|}(\lambda u^n_{h,r}-f(u_{h,r}^n))+\frac{\Delta t}{2|I_{M}|}(\lambda u^n_{h,l}+f(u_{h,l}^n)).
\end{align}
With  $H(a,b,c)$  defined as 
\begin{align}\label{eq:scalar:def:H}
H(a,b,c;\lambda):=\left(1-\frac{\lambda\Delta t}{|I_{M}|}\right)b+\frac{\Delta t}{2|I_{M}|}(\lambda c-f(c))+\frac{\Delta t}{2|I_{M}|}(\lambda a+f(a)),
\end{align}
we have
\begin{align}\label{eq:FirstOrder}
{u}_{h,i}^{M,n+1}=H(u_{h,l}^n,u_{h,i}^{M,n},u_{h,r}^n;\lambda).
\end{align}
Since $|I_M|\geq \delta h$,  it follows by  \eqref {eq:scalar:p0:time} that 
\begin{align}%\nonumber
0&<\frac{\Delta t}{|I_{M}|}\lambda\leq 1,\\
\frac{\partial H(a,b,c;\lambda)}{\partial a}&=\frac{\Delta t}{2|I_{M}|}(\lambda+f'(a))\geq 0,\\
\frac{\partial H(a,b,c;\lambda)}{\partial b}&=1-\frac{\Delta t}{|I_{M}|}\lambda\geq 0,\\
\frac{\partial H(a,b,c;\lambda)}{\partial c}&=\frac{\Delta t}{2|I_{M}|}(\lambda-f'(c))\geq 0.
\end{align}
Further, 
\begin{align}
H(z,z,z;\lambda)=\left(1-\frac{\lambda\Delta t}{|I_{M}|}\right)z+\frac{\Delta t}{2|I_{M}|}(\lambda z-f(z))+\frac{\Delta t}{2|I_{M}|}(\lambda z+f(z))=z.
\end{align}
Thus, it follows that $u_{h,i}^{M,n+1}\in [\mm,\MM]$ and hence $u_{h}^{n+1}$ satisfies the maximum principle under the time step restriction \eqref{eq:scalar:p0:time}. 
\end{proof}

\subsection{Mean values of high order approximations}
We now consider  higher order schemes in space and time, where  the time discretization can be seen as a sequence of Euler steps. We therefore consider \eqref{scheme:cutDG2fully} with $p\ge 1$ together with the reconstruction \eqref{eq:scalar:reconstruction}.  We will show that for sufficiently small time steps the mean values on macro-elements after an Euler step satisfy the  same  maximum principle as in Theorem 3.1 above. We will follow the ideas used to prove the corresponding result for the standard DG method \cite{zhang2010maximum}. 
In the proof we will use the  $q$-point Gauss-Lobatto quadrature rule on each macro-element, which is exact for polynomials of degree $2q-3$.  Let $q$  be the smallest integer satisfying $2 q-3 \geqslant p$.  
For each $M\in \mcM_{h,i}$, we let
$\hat x_{\mu}^M$ denote the $\mu$-th quadrature point, with   $\hat{x}^{M}_{1}=x_{M,l},\hat{x}^{M}_{q}=x_{M,r}$. Here $x_{M,l},x_{M,r}$ are the end points of $I_{M}$,  the part of the macro-element, which is in the subdomain $\Omega_i$.  
 The corresponding weights, normalized to a unit interval, are denoted by $\hat w_\mu, \mu=1,\cdots, q$.  We note that the weights in the chosen quadrature rules are symmetric,  $\hat w_1=\hat w_q$.  
 We also note that after the reconstruction our approximation is a polynomial of degree $p$ on the entire macro-element, which can be integrated exactly by the quadrature rule. The following theorem shows that one Euler step preserves the maximum principle for mean values defined by $\bar{u}^{n}_{h}=(\bar{u}_{h,1}^{n}, \cdots, \bar{u}_{h,N_\Omega}^{n})$ with 
 \begin{align}\label{eq:def:mean:scalar}
 \bar{u}_{h,i}^n|_{M}=\frac{1}{|I_M|}\int_{I_M}u_{h,i}^{M,n}dx=\frac{1}{|I_M|}\int_{I_M}u^n_{h,i}dx, \quad \bar u^{M,n}_{h,i}= \bar{u}_{h,i}^n|_{M}, 
 \end{align}
 for each macro-element $M\in \mathcal{M}_{h,i}$. 
\begin{thm}
Consider the  Cut-DG scheme described by \eqref{scheme:cutDG2fully}-\eqref{eq:schemestep3} with $p\geq 1$.  If the time step $\Delta t$ satisfies
\begin{align}\label{eq:scalar:pp:time}
\frac{\Delta t}{\delta h}\lambda\leq \hat{w}_1,
\end{align}
then  $u_{h,i}^n|_{\Omega_i} \in [\mm,\MM], i=1,\cdots,N_\Omega$ implies that
$\bar{u}_{h,i}^{\mathcal{M},n+1}|_{\Omega_i}\in[\mm,\MM],i=1,\cdots,N_\Omega$, that is, mean values on all macro-elements satisfy the maximum principle.  
Here $h$ is the element size in the background mesh,  and $\delta$ is a given constant, which in \eqref{eq:largeel} defines when an intersection between an element and $\Omega_i$ is classified as large.  
$\lambda$ is the maximum absolute value of the Jacobian $\partial f(u_h^n)/\partial u_h^n$ in the domain $\Omega$. $\hat{w}_1$ is the  normalized weight corresponding to the first interior point for a q-point Gauss-Lobatto rules with $2q-3\geq p$.
\end{thm}
\begin{proof}
As in Section 3.1, we look at  a macro-element $M$ as shown in Figure \ref{fig:macro-element}, with $I_M=K_1\cup K_2$.  
Let $u_{h,l}^n$ and $u_{h,r}^n$ denote  the values adjacent to  the left and right end of  $I_{M}$, respectively, and let $\widehat{f}_r$ and $\widehat{f}_l$ be the corresponding numerical fluxes defined by  \eqref{eq:def:lf}, which is expressed in \eqref{eq:def:lf:left}-\eqref{eq:def:lf:right} on the macro-element. By reconstruction,  $u_{h}^n$ is a single polynomial of degree $p$ in $M$, and can be exactly integrated over $I_{M}$ by the $q-$point  Gauss-Lobatto quadrature rule. 
Let the test function $v_h$ in \eqref{scheme:cutDG2fully} be $1$ on each element of  $M$ and $0$ otherwise, to obtain
\begin{align}\label{eq:scheme:scalar:high}
\sum_{j \in N_M}\int_{K_j}\tilde u^{n+1}_{h,j}dx=\sum_{j \in N_M}\int_{K_j}u_{h,i}^{n}dx-(\widehat{f}_r-\widehat{f}_{l}).
\end{align}
By conservation on macro-elements and the definition  \eqref{eq:scalar:reconstruction}, \eqref{eq:def:mean:scalar}, we have
\begin{align}\label{eq:define:scheme:average}
\overline{u}_{h,i}^{M,n+1}=\overline{u}_{h,i}^{M,n}-\frac{\Delta t}{|I_{M}|}\left(\widehat{f}_r-\widehat{f}_{l}\right).
\end{align}
Applying a $q$-point Gauss-Lobatto quadrature rules gives 
\begin{align}\label{eq:define:average:decompose}
\overline{u}_{h,i}^{M,n}=&\frac{1}{|I_{M}|}\int_{I_{M}}u_{h,i}^{M,n}dx=\sum_{\mu=1}^q \hat{w}_\mu \hat{u}_{\mu}^{M,n}.
\end{align} 
Here $ \hat{u}^{M,n}_{\mu}=u^{M,n}_{h,i}(\hat x_\mu^M)$ denotes the value of the local polynomial at the quadrature point $\hat x_\mu^M$.  
We now follow the idea of the proof of the bound preserving property of the standard DG method for hyperbolic conservation laws, see Theorem 2.2 in  \cite{zhang2010maximum}. But we consider  the mean values $\bar u_{h,i}^{M,n+1}$ of the Cut-DG method on each macro-element. Introduce $\hat F$ as the Lax-Friedrichs flux based on $\hat u^n_{M,1}$ and $\hat u^n_{M,q}$, that is
 \begin{align}
\widehat F=\frac{1}{2}(f(\hat u^{M,n}_{1})+f(\hat u^{M,n}_{q}))-\frac{\lambda}{2}(\hat u^{M,n}_{q}-\hat u^{M,n}_{1}), 
\end{align}
Using the quadrature \eqref{eq:define:average:decompose}, $\hat \omega_1=\hat \omega_q$  and the definition of  $H(a,b,c;\lambda)$  in  \eqref{eq:scalar:def:H},  the equation   \eqref{eq:define:scheme:average} can be rewritten as 
\begin{align*}
\overline{u}_{h,i}^{M,n+1}
=&\sum_{\mu=1}^q \hat{w}_\mu \hat{u}^{M,n}_{\mu}-\frac{\Delta t}{|I_{M}|}\left(\widehat{f}_r
-\widehat{f}_l\right)=\\
=&\sum_{\mu=2}^{q-1} \hat{w}_\mu \hat{u}^{M,n}_{\mu}
+\hat{w}_1\left(\hat{u}^{M,n}_{1}-\frac{\Delta t}{\hat{w}_1|I_{M}|}\left(\hat{F}
-\hat{f}_l)\right)\right)\notag+\hat{w}_q\left(\hat{u}^{M,n}_{q}-\frac{\Delta t}{\hat{w}^n_q|I_{M}|} (\hat{f}_r
-\hat{F})\right)=\\
=&\sum_{\mu=2}^{q-1} \hat{w}_\mu \hat{u}^{M,n}_{\mu}+\hat{w}_1H(\hat{u}^{M,n}_{0},\hat{u}^{M,n}_{1},\hat{u}^{M,n}_{q};\lambda/\hat{w}_1)
+\hat{w}_qH(\hat{u}^{M,n}_{1},\hat{u}^{M,n}_{q},\hat{u}^{M,n}_{q+1};\lambda/\hat{w}_q).
\end{align*}
Here  $\hat{u}^{M,n}_{0}=u_{h,l}^{n}(x_{M,l})$ and $\hat{u}^{M,n}_{q+1}=u_{h,r}^{n}(x_{M,r})$ are used.  
We can see that the above expression is a  linear combination of  two first order schemes \eqref{eq:FirstOrder} and $\sum_{\mu=2}^{q-1}\hat{w}_\mu  \hat{u}^{M,n}_{\mu}$. As in Theorem \ref{thm:scalar:1ord}, we can conclude that $\overline{u}_{h,i}^{M,n+1}\in[\mm,\MM] $ under the condition $\frac{\Delta t}{\delta h }\lambda\leq \hat{w}_1$.  
\end{proof}
We note that the time restrictions in \eqref{eq:scalar:p0:time} and \eqref{eq:scalar:pp:time} are of  the same order as for the corresponding standard DG methods on the background grid. When we set $\delta=1$, the time-step restrictions are same as in the standard un-cut case, but all cut elements are stabilized.

\subsection{A bound preserving limiter for scalar problems}
In our proposed scheme described in \eqref{scheme:cutDG2fully}-\eqref{eq:schemestep3} and in Theorems 3.1-3.2, we require $u_{h,i}^{n}|_{\Omega_i}\in[\mm,\MM],i=1,\cdots,\Omega_N$, but the high order reconstructed solution at next time level may not satisfy this requirement. In order to guarantee that the solution remains within the bounds at later times, we use a bound preserving limiter developed by Zhang and Shu \cite{zhang2010maximum} for high order approximations.  In the standard DG method,  the solution $u_h(x)$ is modified locally in each element $I$ such that mean values are un-changed, and $u_{h}(x)\in[\mm,\MM]$ for $x\in S_I$, where $S_I$ is set of the  Gauss-Lobatto quadrature points on $I$. 
%  Note that in standard DG method, $u_h^n(x)$ is modified locally in each element $K$. 
 In the unfitted case,  a cut element $I$ may be such that $I\cap\Omega_i$ is very small but $I_M=M\cap\Omega_i$ is always large. Therefore we propose to apply the limiter to the reconstructed solution in each  macro-element $M$. 
 For each macro-element $M$ we define the limited solution $u_{h,i}^{M,(l)}(x)$ by
\begin{equation}\label{eq:limit:uj}
u_{h,i}^{M,(l)}(x)=\theta\left(u_{h,i}^{M}(x)-\bar u_{h,i}^{M}\right)+\bar{u}_{h,i}^{M}, \quad \theta=\min \left\{\left|\frac{\MM-\bar{u}_{h,i}^{M}}{\MM_{M}-\bar{u}_{h,i}^{M}}\right|,\left|\frac{\mm-\bar{u}_{h,i}^{M}}{\mm_{M}-\bar{u}_{h,i}^{M}}\right|, 1\right\}.
\end{equation} 
Here we assume $u_{h,i}^M$ is a reconstructed solution on $M$ and 
\begin{align}\label{eq:limit:mj}
\MM_{M}=\max _{x \in {I_M}} u_{h,i}^{M}(x), \quad \mm_{M}=\min _{x \in I_M } u_{h,i}^{M}(x).
\end{align}
Then, we let $u_{h,i}^{\mathcal{M}}|_{M}=u_{h,i}^{M,(l)}(x)$,  which  satisfies the maximum principle on $\Omega_i$.  As for the standard DG method,  we can use $ u^{M,(l)}_{h,i}(x)$ instead of $ {u}_{h,i}^{M}(x)$ without destroying the accuracy.  
We note that  the exact minimum value and maximum value of the polynomials  is needed on the macro-element. 

By combining the sequence of steps \eqref{scheme:cutDG2fully}-\eqref{eq:schemestep3} with \eqref{eq:limit:uj}, we can get a fully discrete, bound preserving scheme for scalar problems for any polynomial order. For clarity we have formulated the version based on a third order Runge-Kutta discretization \eqref{eq:TVDRK3} in time in Algorithm 1. 
\begin{algorithm}[!h]
\caption{Fully discrete bound preserving Cut-DG method for scalar problem}\label{algorithm}
\KwData{Given initial condition $u_0(x)$ at $t=0$}
\KwResult{Bound preserving approximation at final time $T$}
Initialize $\tilde u_h^0$ by \eqref{eq:stabilize:L2} and $n=0$\;
$u_h^0$=Reconstruction and limiting$(\tilde {u}_h^{0})$\;
\While{$t\leq T$}{
Compute suitable time-step $\Delta t$ based on Theorem 3.1/3.2\;
\tcc {Runge-Kutta time discretization}
$\tilde {u}_h^{(1)}$ = Euler-Step (${u}_h^{n}$, $t_n$, $\Delta t$)\tcp*{first stage}
${u}_h^{(1)}$ = Reconstruction and limiting $(\tilde {u}_h^{(1)})$\;
 ${u}_h^{(2')}$ = Euler-Step(${u}_h^{(1)}$, $t_n+\Delta t$, $\Delta t$)\tcp* {second stage}
Set: $\tilde {u}_h^{(2)}=\frac{3}{4}{u}_h^{n}+\frac{1}{4}{u}_h^{(2')}$\;
${u}_h^{(2)}$ = Reconstruction and limiting $(\tilde {u}_h^{(2)})$\;
 $\tilde {u}_h^{(3')}$  =  Euler-Step(${u}_h^{(2)}$, $t_n+\Delta t/2$, $\Delta t$)\tcp* {third stage}
Set: $\tilde {u}_h^{(3)}=\frac{1}{3}{u}_h^{n}+\frac{2}{3}{u}_h^{(3')}$\;
${u}_h^{(3)}$ = Reconstruction and limiting $(\tilde  {u}_h^{(3)})$\;
Set: ${u}_h^{n+1}= {u}_h^{(3)}$\;
$n\leftarrow n+1$\;
$t\leftarrow t+\Delta t$\;
}
\SetKwFunction{proc}{Euler-step}
  \setcounter{AlgoLine}{0}
  \SetKwProg{eulerstep}{Function}{}{}
  \eulerstep{\proc{${u}_h$, $t$, $\Delta t$}}{
Using \eqref{scheme:cutDG2fully} with ${u}_h,t,\Delta t$  to get $\tilde{u}_h=(\tilde u_{h,1}, \tilde{u}_{h,2}, \cdots, \tilde{u}_{h,N_\Omega})\in V_h^p$ \;
 \KwRet $\tilde {u}_h=(\tilde u_{h,1}, \tilde{u}_{h,2}, \cdots, \tilde{u}_{h,N_\Omega})$\;
  }
  \SetKwFunction{procc}{Reconstruction and limiting}
  \setcounter{AlgoLine}{0}
  \SetKwProg{reconlimit}{Function}{}{}
  \reconlimit{\procc{$\tilde {u}_h$}}{
\For{$i=1$ \KwTo $N_\Omega$}{
Using \eqref{eq:scalar:reconstruction} with $\tilde {u}_h$  to reconstruct $u_{h,i}^{\mathcal{M}}, i=1,\cdots,N_\Omega$\;
\lIf{$u_{h,i}^{\mathcal{M}}|_{\Omega_i}\notin[\mm,\MM]$}{
modify $u_{h,i}^{\mathcal{M}}$ by  \eqref{eq:limit:uj}}
$u_{h,i}\leftarrow u_{h,i}^{\mathcal{M}}$\;
}
 \KwRet ${u}_h=(u_{h,1}, {u}_{h,2}, \cdots, {u}_{h,N_\Omega})\in V_h^p$\;
  }
\end{algorithm}

\newcommand{\bw}{\mathbf{u}}
\newcommand{\gm}{\boldsymbol{m}}
\newcommand{\n}{\boldsymbol{n}}
\newcommand{\I}{\mathbf{I}}
\newcommand{\x}{\boldsymbol{x}}
\newcommand{\bnu}{\boldsymbol{\nu}}
\newcommand{\bs}{\boldsymbol{s}}
\newcommand{\bg}{\boldsymbol{G}}
\newcommand{\bfu}{\boldsymbol{F}}
\newcommand{\hatg}{\widehat{\boldsymbol{G}}}
\newcommand{\dw}{(\omega_{j+\frac{1}{2}}-\omega_{j-\frac{1}{2}})}

\newcommand {\Q}{\mathbb{Q}}
\newcommand {\R}{\mathbb{R}}
\newcommand {\N}{\mathbb{N}}
\newcommand {\Z}{\mathbb{Z}}
\section{A Positivity preserving Cut-DG method for the Euler equations}
\label{sec:euler:1d}
In this section, we will show how the positivity property  for the Euler equations  can be built into our proposed Cut-DG method. The Euler equations in one space dimension are 
\begin{align}
\label{euler:eq}
& \bw_t+\partial_x\mathbf{f}{( \bw)}=0,\quad t\geq0,\; x \in \Omega\subset \mathbb{R},\\
\label{euler:eq:flux}
& \bw=\begin{pmatrix}
      \rho    \\
      m\\
      E
\end{pmatrix}
,\quad \mathbf{f}{(\bw)}=\begin{pmatrix}
      m   \\
      \rho u^2+p\\
      (E+p)u
\end{pmatrix}.
\end{align}
Here  $\rho$ is the density, $m$  the momentum,  and $E$  the total energy.  %The vector$\mathbf{f}(\bw)$ denotes the flux vector. 
These quantities are often called the conservative variables, and are related to other physical quantities such as the velocity $u$, the pressure $p$, and the internal energy $e$ through $E=\frac{1}{2}\rho u^2+\rho e,\, m =\rho u,$ and $p=(\gamma-1)\rho e$. Here $\gamma$ is the specific gas constant. Note that the pressure $p$  can be written as a function of the conservative variables,
\begin{equation}\label{eq:def:p}
p(\bw)=(\gamma-1)\left(E-\frac{1}{2}\frac{m^2}{\rho}\right).
\end{equation}
 Also important is the sound speed, given by $c=\sqrt{\gamma p/\rho}$.
 
 In physically relevant solutions, the density and pressure should be positive. Without positivity hyperbolicity is lost, and the mathematical problem is ill-posed. If  the solution is initially positive, it can be proven that  the solution of the compressible Euler equations cannot loose positivity of density and pressure, see  \cite{Glimm1965,dafermos2005hyperbolic,chen2017}. Therefore 
 we define the following set 
\begin{align}\label{euler:set}
G=\left\{\bw=(\rho,m,E)^\T
\middle|
\;\rho > 0,\; p(\bu)=(\gamma-1)\left(E-\frac{1}{2}\frac{m^2}{\rho}\right) >0
\right\},
\end{align}
as the set of admissible solutions for the compressible Euler equations \eqref{euler:eq}.  Since $-p(\bu)$ is convex, we have using  Jensen's inequality for any $0\leq s\leq 1$, $\bw_1=(\rho_1,m_1,E_1)^\T\in G$, and $\bw_2=(\rho_2,m_2,E_2)^\T\in G$  we have
\begin{align}
p(s\bw_1+(1-s)\bw_2)\geq sp(\bw_1)+(1-s)p(\bw_2)>0, \quad \text{if }\; \rho_1\geq 0,\;\rho_2\geq0.
\end{align}
Thus $G$ is a convex set.

Numerical methods that do not  preserve the positivity property,  may cause numerical instability or nonphysical features in approximated solutions of problems with low density, or high Mach number \cite{ha2005numerical}. Preserving positivity of pressure and density is therefore an important property of a numerical method for the Euler equations, and it can be seen as a generalization of the maximum principle for scalar conservation laws. As for the scalar problems,  we will be working with  mean values on macro-element $M\in\mathcal{M}_{h,i},i=1,\cdots,N_\Omega$.  For each component of $\bu_h$, we define the mean values on the macro-element $M$ as in \eqref{eq:def:mean:scalar}, and for the pressure we define $\overline p_{h,i}^{M}:=p(\overline{\bw}_{h,i}^{M})=(\gamma-1)\left(\overline{E}_{h,i}^{M}-\frac{1}{2}\left(\bar{m}_{h},i^{M}\right)^2 / \bar{\rho}_{h,i}^{M}\right)$. By $\overline{\bw}_{h,i}^M\in G$ we mean $\overline\rho_{h,i}^M>0$ and $\overline p_{h,i}^M>0$.

\subsection{Piecewise constant approximations}
In this subsection, we will study the piecewise constant scheme, \eqref{scheme:cutDG2fully}-\eqref{eq:schemestep3}, for the compressible Euler equations \eqref{euler:eq}. 
We are particularly interested in the positivity property,  which was discussed above.  For the piecewise constant method we prove the following positivity result.  
\begin{thm}\label{thm:Euler:1ord}
  Consider the piecewise constant Cut-DG scheme described by \eqref{scheme:cutDG2fully}-\eqref{eq:schemestep3} applied to the Euler equations \eqref{euler:eq}. The reconstruction \eqref{eq:scalar:reconstruction} is done component wise. 
  If the time step $\Delta t$ satisfies 
\begin{align}\label{eq:euler:p0:time}
\frac{\Delta t}{\delta h}\lambda\leq 1,
\end{align}
then  $\bw_{h,i}^n|_{\Omega_i} \in G, i=1,\cdots,N_\Omega$ implies % where $m=\min_{x\in\Omega}{u_0(x)},M=\max_{x\in\Omega}{u_0(x)}$, 
 $\bw^{n+1}_{h,i}|_{\Omega_i}\in G, i=1,\cdots,N_\Omega$. 
  Here $h$ is the element size in the background mesh,  and $\delta$ is a given constant, which in \eqref{eq:largeel} defines when an intersection between an element and $\Omega_i$ is classified as large. 
  $\lambda=\max_{x\in\Omega}|u|+c$ is the largest absolute eigenvalue of Jacobian matrix $\frac{\partial \mathbf{f}(\bu_h^n)}{\partial \bu_h^n}$ in the entire domain $\Omega$.
\end{thm}
\begin{proof}
As before we consider each macro-element $M\in \mcM_{h,i}, i=1,\cdots, N_\Omega$ separately. Let $\bu_{h,l}^n=(\rho_{h,l}^n,m_{h,l}^n,E_{h,l}^n)$, $\bu_{h,r}^n=(\rho_{h,r}^n,m_{h,r}^n,E_{h,r}^n)$, $\bu_{h,i}^{M,n}$ denote the solution on the left side and right side of $I_\Me$, and in $I_M$, respectively. By assumption $\bw_{h,l}^n,\bw_{h,i}^{M,n},\bw_{h,r}^n\in G$,  and   $\bw_{h,i}^{M,n}$ is constant in $I_M$. Following the same analysis as for the scalar conservation law, we get
\begin{align}\label{scheme:Euler}
{\bu}_{h,i}^{M,n+1}&=\boldsymbol{H}(\bw_{h,l}^n,\bu_{h,i}^{M,n},\bu_{h,r}^n;\lambda),\\
\boldsymbol{H}(\bu_{h,l}^n,\bu_{h,i}^{M,n},\bu_{h,r}^{n};\lambda)&:=
{\bu}_{h,i}^{M,n}-\frac{\Delta t}{|I_M|}\left( \hat{\mf}_r-\hat{\mf}_l \right).
\label{def:euler:H}
\end{align}
Here $\hat{\mf}_r$ and $\hat{\mf}_l$  are  the numerical fluxes on the right and  the left edges of $I_M$, respectively. We use the global Lax-Friedrichs flux \eqref{eq:def:lf} with $\lambda=\max_{x_\in \Omega}\{|u|+c\}$ , which is the largest absolute value of eigenvalues of the Jacobian matrix $\frac{\partial \mathbf {f(\bu_h^n)}}{\partial \bu_h^n}$.  
Next, we study the positivity property for the density and pressure of \eqref{scheme:Euler}. From the definition of  the global Lax-Friedrichs   flux \eqref{eq:def:lf} applied to the particular Euler flux  \eqref{euler:eq:flux}  we have
\begin{align}
{\rho}_{h,i}^{M,n+1}&={\rho}_{h,i}^{M,n}-\frac{\Delta t}{2|I_M|}(m_{h,i}^{M,n}+m^n_{h,r}-\lambda(\rho^n_{h,r}-\rho_{h,i}^{M,n}))+\frac{\Delta t}{2|I_M|}(m_{h,i}^{M,n}+m^n_{h,l}-\lambda(\rho_{h,i}^{M,n}-\rho^n_{h,l}))\notag\\
&=\left(1-\frac{\Delta t}{|I_M|}\lambda\right)\rho_{h,i}^{M,n}+\frac{\Delta t}{2|I_M|}\rho^n_{h,r}(\lambda-u^n_{h,r})+\frac{\Delta t}{2|I_M|}\rho^n_{h,l}(\lambda+u^n_{h,l}).\label{eq:euler:rho}
\end{align}
From $\lambda=\max_{x\in\Omega}|u|+c$ it follows that $\lambda\pm u\geq 0$, and since $|I_M|\geq \delta h$ and 
\begin{align}
\rho_{h,i}^{M,n}>0, \, \rho^n_{h,l}>0, \, \rho^n_{h,r}>0, \quad 1-\frac{\lambda\Delta t}{|I_M|}>0,
\end{align}
we   conclude from equation \eqref{eq:euler:rho}  that ${\rho}_{h,i}^{M,n+1}>0$ under time step restriction \eqref{eq:euler:p0:time}. 
To show that $\bw_{h,i}^{M,n+1}\in G$ it remains to  consider the pressure that corresponds to $\bw_{h,i}^{M,n+1}$.  Following the analysis in \cite{FU2022111600}, we define $V_a(\bw_{h,l}^n)={\mf}(\bw_{h,l}^n)/{\lambda}+\bw_{h,l}^n$ and $V_b(\bw_{h,r}^n)=\bw_{h,r}^n-{\mf}(\bw_{h,r}^n)/{\lambda}$, and rewrite \eqref{def:euler:H} as
\begin{align}
\boldsymbol{H}(\bw_{h,l}^n,\bw_{h,i}^{M,n},\bw_{h,r}^n;\lambda)=(1-\frac{\lambda\Delta t}{|I_M|}){\bw}_{h,i}^{M,n}+\frac{\lambda\Delta t}{2|I_M|}V_a(\bw_{h,l}^n)+\frac{\lambda\Delta t}{2|I_M|}V_b(\bw_{h,r}^n).
\end{align}
As in \cite{FU2022111600} (see Lemma 3.1),  we have  $p(V_a)>0$ and $p(V_b)>0$  if  $\bw_{h,l}^n,\bw_{h,r}^n\in G$. By the convexity of $G$ we obtain the desired result. We refer to the  corresponding analysis for the   DG method on a non-cut mesh  in \cite{zhang2010positivity} and \cite{FU2022111600}  for more details.  
\end{proof}

\subsection{Mean values for higher order approximations}

Next we consider  high order schemes, where  as before the time discretization can be seen as a sequence of Euler steps. We therefore consider \eqref{scheme:cutDG2fully} with $p\ge 1$ together with the reconstruction in \eqref{eq:scalar:reconstruction}-\eqref{eq:schemestep3}.  Similarly to the scalar case we show a sufficient condition for the mean values on each macro-element to satisfy the  positivity property. We again use the $q$-point Gauss-Lobatto quadrature rule on each macro element, with $2 q-3 \geqslant p$. We will prove the following theorem.
\begin{thm}
Consider \eqref{scheme:cutDG2fully}-\eqref{eq:schemestep3} with $p\ge 1$ for  the Euler equations \eqref{euler:eq}, and with the reconstruction  done component-wise. 
 If the time step  $\Delta t$ satisfies 
\begin{align}\label{eq:euler:pp:time}
\frac{\Delta t}{\delta h}\lambda\leq \hat{w}_1,
\end{align}
then  $\bu_{h,i}^n|_{\Omega_i} \in G, i=1,\cdots,N_\Omega$ implies that the mean value on each macro element satisfies
$\bar{\bu}_{h,i}^{\mathcal{M},n+1}|_{\Omega_i}\in G$ for $i=1,\cdots,N_\Omega$. Here $h$ is the element size in the background mesh,  and $\delta$ is a given constant, which in \eqref{eq:largeel} defines when an intersection between an element and $\Omega_i$ is classified as large.  
$\lambda$ denotes the largest absolute value of the  Jacobian matrix $\partial \mf(\bu_h^n)/\partial \bu_h^n$ in the entire domain $\Omega$. Further $\hat w_1$ is the normalized  quadrature weight corresponding to the first quadrature  point of a $q$-point Gauss-Lobatto quadrature rule with  $2q-3\geq p$.
\end{thm}
\begin{proof} Similar to the high order Cut-DG scheme \eqref{eq:define:scheme:average} for scalar problems,   we have for the mean values of the  reconstructed approximation on  $I_M$ in \eqref{eq:def:IM} that 
\begin{align}\label{eq:scheme:uhm:2}
\overline{\bw}_{h,i}^{M,n+1}=\overline{\bw}_{h,i}^{M,n}-\frac{\Delta t}{|I_M|}(\hat{\mf}_{r}-\hat{\mf}_{l}).
\end{align}
We will now show that $\overline{\bu}_{h,i}^{M,n+1} \in G$.  We introduce the Gauss-Lobatto quadrature to decompose the mean values on each macro-element, see \eqref{eq:define:average:decompose}, and rewrite  \eqref{eq:scheme:uhm:2} as
\begin{align*}
\overline{\bw}_{h,i}^{M,n+1}=&\sum_{\mu=1}^q \hat{w}_\mu\hat{\bw}^{M,n}_{\mu}-\frac{\Delta t}{|I_M|}(\hat{\mf}_r-\hat{\mf}_l)=\\
=&\sum_{\mu=2}^{q-1} \hat{w}_\mu\hat{\bw}^{M,n}_{\mu}
+\hat{w}_1\boldsymbol{H}(\hat{\bw}_{0}^{M,n},\hat{\bw}^{M,n}_{1},\hat{\bw}^{M,n}_{q};
\frac{\lambda}{\hat{w}_{1}})
+\hat{w}_q\boldsymbol{H}(\hat{\bw}^{M,n}_{1},\hat{\bw}^{M,n}_{q},\hat{\bw}^{M,n}_{q+1};\frac{\lambda}{\hat{w}_{q}}).
\end{align*}
Here  $\hat{\bu}^{M,n}_{0}=\bu_{h,l}^{n}(x_{M,l})$ and $\hat{\bu}^{M,n}_{q+1}=\bu_{h,r}^{n}(x_{M,r})$ denote the approximations from the left side and right side of $I_M$, respectively.    With the positivity-property of the piecewise constant scheme \eqref{scheme:Euler} and time step restriction \eqref{eq:euler:pp:time},  it follows that $\frac{\Delta t}{|I_M|}
\lambda\leq\hat{w}_1$ implies $\boldsymbol{H}(\hat{\bw}_{0}^{M,n},\hat{\bw}^{M,n}_{1},\hat{\bw}^{M,n}_q;\frac{\lambda}{\hat{w}_{1}})\in G$ and $\boldsymbol{H}(\hat{\bw}^{M,n}_1,\hat{\bw}^{M,n}_q,\hat{\bw}^{M,n}_{q+1};\frac{\lambda}{\hat{w}_q})\in G$. Thus, we  get $\overline{\bw}_{h,i}^{M,n+1}\in G$ since $\overline{\bu}_{h,i}^{M,n+1}$ is a convex combination of solutions in $G$.
\end{proof}

\subsection{Positivity preserving limiting for the Euler equations}
As for the standard DG method for the Euler equations on a non-cut mesh, we need to apply a positivity preserving limiter to ensure the positivity of density and pressure for $p\geq 1$. We will use the limiter developed by Zhang ad Shu \cite{zhang2010positivity}, which is  briefly introduced below. It retains conservation everywhere, and accuracy in smooth regions. We will apply it  to the reconstructed solution in each inner stage of each time step. By the previous theorem, we know that  mean values on all macro-elements  are in $G$. Let $M\in \mathcal{M}_{h,i}$ represent a macro-element in $\Omega_i$ and $\bu_h^M=(\rho_h^M, m_h^M, E_h^M)$ denote the corresponding solution on the macro-element $M$.  By  the previous theorem  $\overline \rho_{h}^{M}>0$ and $\overline p_{h}^{M}>0$. We assume there exists a small number $\varepsilon>0$ such that $\bar{\rho}_{h}^{M} \geqslant \varepsilon$ and $\bar{p}_{h}^{M} \geqslant \varepsilon$ for all macro-elements. In our computation,  we take $\varepsilon=10^{-8}$.  We first limit the density by replacing $\rho_h^M(x)$ by $\rho_h^{M,(l)}(x)$
\begin{align}\label{euler:limit:rho}
{\rho}_{h}^{M,(l)}(x)=\theta_1\left(\rho_{h}^{M}(x)-\bar{\rho}_{h}^{M}\right)+\bar{\rho}_{h}^{M},
\end{align}
where
\begin{align}\label{euler:limit:rho:theta1}
\theta_1=\min \left\{\frac{\bar{\rho}_{h}^{M}-\varepsilon}{\bar{\rho}_{h}^{M}-\rho_{\min }}, 1\right\}, \quad \rho_{\min }=\min _k \rho_{h}^{M}\left(\widehat{x}_{\mu}^{M}\right) .
\end{align}
 Then, let $\widehat{\mathbf{q}}_h^M(x)=\left({\rho}_{h}^{M,(l)}(x), m_{h}^{M}(x), E_{h}^{M}(x)\right)^T$ and  denote $\widehat{\mathbf{q}}_h^M\left(\hat x_\mu^M\right)$  by $\widehat{\mathbf{q}}_{\mu}^{M}$.    
 Next, $\widehat{\mathbf{q}}_{h}^{M}$  is limited to enforce the positivity of the pressure.  Define
\begin{align}\label{euler:limiter:G}
&G^{\varepsilon}=\left\{\bw=(\rho,m,E)^\T \mid \rho \geqslant \varepsilon \text { and } p \geqslant \varepsilon\right\}, \, \partial G^{\varepsilon}=\left\{\bw=(\rho,m,E)^\T\mid \rho \geqslant \varepsilon \text { and } p=\varepsilon\right\},
\end{align}
and
\begin{align}\label{euler:limiter:s}
\mathbf{s}^\mu(t)=(1-t) \overline{\bw}_{h}^{M,n}+t\,\widehat{\mathbf{q}}_{h}^{M}\left(\widehat{x}_{\mu}^{M}\right), \quad 0 \leqslant t \leqslant 1 .
\end{align}

Let  $\mathbf{s}_{\varepsilon}^\mu$ be solution of  $p\left(\mathbf{s}^\mu\left(t_{\varepsilon}^k\right)\right)=\varepsilon$ when  $p\left(\widehat{\mathbf{q}}_{\mu}^{M}\right)<\varepsilon$, and $\mathbf{s}_{\varepsilon}^\mu=\widehat{\mathbf{q}}_{\mu}^{M}$ if $p\left(\widehat{\mathbf{q}}_{\mu}^{M}\right) \in G^{\varepsilon}$. 
Finally, we modify the solution by 
\begin{align}\label{euler:limiter:p}
{\bw}_{h}^{M,(l)}(x)=\theta_2\left(\hat{\mathbf{q}}_{h}^{M}(x)-\overline{\bw}_{h}^{M}\right)
+\overline{\bw}_{h}^{M}, \quad \theta_2=\min _{\mu=1,2, \ldots, q} \mathbf{s}_{\varepsilon}^\mu. 
\end{align}
We note that it is hard to find the exact minimum value $\theta_2$ on the cut mesh. We use the Gauss-Lobatto points on the macro-element and on each element intersection $\Omega_i$ involved in the macro-element $M$. Finally, we set $\bu_{h,i}^{\mathcal{M},n}|_{\Omega_i}={\bw}_{h}^{M,(l)}(x)$ on the macro-element $M$.  

For clarity, we formulate our proposed positivity preserving Cut-DG method  for the  Euler equations in Algorithm 2 in the Appendix.

\section{Numerical results} \label{sec:numres}
In this section we present some  numerical examples that verify the bound preserving properties of our proposed bound preserving  Cut-DG scheme, which includes reconstruction on macro-elements and applying bound preserving limiters.  In all  the Cut-DG computations we  use a uniform mesh with $N$  elements as the background mesh and $N\geq 4$.  We assume we have $N/4$ subdomains interfaces, which cut  $N/4$ elements in the middle part of the domain. Each interface cuts an element  into two pieces with size $\alpha_k h$ and $1-\alpha_k h$, respectively, where $\alpha$ is fixed and 
\begin{equation}\label{eq:define:alphak}
\alpha_k=s\alpha,  \,  s \in[10^{-6},1] \text{ is a random number.}
\end{equation}
Thus we have $N/4+1$ subdomains. The parameters in the stabilization (see \eqref{stable1}) and  in \eqref{eq:largeel} are set to $\gamma_1=0.75,\gamma_0=0.25$, $\omega_k=1/(k!)^2$, and $\delta=0.2$,  respectively.

Two different high order strong stability preserving (SSP) methods are as time discretization. To define the methods consider $u_t=L(u)$, where $L(u)$ is the spatial operator. The first method is the third order Runge-Kutta method,  which is  convex combinations of  the forward Euler method
\begin{align}\label{eq:TVDRK3}
&u^{(1)}=u^{n}+\Delta t L\left(u^{n}\right) \\
&u^{(2)}=\frac{3}{4} u^{n}+\frac{1}{4}u^{(1)}+\se{\frac{\Delta t }{4}}L\left(u^{(1)}\right), \\
&u^{n+1}=\frac{1}{3} u^{n}+\frac{2}{3}u^{(2)}+\frac{2\Delta t}{3} L\left(u^{(2)}\right).
\end{align} 
 The second scheme is the third order  multi-step time discretization \cite{gottlieb2001strong} 
\begin{align}\label{eq:multi-step3}
u^{n+1}=\frac{16}{27}\left(u^{n}+3 \Delta t L\left(u^{n}\right)\right)+\frac{11}{27}\left(u^{n-3}+\frac{12}{11} \Delta t L\left(u^{n-3}\right)\right).
\end{align}
By the previous analysis each time step and interior stage will be bound preserving if the time step is sufficiently restricted. It follows that the proposed strategy, with both time discretizations, will satisfy  the maximum principle or positivity of density and pressure.  In the next section, we compute the two time discretization schemes and show that optimal accuracy is obtained with third order multi-step time discretizations. We note that we use $P^p$ with $p=1,2,3$ to denote the space of piecewise $p$-th degree polynomials $V_h^p$.

\subsection{Scalar linear case}
In this subsection,  we solve the one dimensional linear advection equation with  periodic boundary condition
\begin{align}\label{eq:test:linear}
&u_t+u_x=0, t>0,
\end{align}
together with  smooth and non-smooth initial data, respectively.
\subsubsection{Accuracy test with smooth initial data}
We first use 
our proposed bound preserving Cut-DG method \eqref{scheme:cutDG2fully}-\eqref{eq:schemestep3} together with the bound preserving limiter in \eqref{eq:limit:uj} to solve \eqref{eq:test:linear},  with smooth initial data  $u(x,0)=1.0+0.5\sin(\pi x)$ in the domain $[0,2]$. The exact solution is $u(x,t)=1+0.5\sin(\pi(x-t))$. The problem 
  is solved on a mesh where the elements in interval $[0.75,1.25]$ are cut. In our test, we  use $\alpha=0.1$ in \eqref{eq:define:alphak} with $N=20,40,80,160,320,640$.  Other alpha values have also been tested, with similar results ( and therefore not reported here).  
We first solve the problem using the third order SSP-RK method \eqref{eq:TVDRK3} and the time step is taken to be $\Delta t=Ch$ for second and third order approximations and $\Delta t=Ch^{4/3}$ for fourth order approximations, where $C=0.5(1-\max(\alpha_k))\hat{w}_1$.  Here $h^{4/3}$ is used for $P^3$ approximations to balance the errors  caused by time and space discretization, and $\hat{w}_1$ is the weight  used in the Gauss-Lobatto quadrature.  
 The results at $t=1$ are shown in Figure \ref{fig:smoothinitial:accuracy}. We can observe that the Cut-DG method using the third Runge-Kutta  method and the bound preserving limiter has accuracy degeneracy. This is also the case for  the standard DG method with Runge-Kutta time discretization \cite{zhang2010maximum}.  When the third order  multi-step method \eqref{eq:multi-step3} is applied with $\Delta t=h/24, h^{4/3}/15$ for $P^2,P^3$ approximation, the optimal accuracy is obtained, see the right panel of  Figure \ref{fig:smoothinitial:accuracy}.  
The results verify that  the reconstruction of solutions and  the bound preserving limiter on  macro-elements in our proposed  method can keep the same  high order accuracy  as the standard Cut-DG method \cite{fu2021high,fu2022high}.  The results also show that the  proposed scheme enforces the maximum principle.  

\begin{figure}[tbhp]
  \centering
%\subfigure[$P^1$]{
%\includegraphics[width=2.0in]{Linearaccuracy_DGRK3}
\includegraphics[width=2.8in]{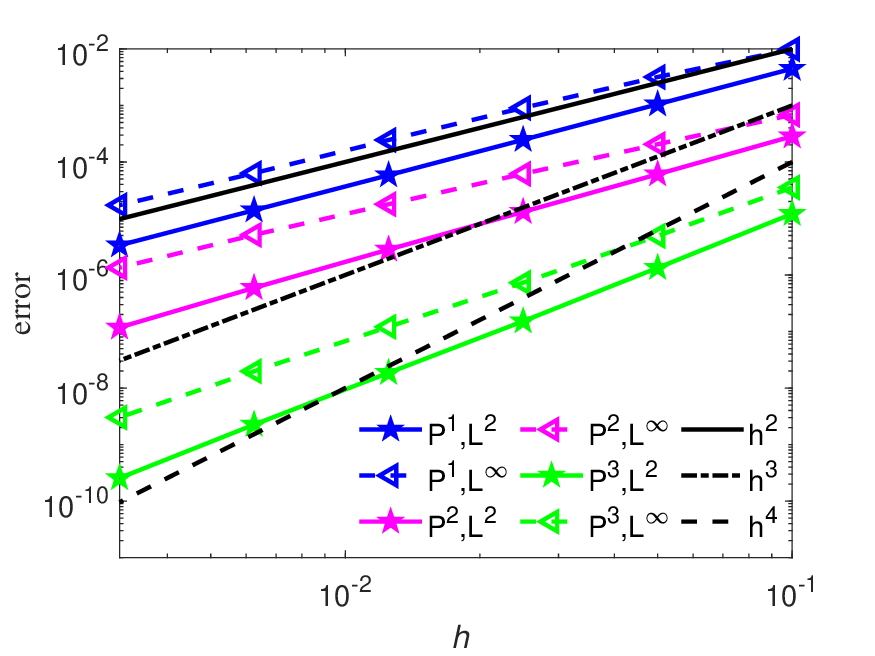}
\includegraphics[width=2.8in]{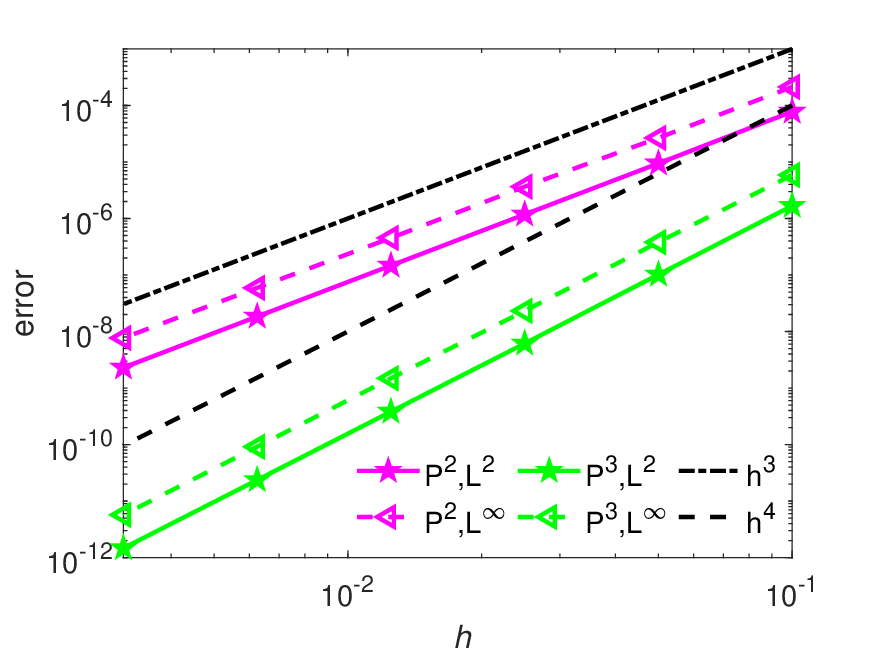}
\caption{The errors in the numerical solutions  for the advection equation  with smooth initial data at $t=1.0$. Errors are shown in both the $L^2$-norm (solid lines) and $L^\infty$-norm (dashed lines).  Left:  time discretization using  third order Runge-Kutta method. Right: time discretization  using  third order multi-step discretization in \eqref{eq:multi-step3}.
 }\label{fig:smoothinitial:accuracy}
\end{figure}

\subsubsection{Non-smooth initial data}
Here, we solve the advection problem \eqref{eq:test:linear} with periodic boundary condition on the   domain $\Omega=[0,1]$ and non-smooth initial data
\begin{equation}\label{eq:nonsmoothinitial}
u(x,0)=
\left\{\begin{array}{ll}
{1} & {0.1<x<0.5}, \\
{0} & {\text{otherwise.}}
\end{array}\right.
\end{equation}
 We use a computational mesh, where the elements in the interval $[0.375,0.625]$ are cut.  In particular we note that the discontinuity is initially  located on the common edge of one cut element and its neighbour.  
We first solve this problem using the cut DG discretization  \eqref{scheme:cutDG2} with $p=1,2$, but without reconstruction and without any bound preserving limiter. We  observe  overshoots and undershoots of the solution near the discontinuities, see the left panel of Figure  \ref{fig:nonsmoothinitial:p1:overshoot}.
 In the middle panel of Figure \ref{fig:nonsmoothinitial:p1:overshoot},  we show the results when the bound preserving limiter \eqref{eq:limit:uj} is applied to the solution on each element. A  clear overshoot still exists in the second order approximations, which will decrease with increasing  degree of the  polynomials. We point out that we still observe overshoots or undershoots during time evolution in the $P^2,P^3$ approximations but they are quite small and decrease with time. The results from our proposed Cut-DG method with bound preserving limiter \eqref{eq:limit:uj} and reconstruction \eqref{eq:scalar:reconstruction} on the macro-elements are shown in the right panel of Figure \ref{fig:nonsmoothinitial:p1:overshoot}. We did not observe any overshoots or undershoots up to the final time. We also solved this problem on a mesh where the cut elements are randomly located and have  different cut sizes, with  similar results.
\begin{figure}[!htbp]
  \centering
%\subfigure[$P^1$]{
\includegraphics[width=1.8in]{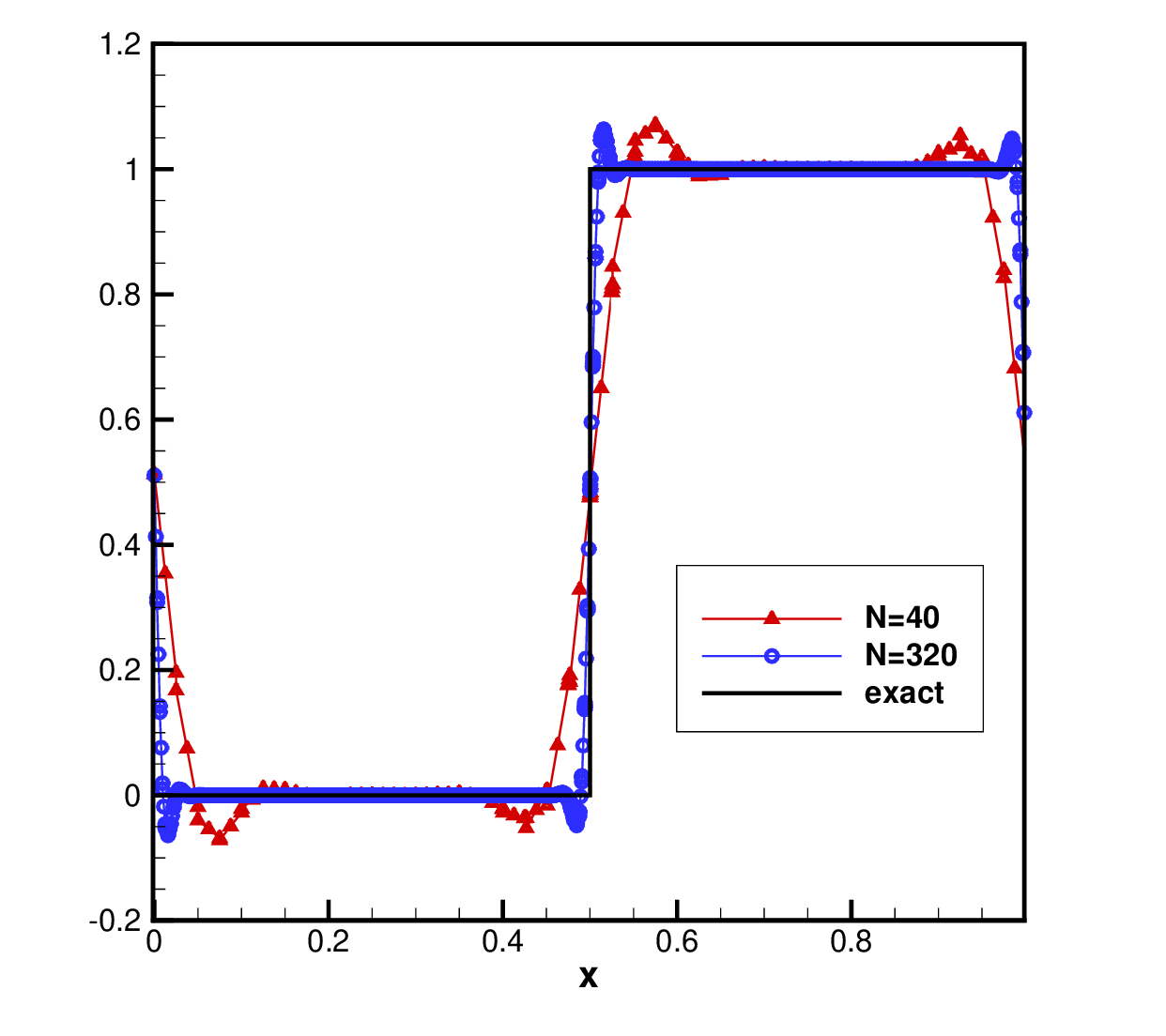}
\includegraphics[width=1.8in]{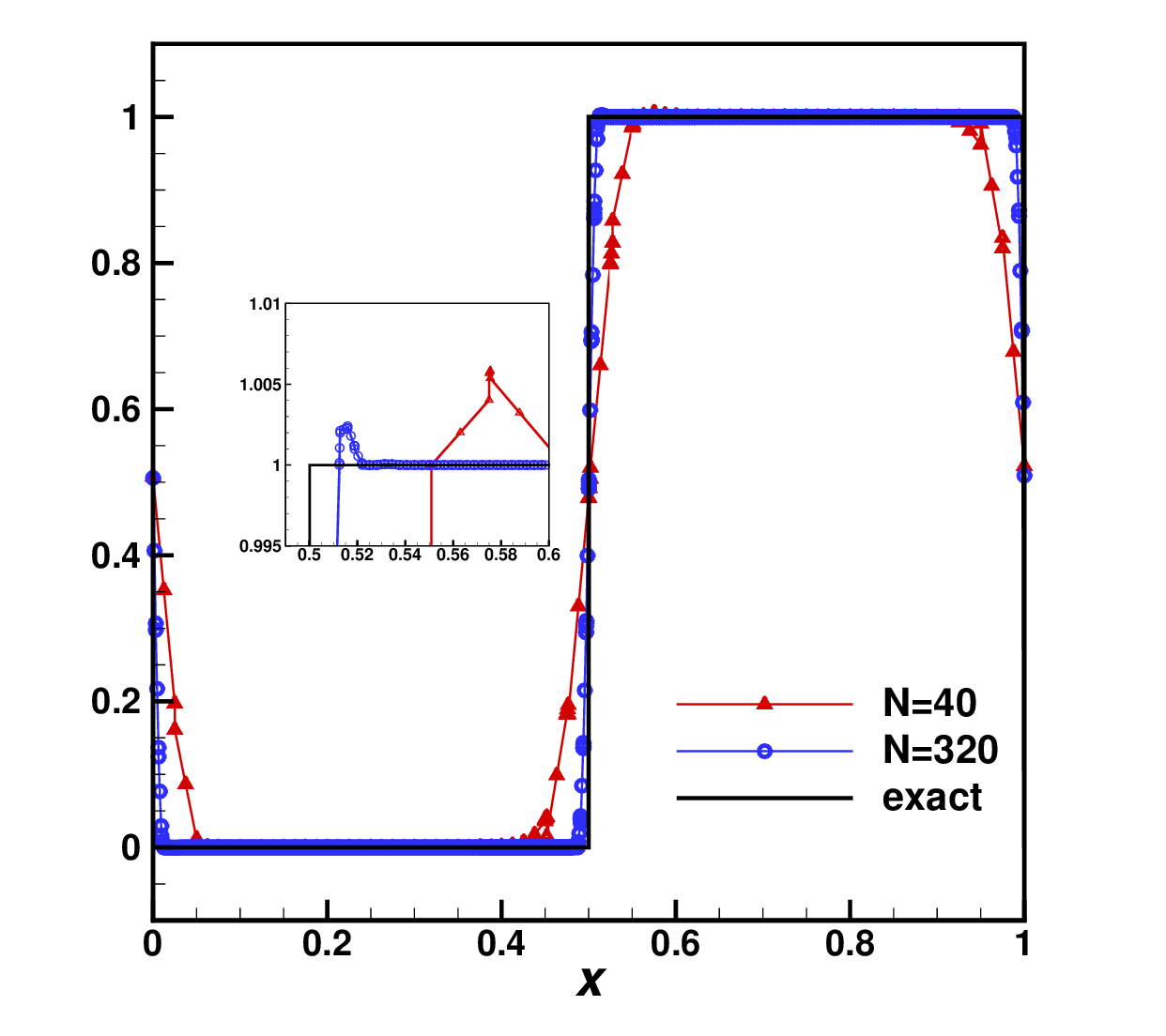}
\includegraphics[width=1.8in]{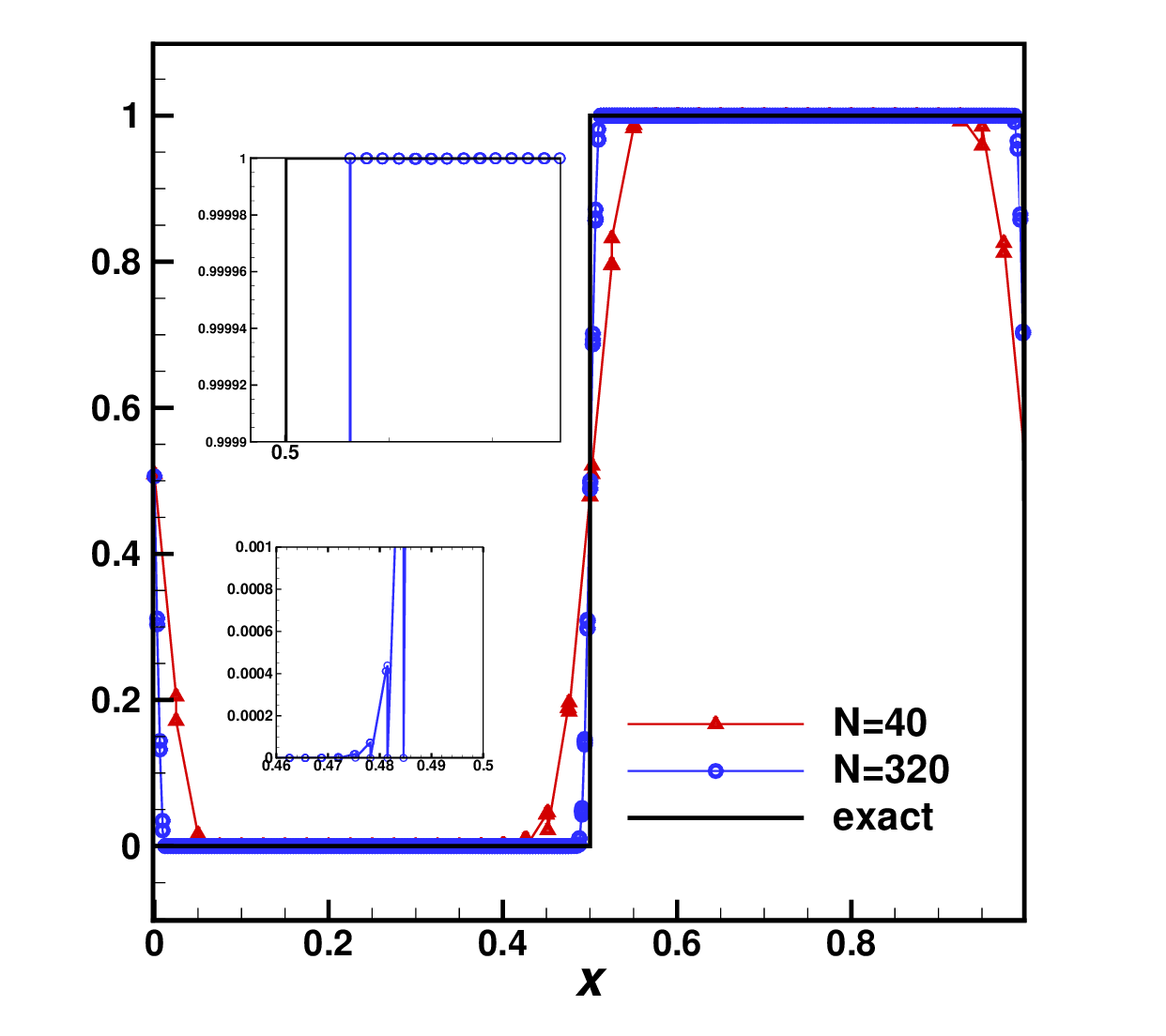}
\includegraphics[width=1.8in]{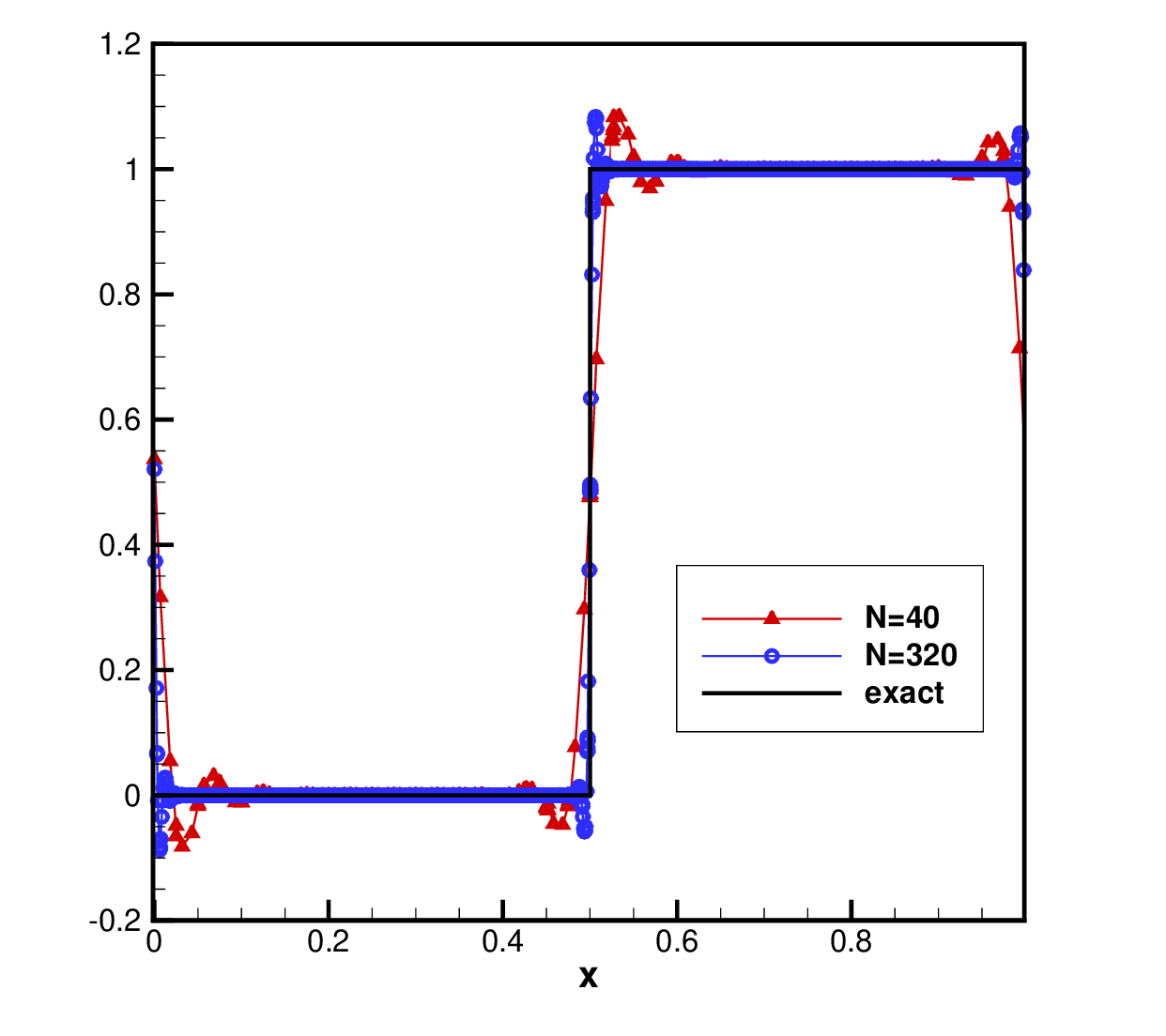}
\includegraphics[width=1.8in]{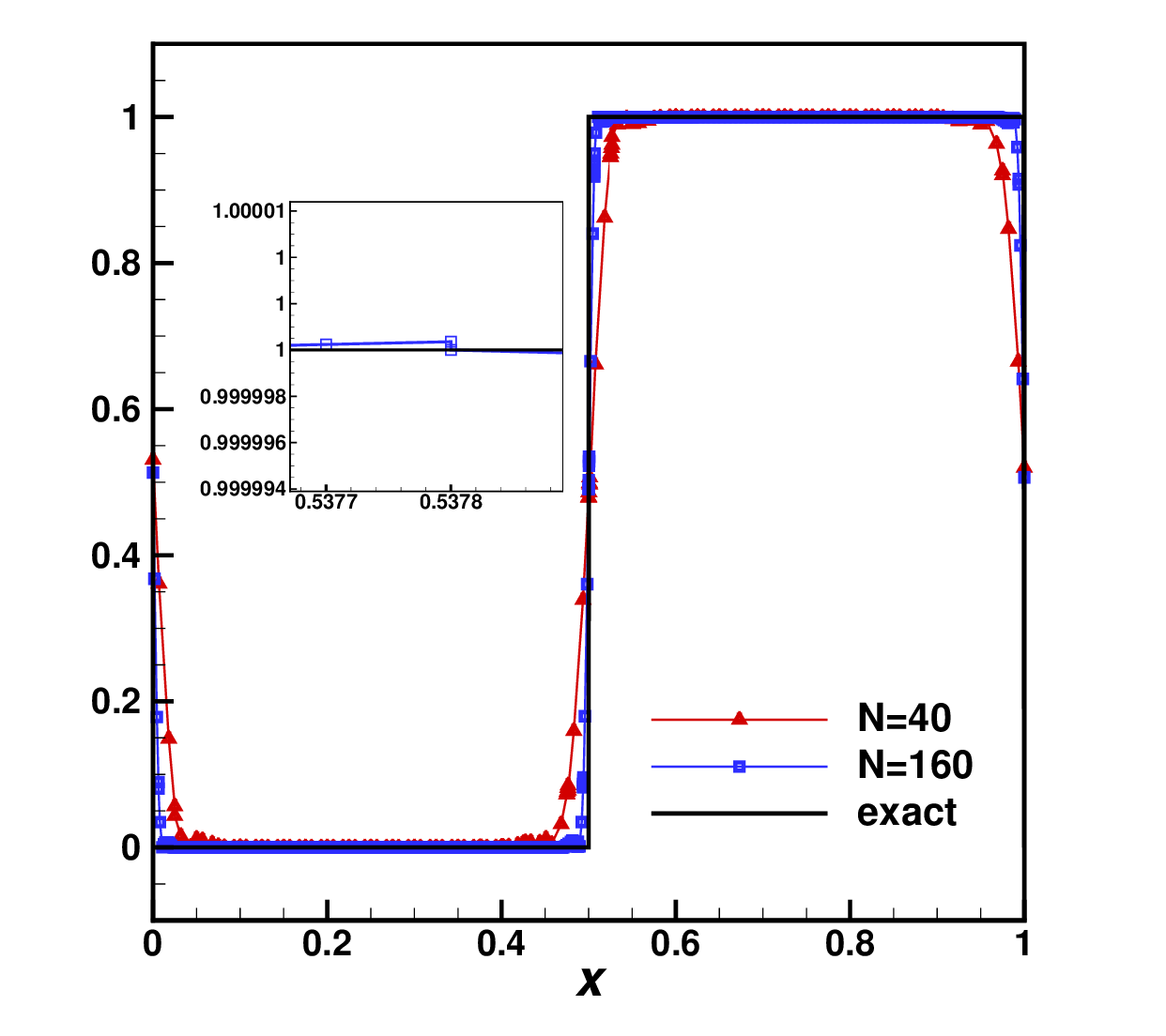}
\includegraphics[width=1.8in]{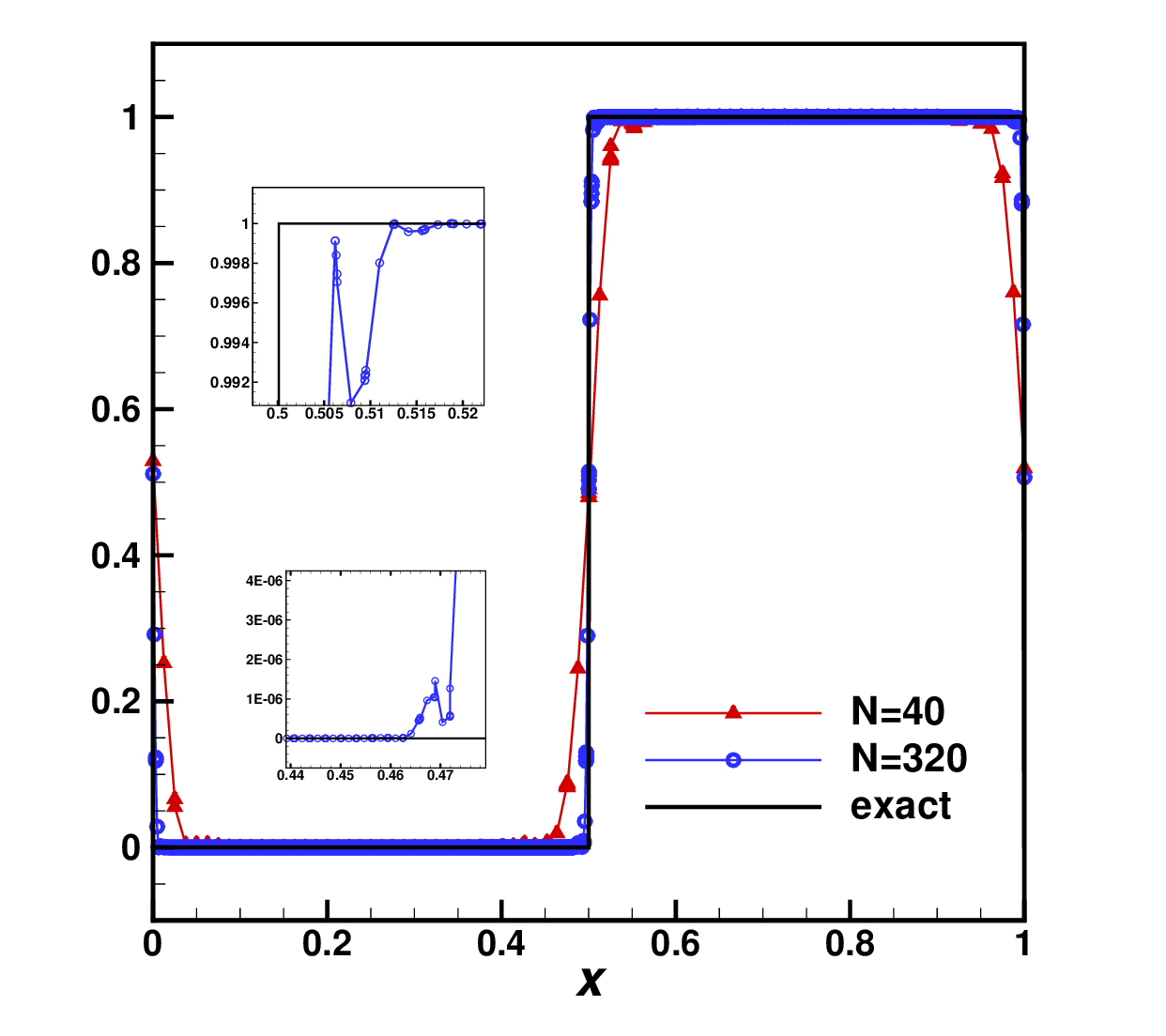}
\caption{Numerical $P^1$ (top), $P^2$ (bottom) solutions  for the advection equation  with non-smooth initial data at $t=1.0$.  
Left: without reconstruction and without limiter. 
 Middle: without reconstruction but  limiter \eqref{eq:limit:uj} is applied. Right: solution with reconstruction and limiter \eqref{eq:limit:uj} applied on macro-elements.
 }\label{fig:nonsmoothinitial:p1:overshoot}
\end{figure}

\subsection{Burgers'  equation}
In this subsection, we consider the Burgers' equation
\begin{align}\label{eq:Burges}
&u_t+\left(\frac{u^2}{2}\right)_x=0,\quad t>0,
\end{align}
together with different initial data  and suitable boundary conditions.

\subsubsection{Smooth initial data}
We first solve  the Burgers' equation \eqref{eq:Burges} with smooth  initial data $u_0(x)=\sin(\pi x), x\in[0,2]$ and periodic boundary condition. This problem has a known solution, which we use as a reference when computing errors. The solution is initially smooth, but at $t=\frac{1}{\pi}$ a shock forms at $x=1$.  We compute until time $t=0.2$, which is before the shock appears. Uniform meshes with $N=20,40,80,160,320,640$ elements are used as background mesh. The cut elements are located in $[0.75,1.25]$. 
The  third order multi-step \eqref{eq:multi-step3} is used for time discretization. In Figure \ref{fig:burgers:smoothinitial:accuracy}, the $L^2$- and $L^\infty$-errors are shown from our  proposed bound preserving Cut-DG scheme. We observe the third and fourth order accuracy  for $P^2,P^3$ approximations,  respectively.  

\begin{figure}[!htbp]
  \centering
%\subfigure[$P^1$]{
\includegraphics[width=2.8in]{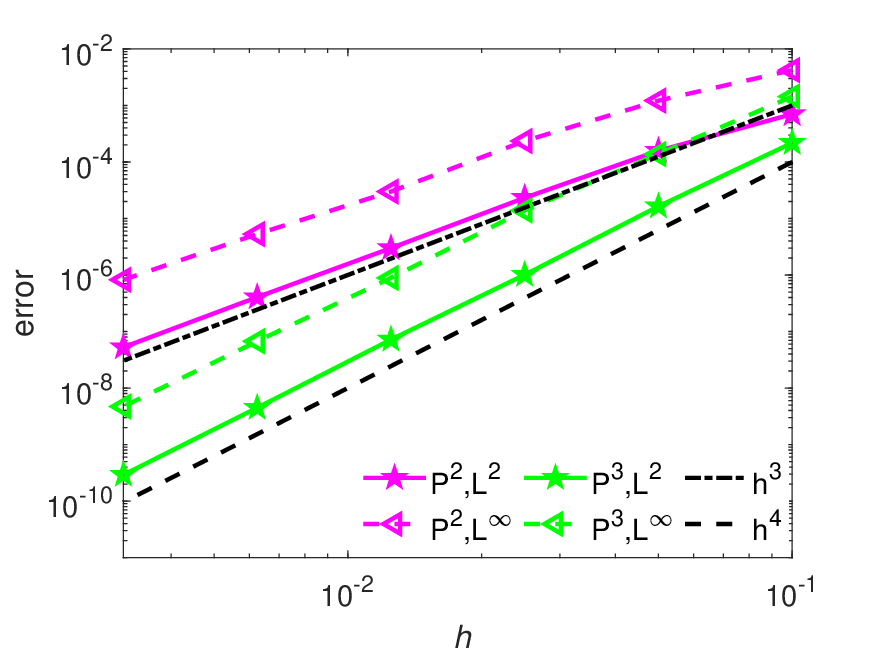}
\caption{ Errors in the numerical solution of Burgers' equation  with smooth initial data, before a shock has formed. Our proposed Cut-DG method based on piecewise $P^2$ and $P^3$ approximations, with reconstruction and bound preserving limiter  \eqref{eq:limit:uj}was used  on macro-elements.  
 }\label{fig:burgers:smoothinitial:accuracy}
\end{figure}

Next we solve Burgers' equation with $P^3$ approximations until time $t=0.5$,  when a shock has formed. The third order RK method \eqref{eq:TVDRK3} is used for the time discretization to save computational cost, since multi-step methods need to save more informations on different time levels. 
Results for $h=1/20,1/320$ (corresponding to $10, 160$ small cut elements) with the bound preserving limiter are shown in Figure \ref{figure:burgesshock2:onlymacro}.  The left panel in  Figure \ref{figure:burgesshock2:onlymacro} shows the results from the scheme \eqref{scheme:cutDG2fully} without reconstruction but with bound preserving limiter. Note that  the numerical  solutions do not satisfy maximum principle.  The middle panel in Figure \ref{figure:burgesshock2:onlymacro}   shows the results from our proposed bound preserving Cut-DG method with reconstruction on all macro-elements, which verifies our theoretical analysis.  We did not observe any overshoots/undershoots at any time. Note that the proposed scheme on the fine mesh can capture the shock quite well, even though we did not apply a slope limiter. On the right side of Figure \ref{figure:burgesshock2:onlymacro}, we show the results from the scheme with the reconstruction \eqref{eq:def:uhm} applied only on those macro-elements where the solution does not satisfy the maximum principle in an element belonging to these macro-elements. We can observe that the numerical solution still satisfies the maximum principle, even when the reconstruction is not applied on all macro-elements.
\begin{figure}[!htbp]
  \centering
  \includegraphics[width=1.8in]{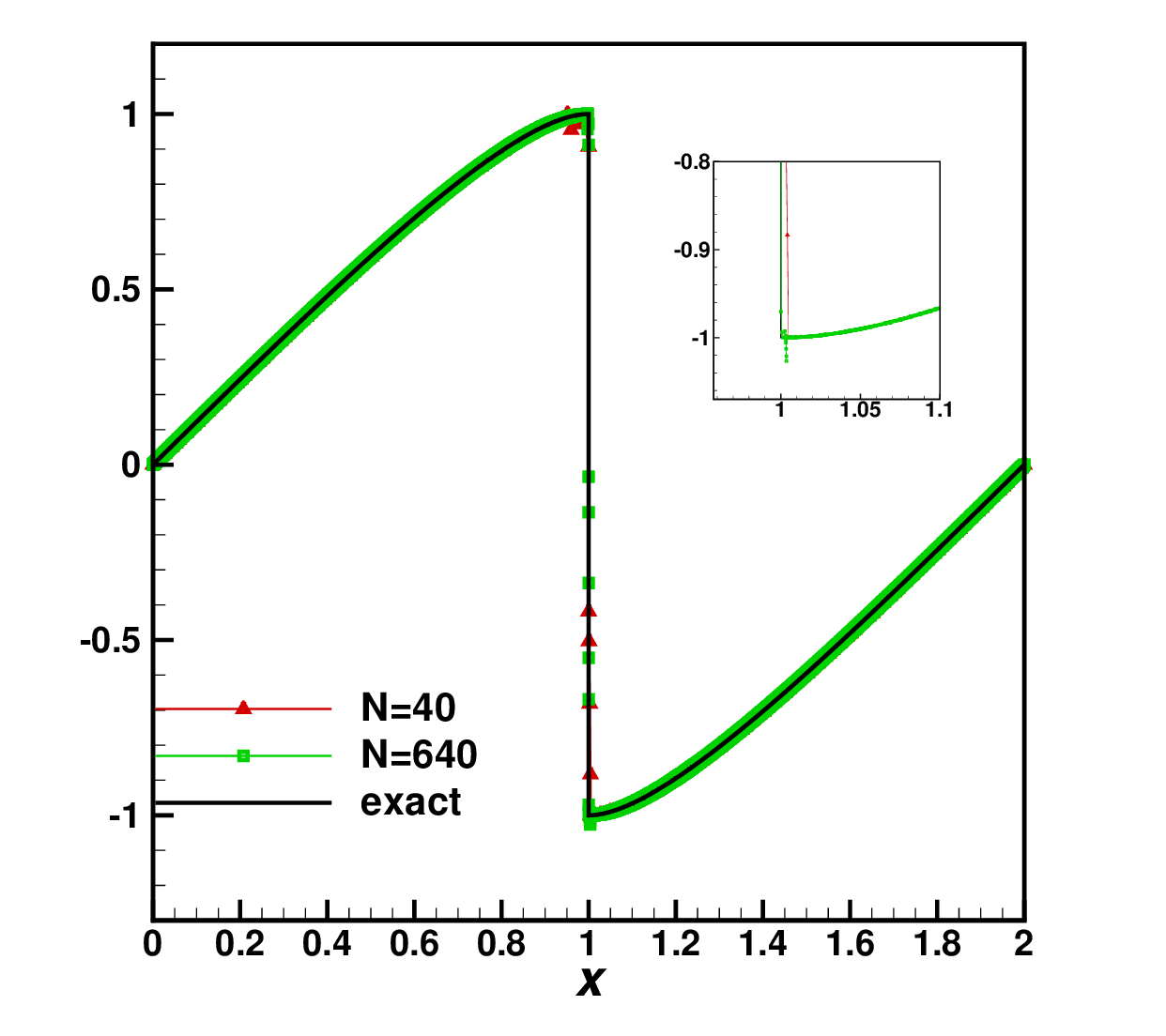}
  \includegraphics[width=1.8in]{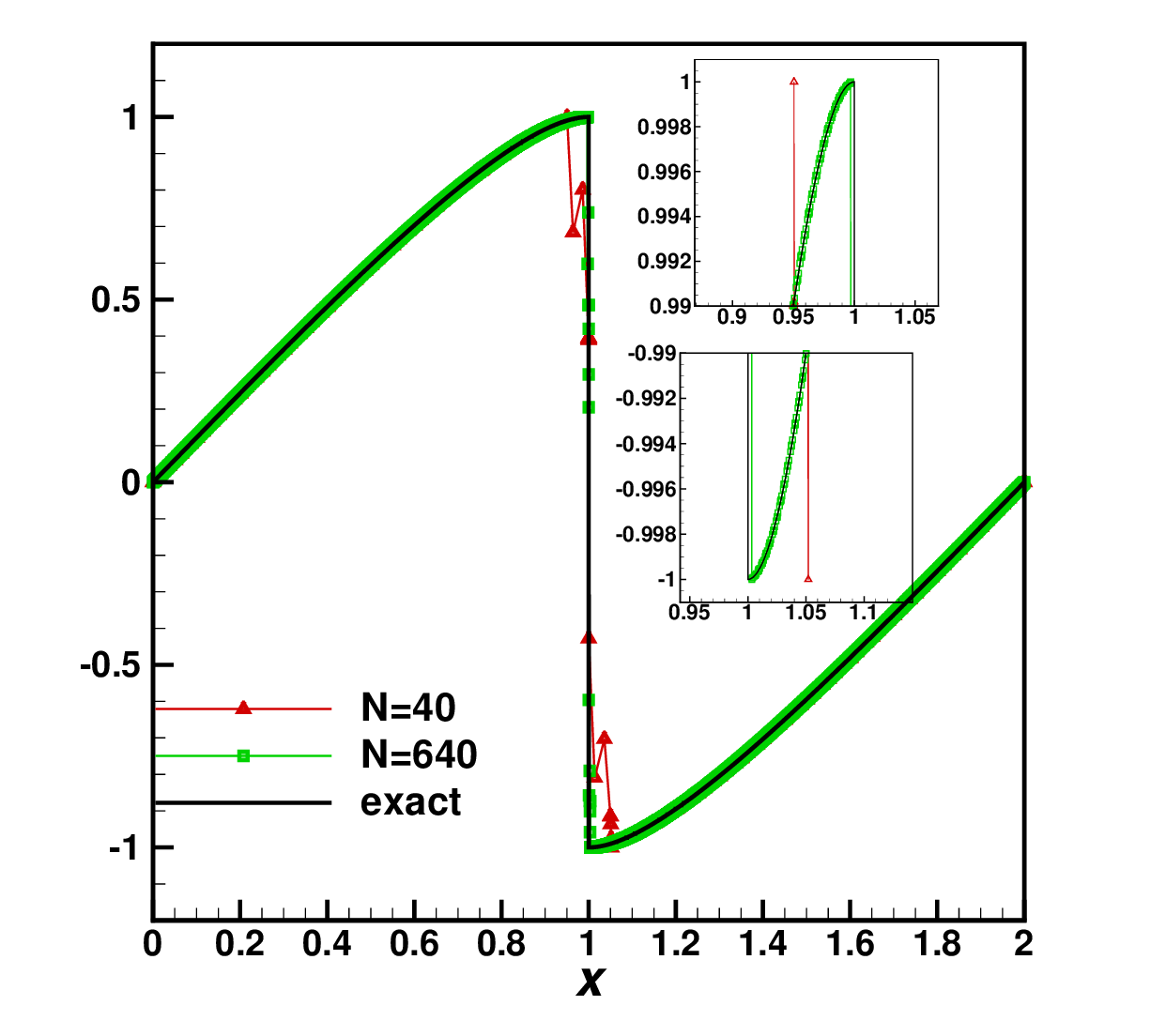}
   \includegraphics[width=1.8in]{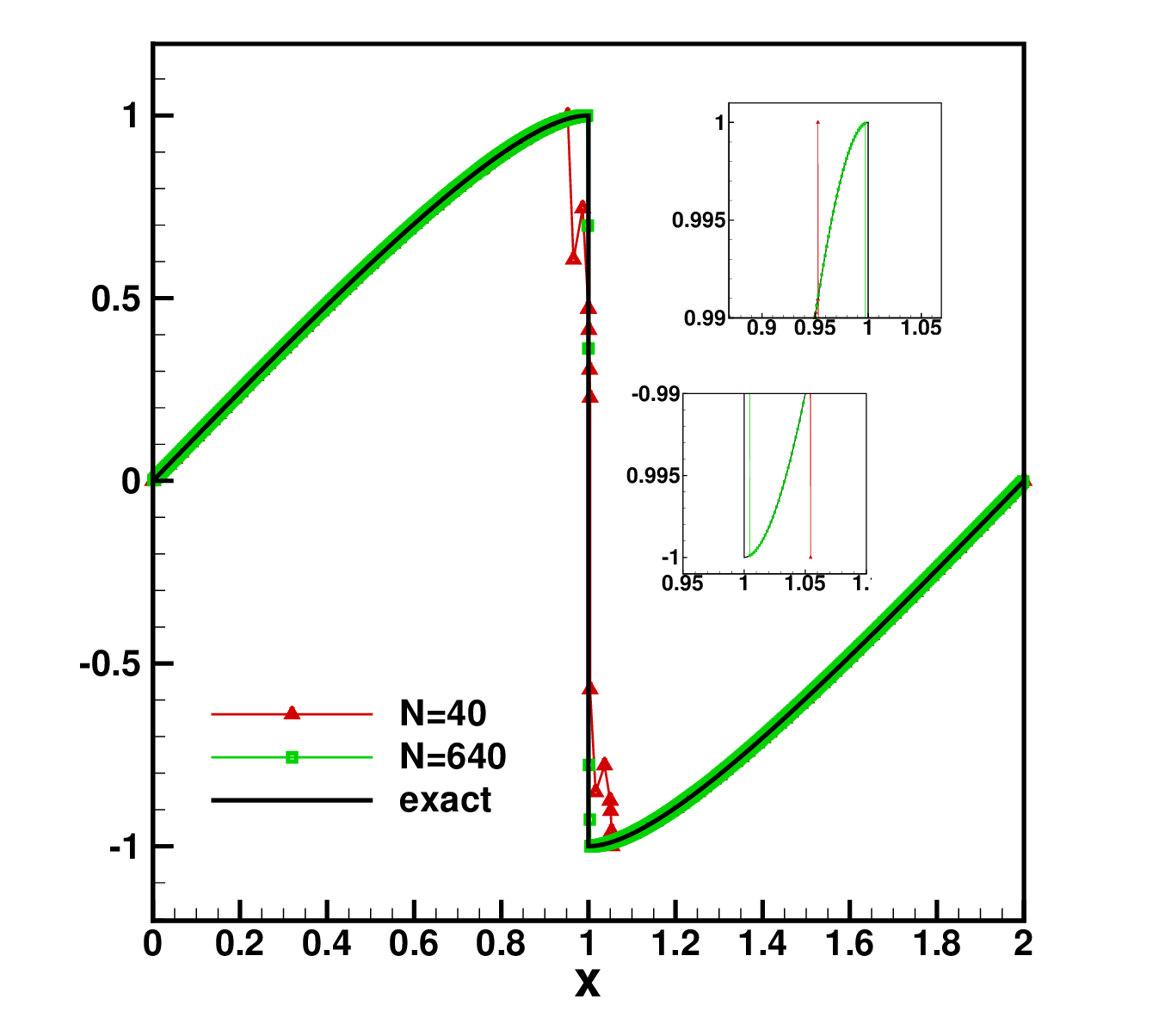}
  \caption{Numerical fourth order solutions  for the Burgers' equation  with smooth initial data at $t=0.5$.  Left: Basic scheme without reconstruction but with limiter. Middle: Proposed scheme including reconstruction and limiter \eqref{eq:limit:uj} applied on all  macro-elements. Right: Proposed scheme including reconstruction and limiter \eqref{eq:limit:uj} is applied only on macro-elements, where the numerical solution does not satisfy the maximum principle on the involved elements.}\label{figure:burgesshock2:onlymacro}
\end{figure}

\subsubsection{Riemann problems}
Consider Burgers' equation \eqref{eq:Burges} on $[-2,2]$ with initial data
\begin{align*}
u_0(x)=
\left\{\begin{array}{ll}
{u_l,} & {x\leq0,} \\
{u_r,} & {x>0.}
\end{array}\right.
\end{align*}
We will use our proposed bound preserving  Cut-DG scheme \eqref{scheme:cutDG2fully}-\eqref{eq:schemestep3} with $P^1,P^2,P^3$ approximations together with the bound preserving limiter \eqref{eq:limit:uj}. All elements in the interval $[-0.5,0.5]$ are cut. We solve with two sets of initial data, $u_l=-1<0<u_r=1$ and $u_l=1>0>u_r=-0.5$, respectively. 
Figure \ref{figure:Rarefaction:moveshock} shows the numerical solutions from the proposed bound preserving Cut-DG scheme based on $P^p, p=1,2,3$ approximations, together with third order RK method \eqref{eq:TVDRK3}  for time discretization. Our proposed scheme can simulate these two Riemann problems without violating
the maximum principle.  We note that violations of the maximum principle occur    near the shock in the case $u_l=1>0>u_r=-0.5$, when the bound preserving limiter is  applied to the basic  scheme \eqref{scheme:cutDG2fully} without reconstruction.
\begin{figure}[bthp]
  \centering
 \includegraphics[width=1.8in]{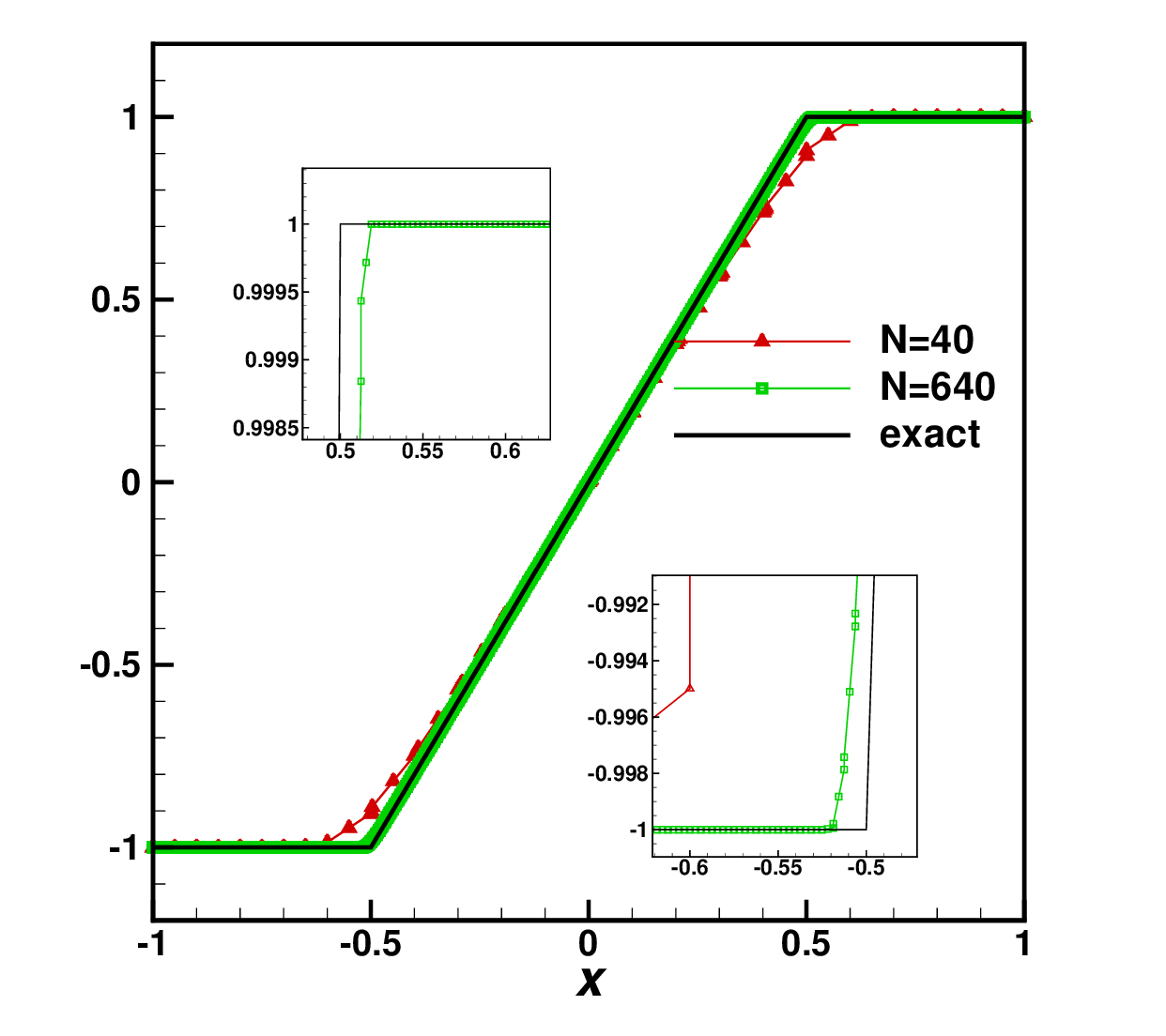}
\includegraphics[width=1.8in]{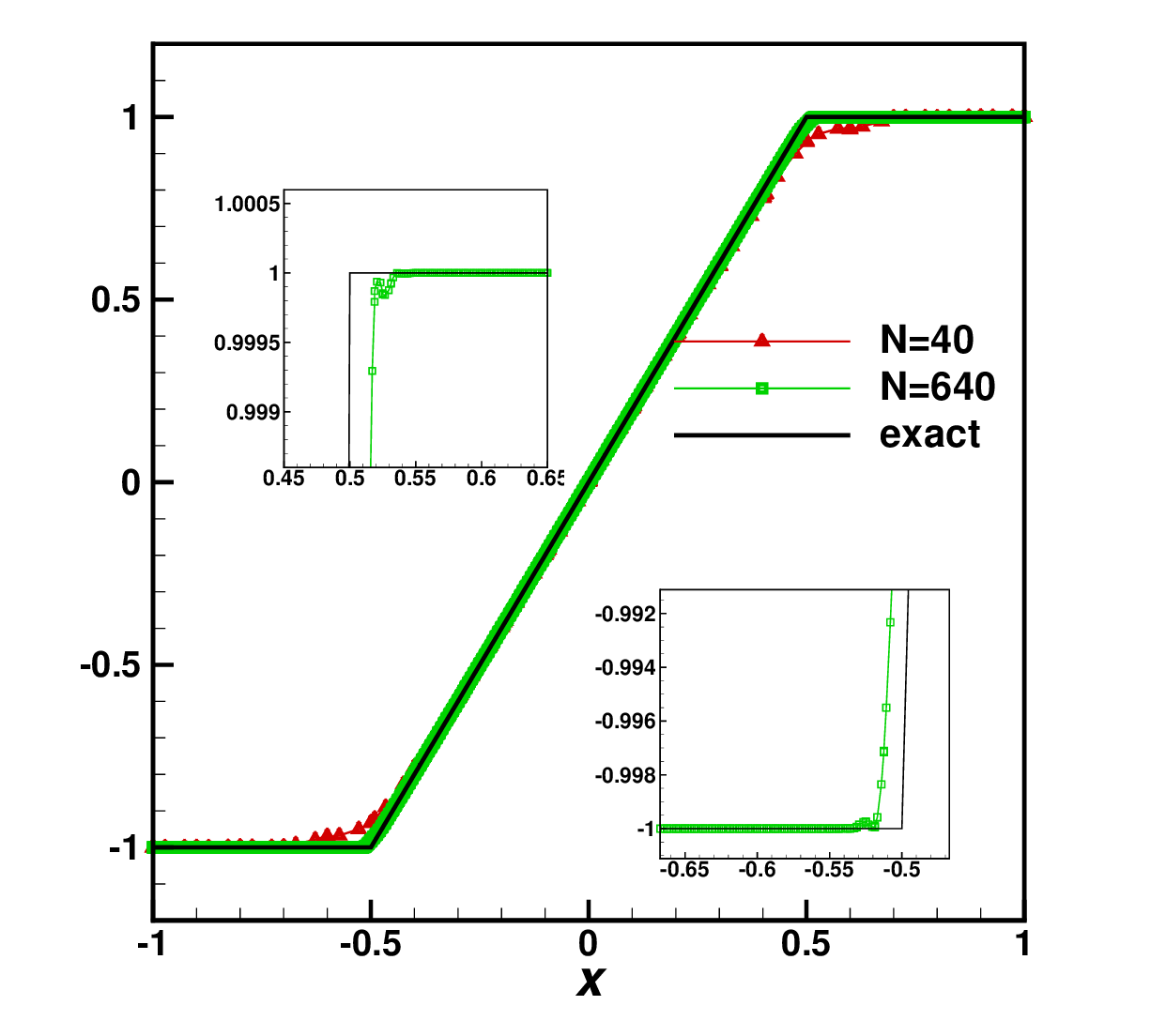}
\includegraphics[width=1.8in]{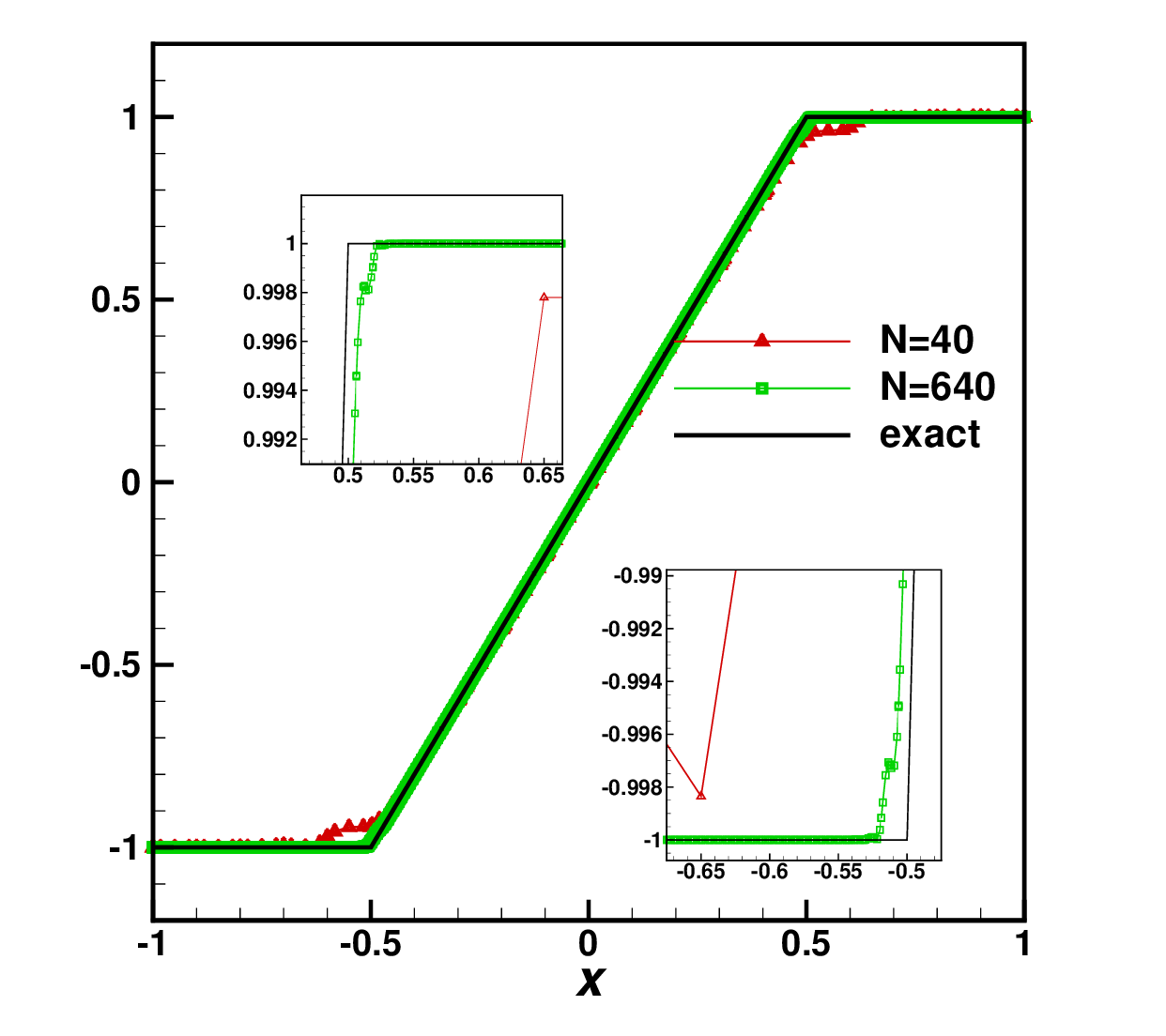}
\includegraphics[width=1.8in]{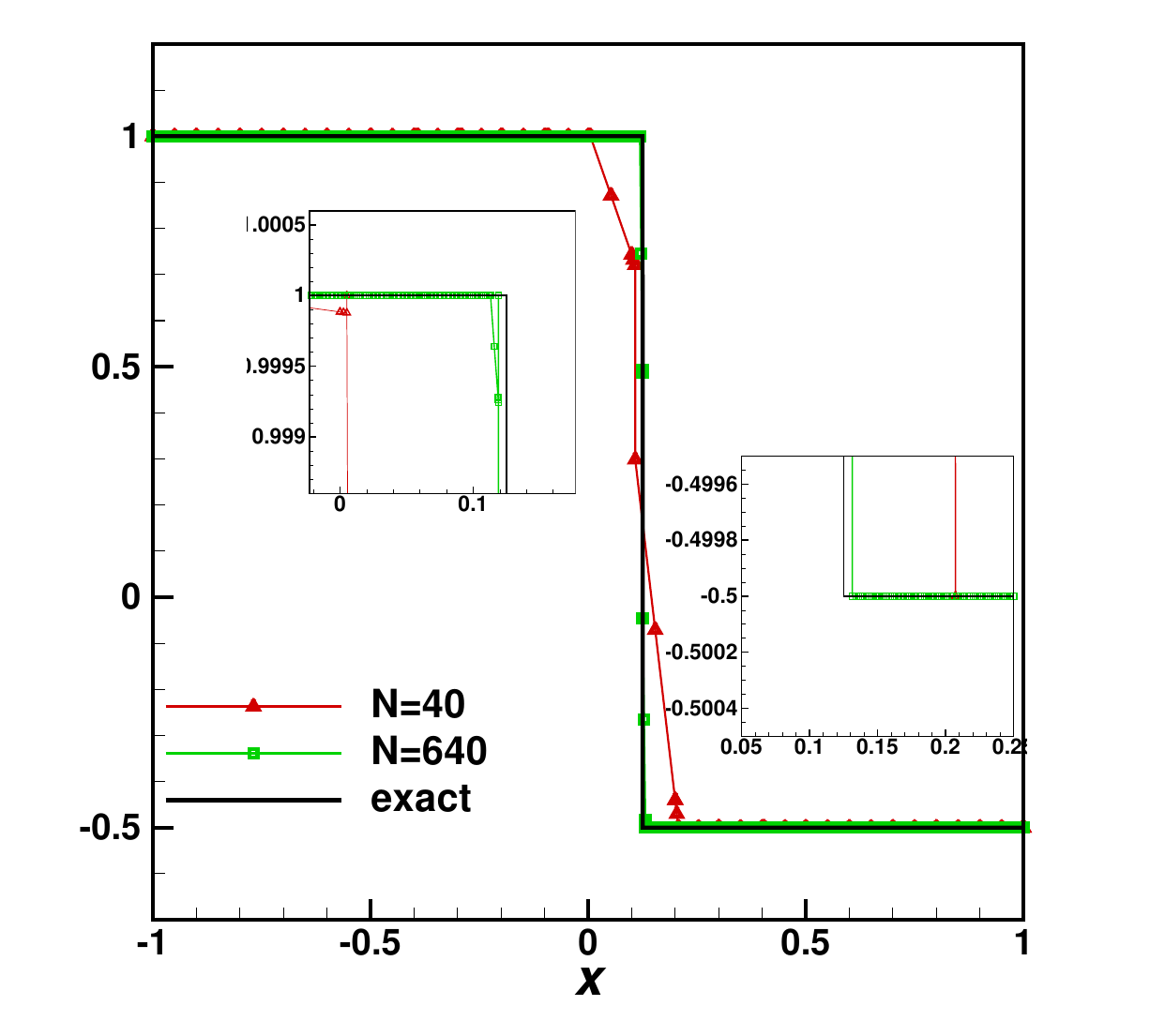}
\includegraphics[width=1.8in]{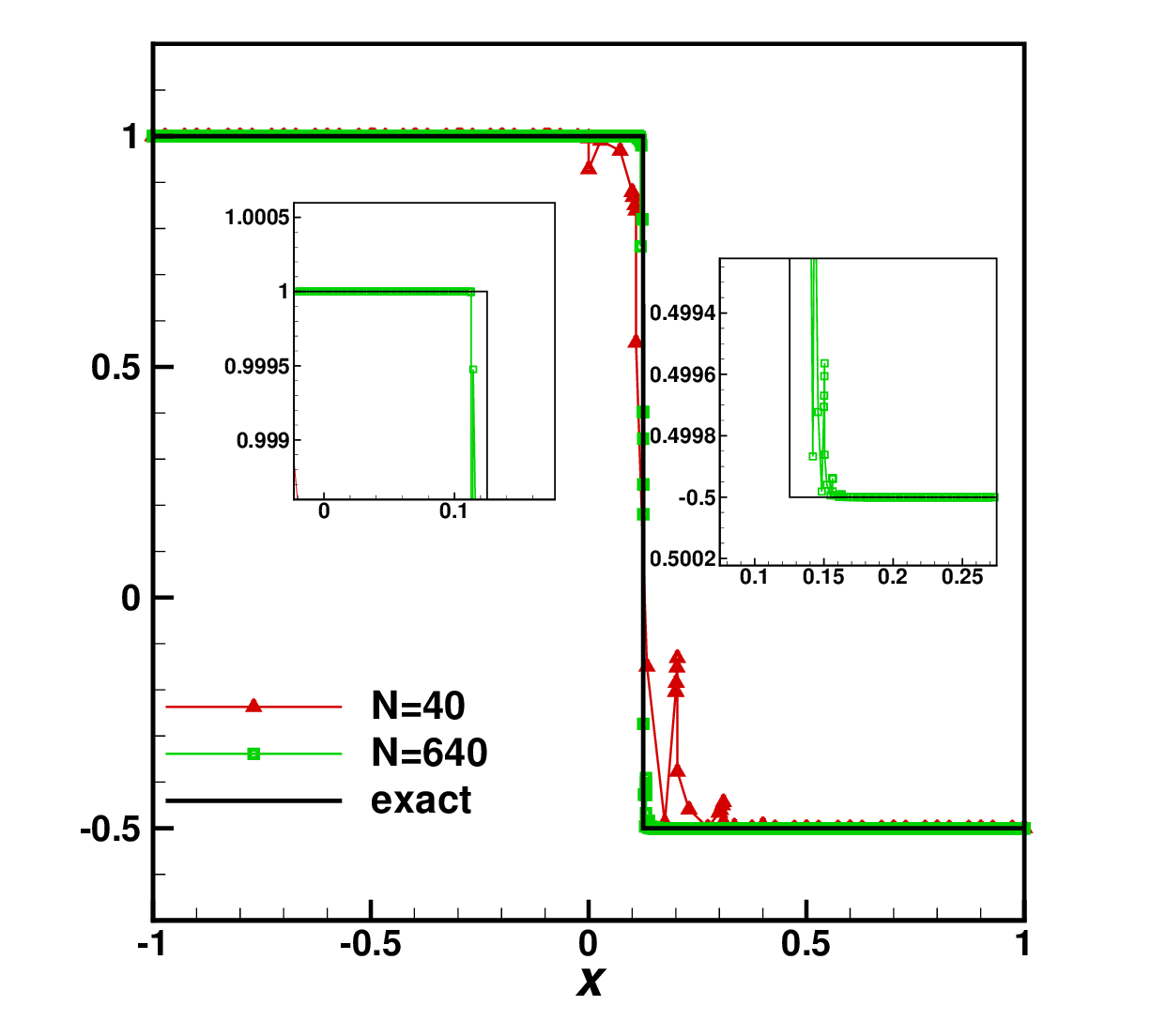}
\includegraphics[width=1.8in]{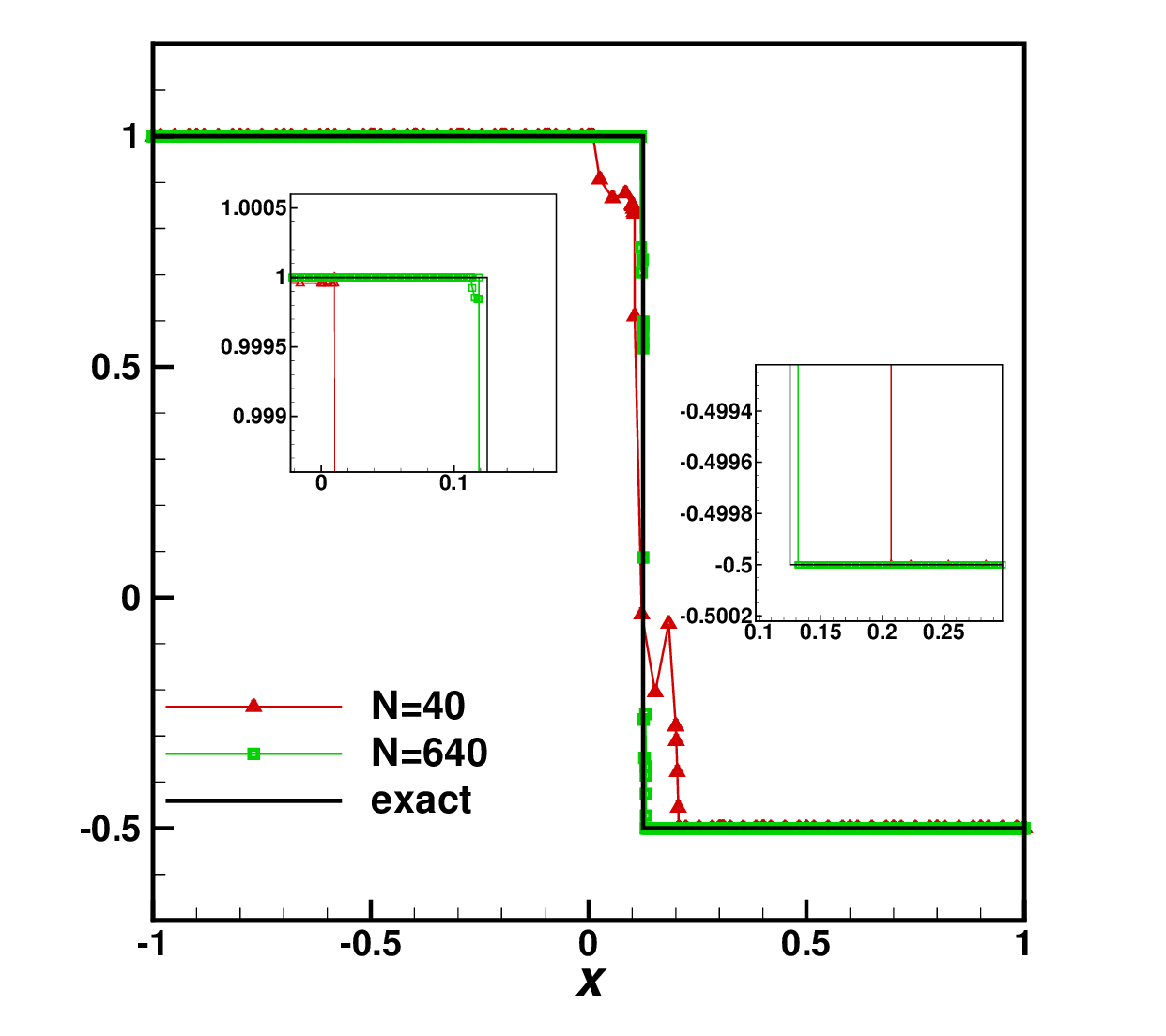}
  \caption{Numerical solutions of the Riemann problem for Burgers' equation at $t=0.5$  based on polynomial degree $p (p=1,2,3$ from left to right). Reconstruction on macro-elements and bound preserving limiter \eqref{eq:limit:uj} are used in the computations.  
  Top: a rarefaction wave. Bottom: a moving shock. }\label{figure:Rarefaction:moveshock}
\end{figure}

\subsection{Scalar case with discontinuous flux function}\label{sec:discontinuous flux}
In this subsection, we consider a hyperbolic conservation law with an interface, where  the flux function is discontinuous. Let 
\begin{align}\label{eq:def:discon:flux}
f(u)=H(x-x_\Gamma)h(u)+(1-H(x-x_\Gamma))g(u), \quad x\in[-1,1],
\end{align}
where $H(x)$ is the Heaviside function and $g(u)=u^2/2,h(u)=u$.  In the example, which is taken  from  \cite{badwaik2020convergence},  the transport equation switches to the Burgers' equation across the interface $x=x_\Gamma$. Here, we take $x_\Gamma=2 \cdot 10^{-5}$ and use a uniform mesh with $N$ elements and  $h=|\Omega|/N$. Thus, there is  a small cut cell with size $\alpha h$ with $ \alpha=N\cdot 10^{-5}$ in the interior of the  domain.  The initial condition is taken as 
\begin{align}
u(x,0)=
\begin{cases}
     0.5 & \text{$x<-0.5$ }, \\
     2 & \text{otherwise}.
\end{cases}
\end{align}
This condition is chosen such that the Rankine-Hugoniot condition at the interface $x_\Gamma$ is  $[u]_{\Gamma}=0$ before the discontinuity interacts with the interface, and $[f(u)]_\Gamma=0$ afterwards. We use our proposed bound-preserving CutDG scheme with the Lax-Friedrichs flux on the interior element edges and the upwind flux on the interface $x_\Gamma$. Thus, the conservation is satisfied.  For how to guarantee conservation at the interface we refer  to our previous work \cite{fu2022high} (here we choose $\lambda_1=0,\lambda_2=-1$).  \se{This interface treatment means that information is passed only from the left region to the right region, and that the flux at the interface is $\hat f_\Gamma=u_L$, where $u_L$ is the solution directly to the left of $x_\Gamma$. Since the solution of the  advection equation to the left is  independent of the right region, it satisfies a maximum principle with upper $\MM=2$ and lower bounds $\mm=0.5$. The first macro element to the right of the interface needs special attention. There the flux at the right end point is the Lax-Friedrichs flux based on $f(u)=u^2 /2$, while the left flux is $\hat f_\Gamma$. Analysis similar to that in $Section 3$ shows that the value in this element will satisfy the same maximum principle as to the left, since $\mm-\mm^2/2\geq 0$ and $\MM-\MM^2/2\leq 0$. For other $\mm$ and $\MM$ values in the left region the interval that bounds the solution in the right region may need to be extended. }
In Figure \ref{fig:interfacecase},  the results from a third order scheme are shown at different time instances. In this case, the discontinuity is transported to the right until it reaches the interface. When the discontinuity enters the region where  Burgers' equation holds, a rarefaction wave forms. With the proposed scheme, the solution is captured very well. Comparing the results ($h=2/64$) in \cite{badwaik2020convergence}, our results are more accurate on the coarser mesh (N=40, $h=1/20$)  and the numerical solution is both conservative and satisfies the maximum principle.  We also applied second order and fourth order schemes to solve this problem and the results are very similar. 
\begin{figure}[!htbp]
\centering
\includegraphics[width=2.0in]{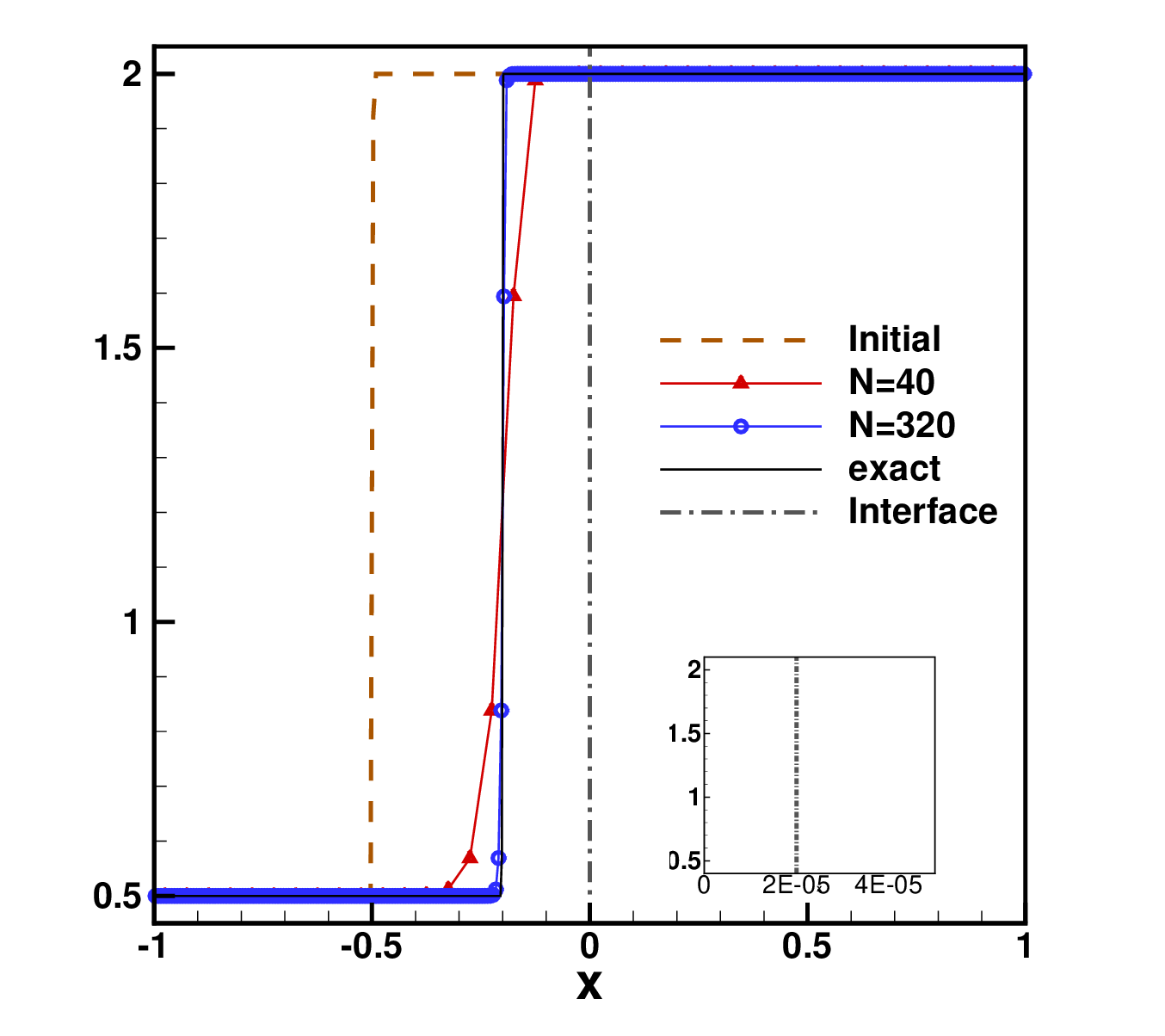}
\includegraphics[width=2.0in]{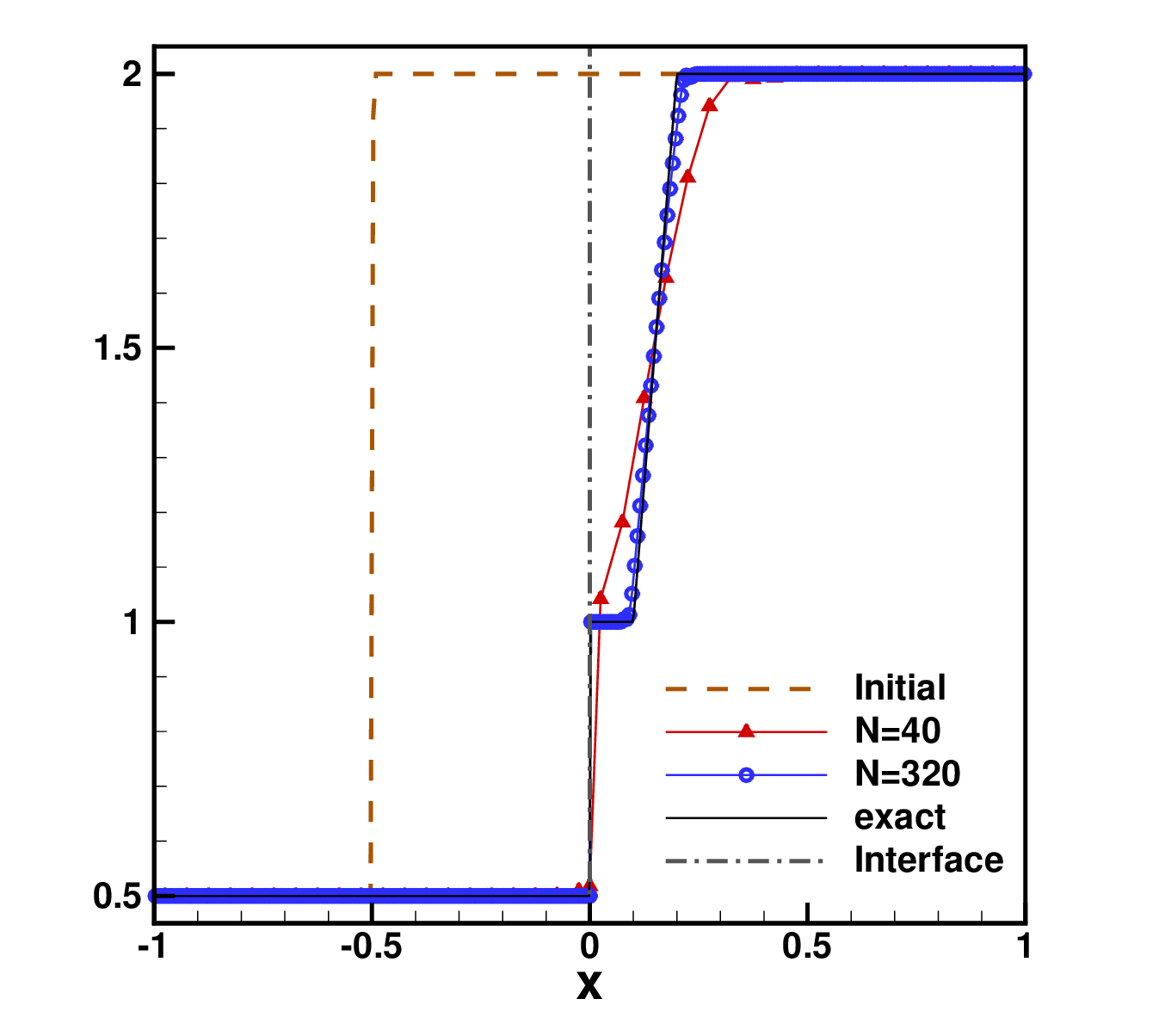}
\includegraphics[width=2.0in]{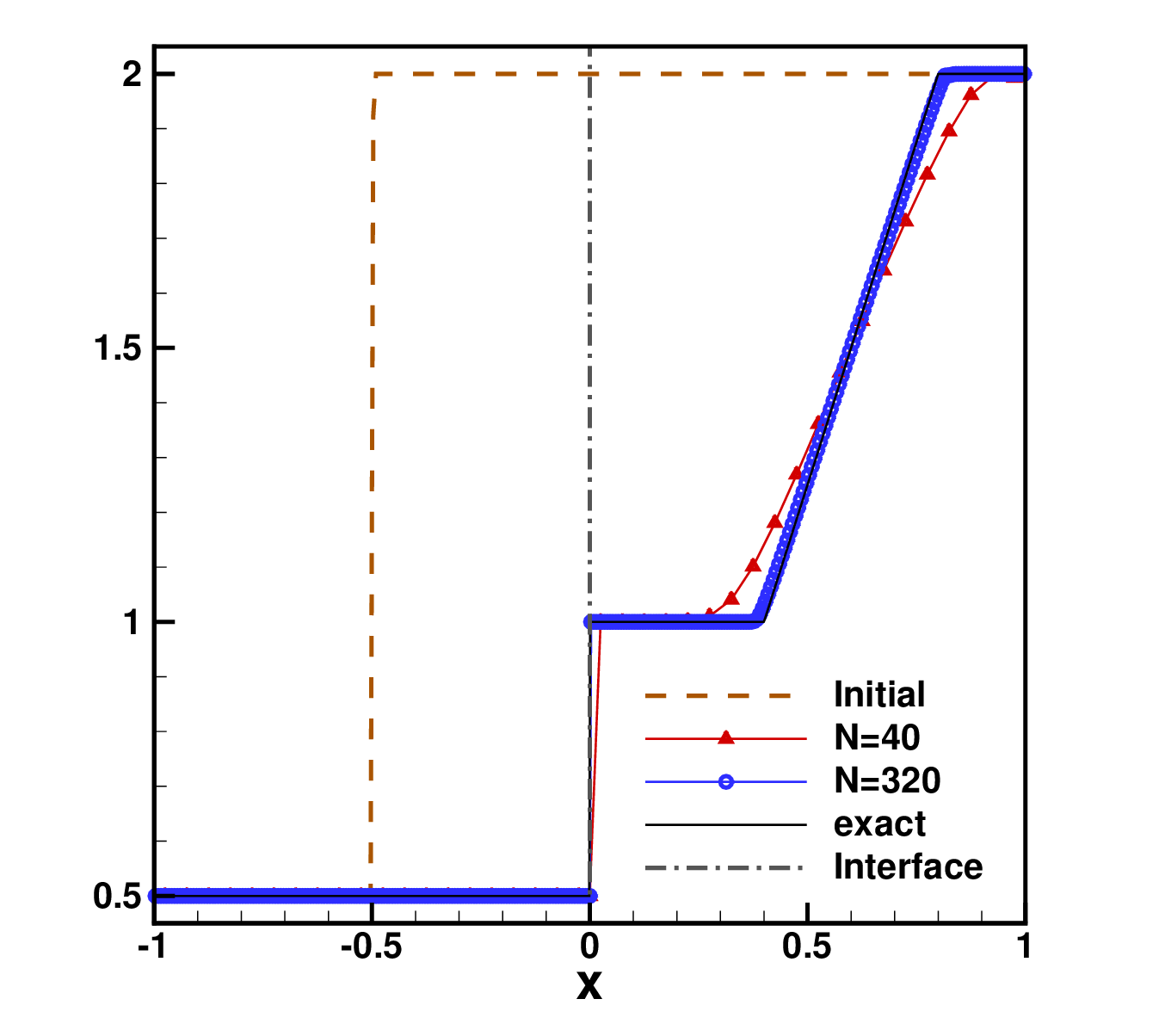}
\caption{A third order  approximation of a  scalar hyperbolic conservation law with the discontinuous flux function \eqref{eq:def:discon:flux} at different time instances.  Left: $t=0.3$, middle $t=0.6$, right $t=0.9$.}\label{fig:interfacecase}
\end{figure}

\subsection{The Euler equations}
In this subsection, we use our proposed Cut-DG method \eqref{scheme:cutDG2fully}-\eqref{eq:schemestep3}  with different piecewise  polynomials, together with the positivity preserving limiter in Section 4.3 for Euler equations \eqref{euler:eq}.  Recall that each cut element is split into two parts of length $\alpha_k$ and $1-\alpha_k$, with $\alpha_k$ defined  in \eqref{eq:define:alphak} and $\alpha=0.01$.  In the examples we consider $\gamma=1.4$, which corresponds to an ideal gas. 

\subsubsection{Accuracy test: smooth problem}
First, we consider a low density problem in one dimension. We take the initial condition as
\begin{align}
\rho_0(x)=1+0.99\sin(x),\quad u_0=1,\;p_0=1.
\end{align}
The domain is taken to be $[0,2\pi]$. The exact solution of this problem is
\begin{align}
\rho_0(x,t)=1+0.99\sin(x-t),\quad u(x,t)=1,\;p(x,t)=1.
\end{align}
Without the positivity-preserving limiter the numerical scheme may be unstable due to the negative density in long time simulation. In Figure  \ref{fig:euler:accuracy}, we show the $L^2$- and $L^\infty$-errors from our proposed bound-preserving Cut-DG methods at $t=1$.  We observe the optimal  $p+1$-th order of  accuracy for approximations of polynomial  order  $p=1,2,3$. 
\begin{figure}[tbhp]
  \centering
\includegraphics[width=2.8in]{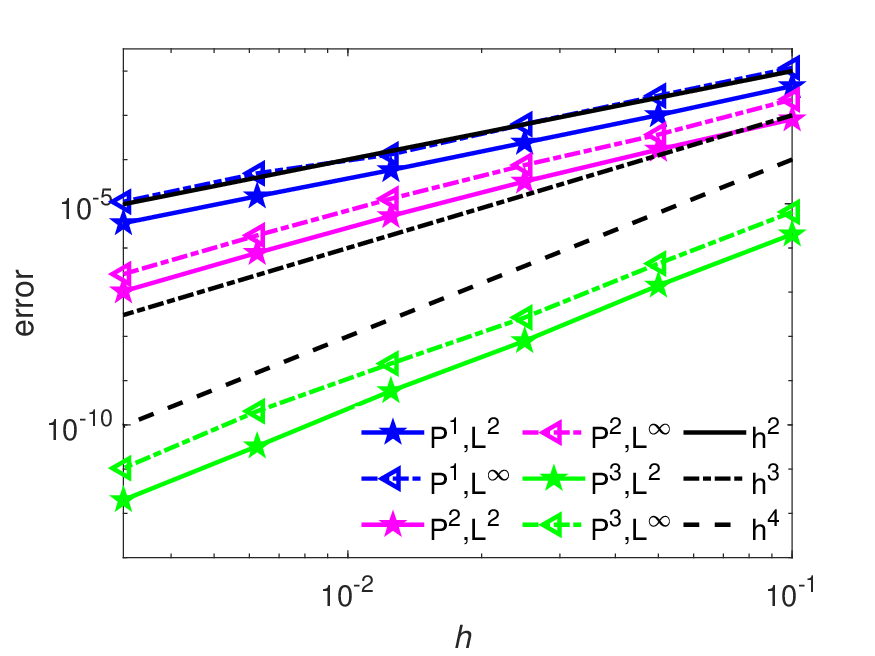}
\caption{The $L^2$- and $L^\infty$-errors for the numerical solutions of the low density problem in Section 5.3.1 from our proposed bound-preserving  Cut-DG method.
 }\label{fig:euler:accuracy}
\end{figure}

\subsubsection{Riemann problems}
In this subsection, we consider two Riemann problems. We use the proposed bound preserving  Cut-DG method \eqref{scheme:cutDG2fully}-\eqref{eq:schemestep3}  with spaces $P^2,P^3$.

First, we consider the Sod shock tube problem in the domain $[0,1]$.   The initial condition is taken as 
\begin{align*}
(\rho_L, u_L, p_L) = (1,0,1), \, x<0.5, \quad (\rho_R, u_R, p_R) = (0.125,0,0.1), \, x>0.5.
\end{align*}
 We solve this problem up to $t=0.2$. 
The exact solution of Sod shock tube problem contains a rarefaction wave, a contact discontinuity and a shock discontinuity.  Figure \ref{fig:euler:sod} shows the  approximations of density (left), velocity (middle) and pressure (right) from our  proposed bound-preserving  Cut-DG scheme.  In this case,  TVD limiting  is needed to control oscillations. This limiter is applied  to the reconstructed solutions on macro-elements before the positivity-preserving limiter is applied.  
We observe that our proposed Cut-DG scheme can simulate this Riemann problem well and can capture the rarefaction wave and discontinuities.  
Just as for the standard DG method \cite{ShuDG3} on a fitted mesh, when the TVD limiter is applied to the  variables component-wise, we see some numerical artifacts in form of oscillations, which decrease with mesh refinement.
\begin{figure}[!htbp]
\centering
\includegraphics[width=2.1in]
{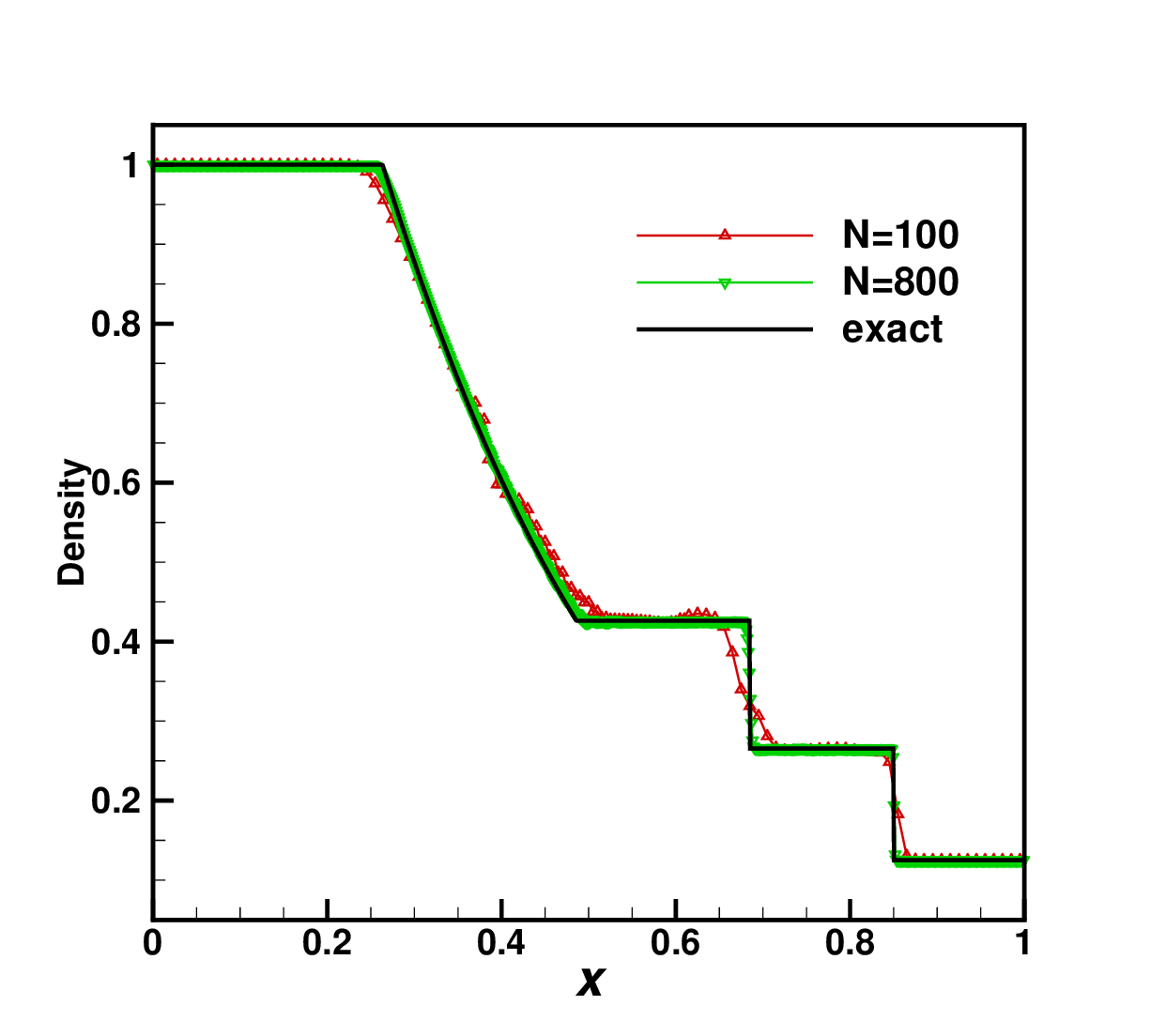}
\includegraphics[width=2.1in]
{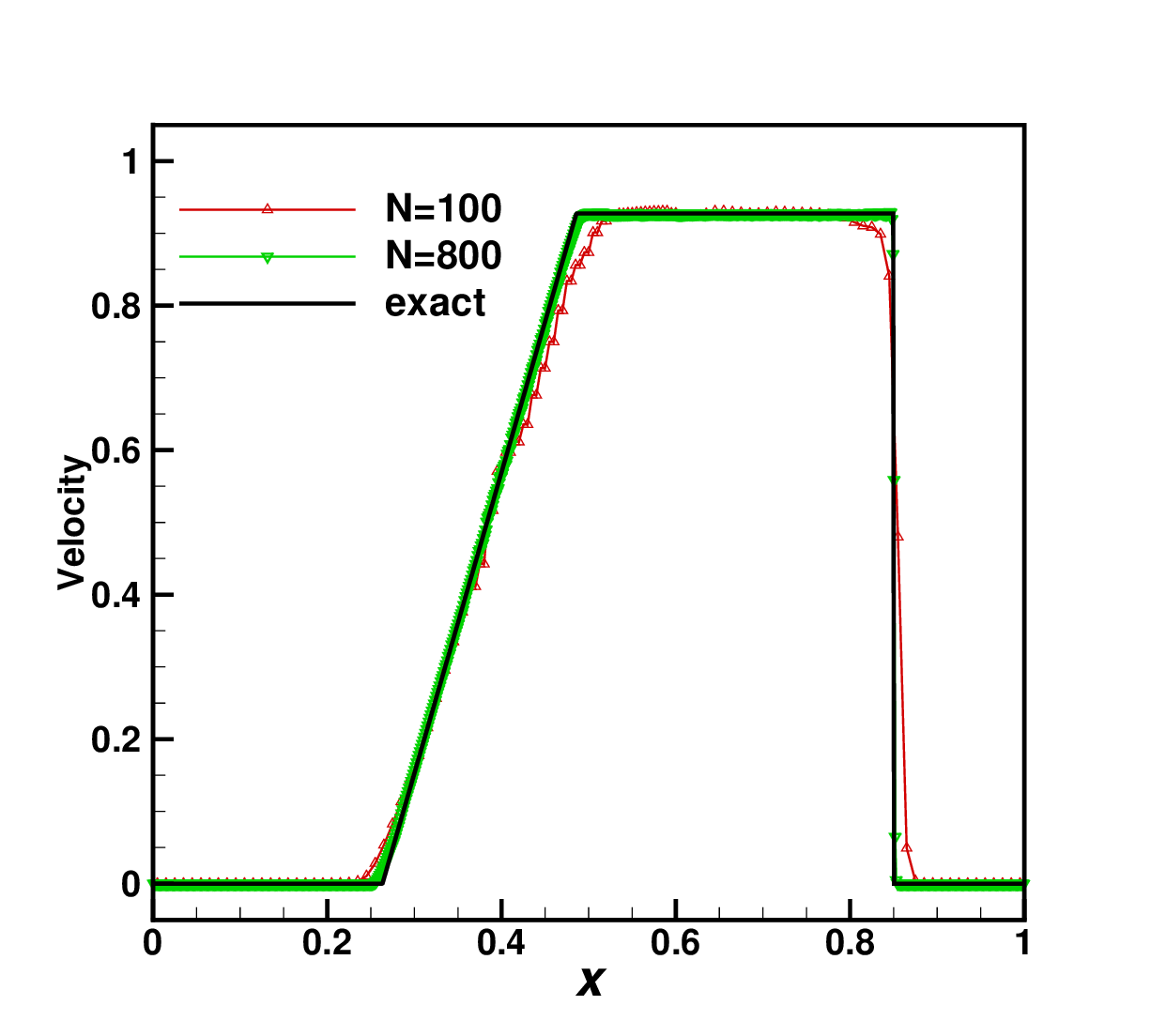}
\includegraphics[width=2.1in]{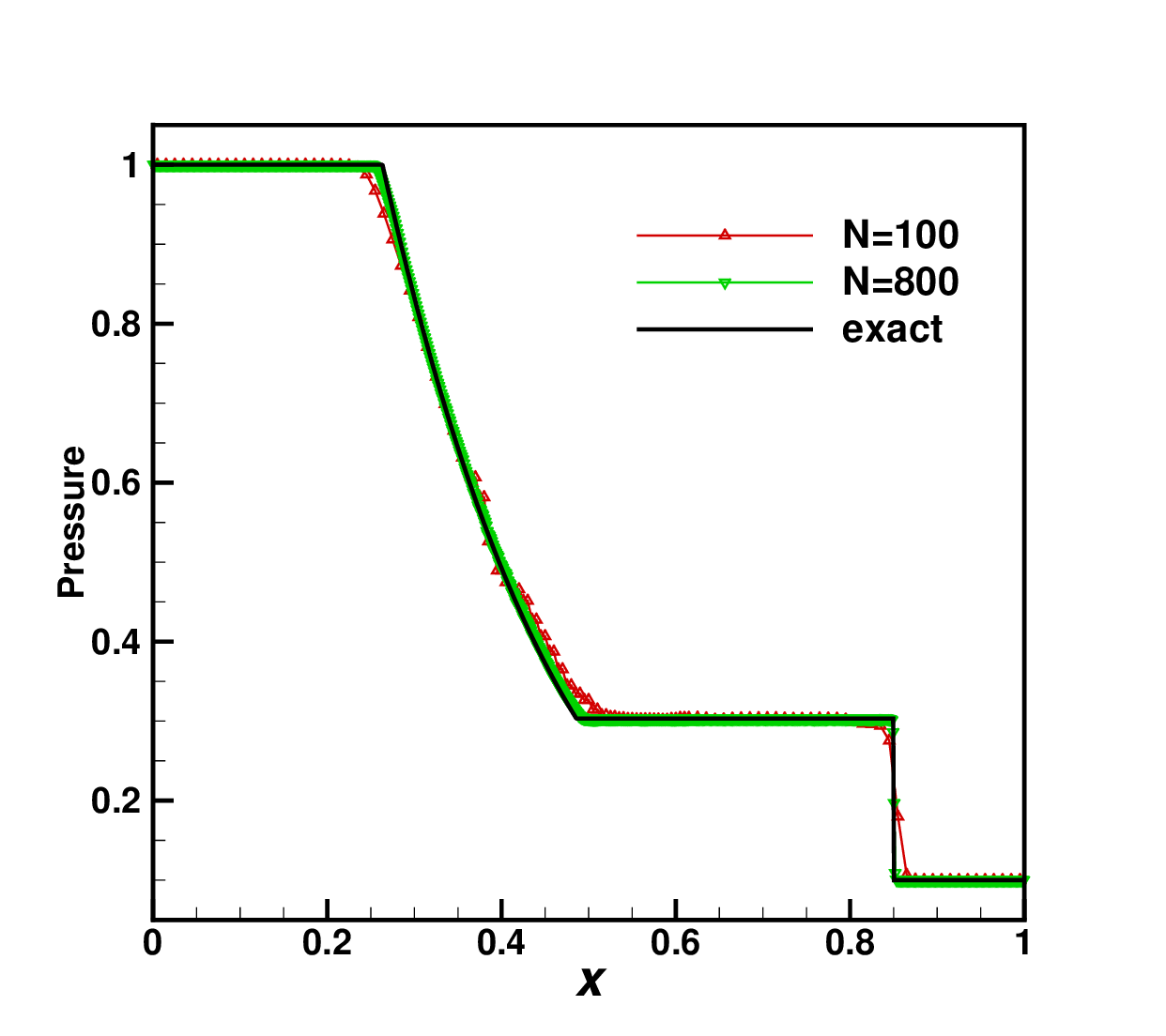}
\includegraphics[width=2.1in]{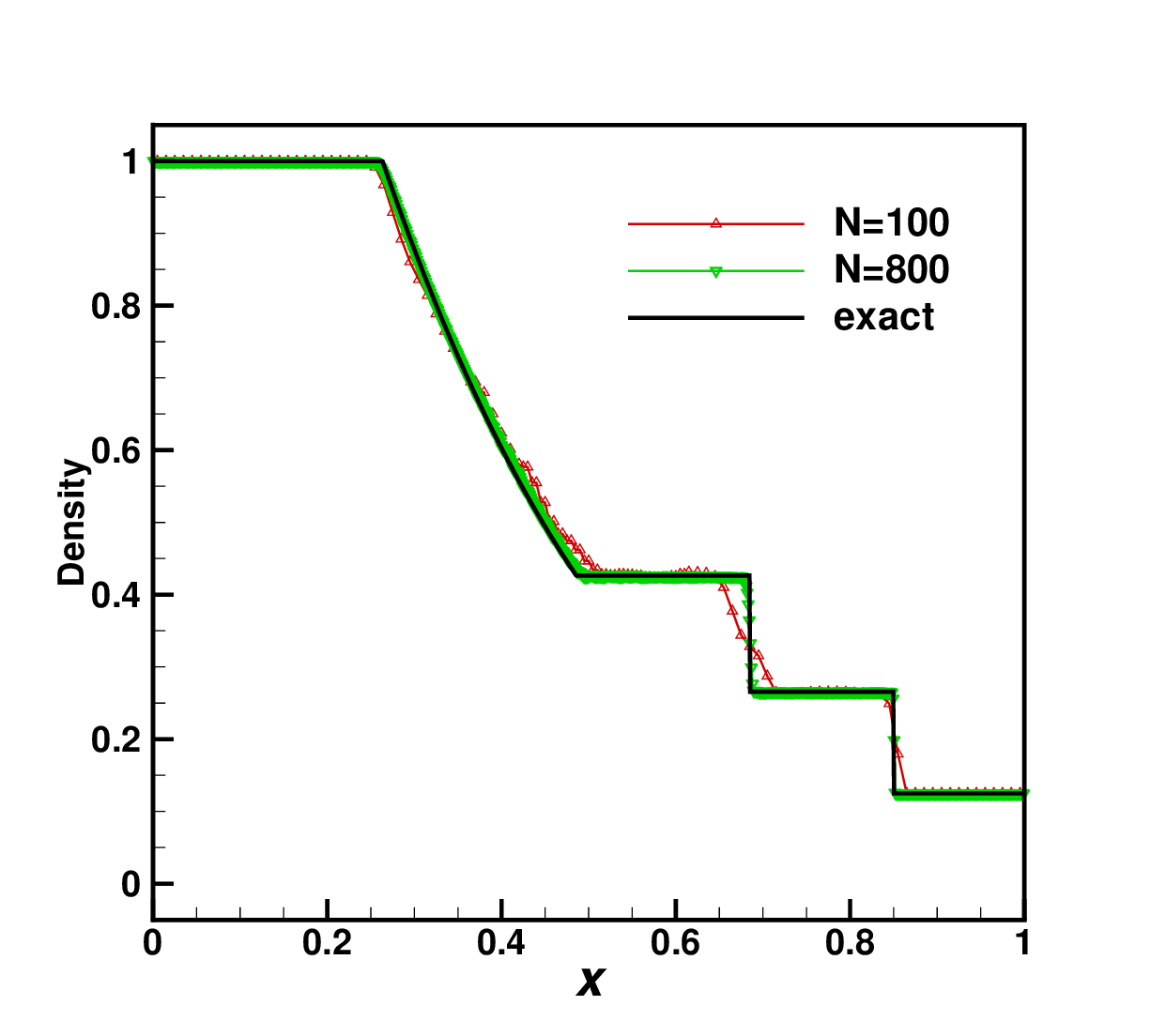}
\includegraphics[width=2.1in]{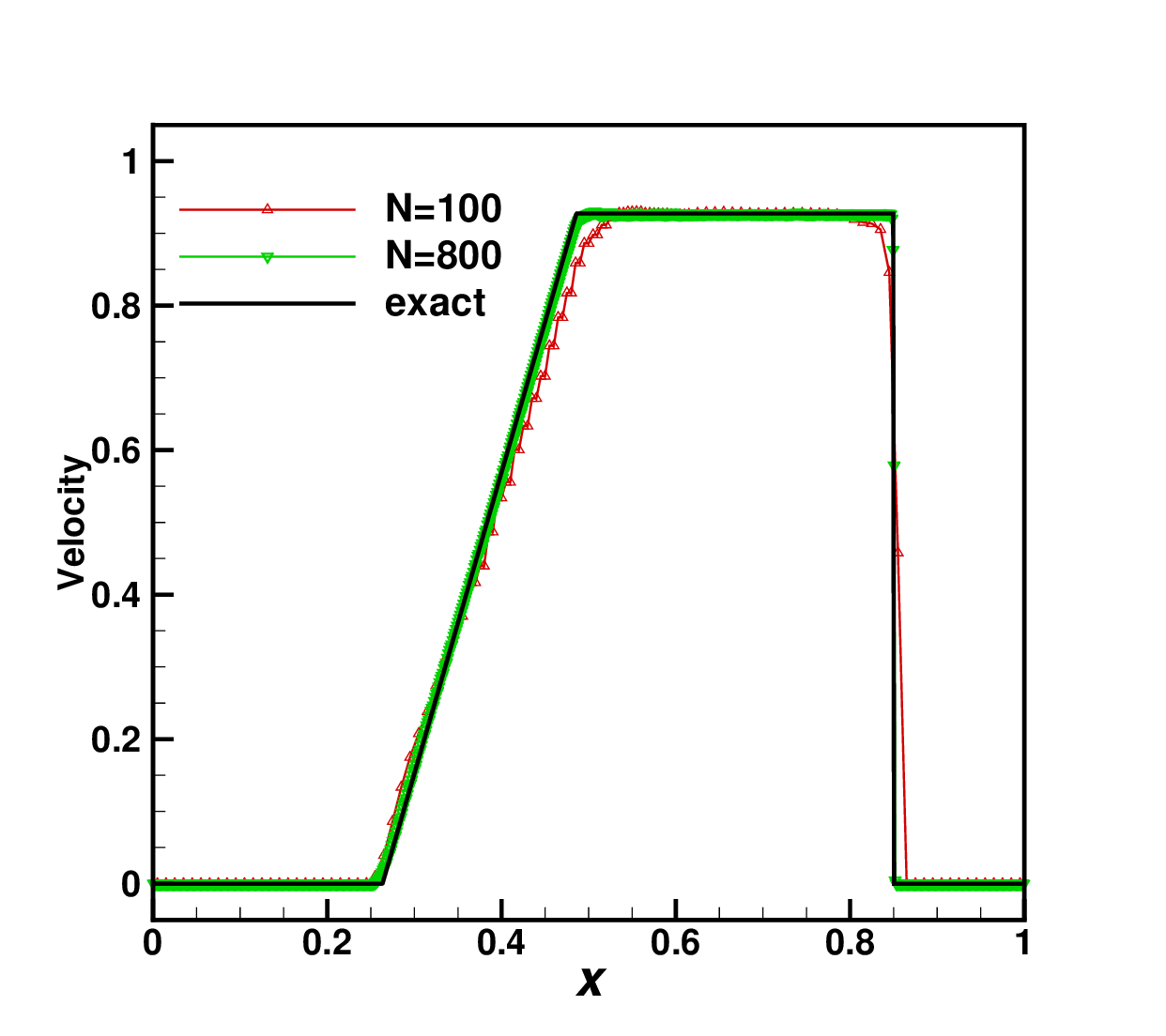}
\includegraphics[width=2.1in]{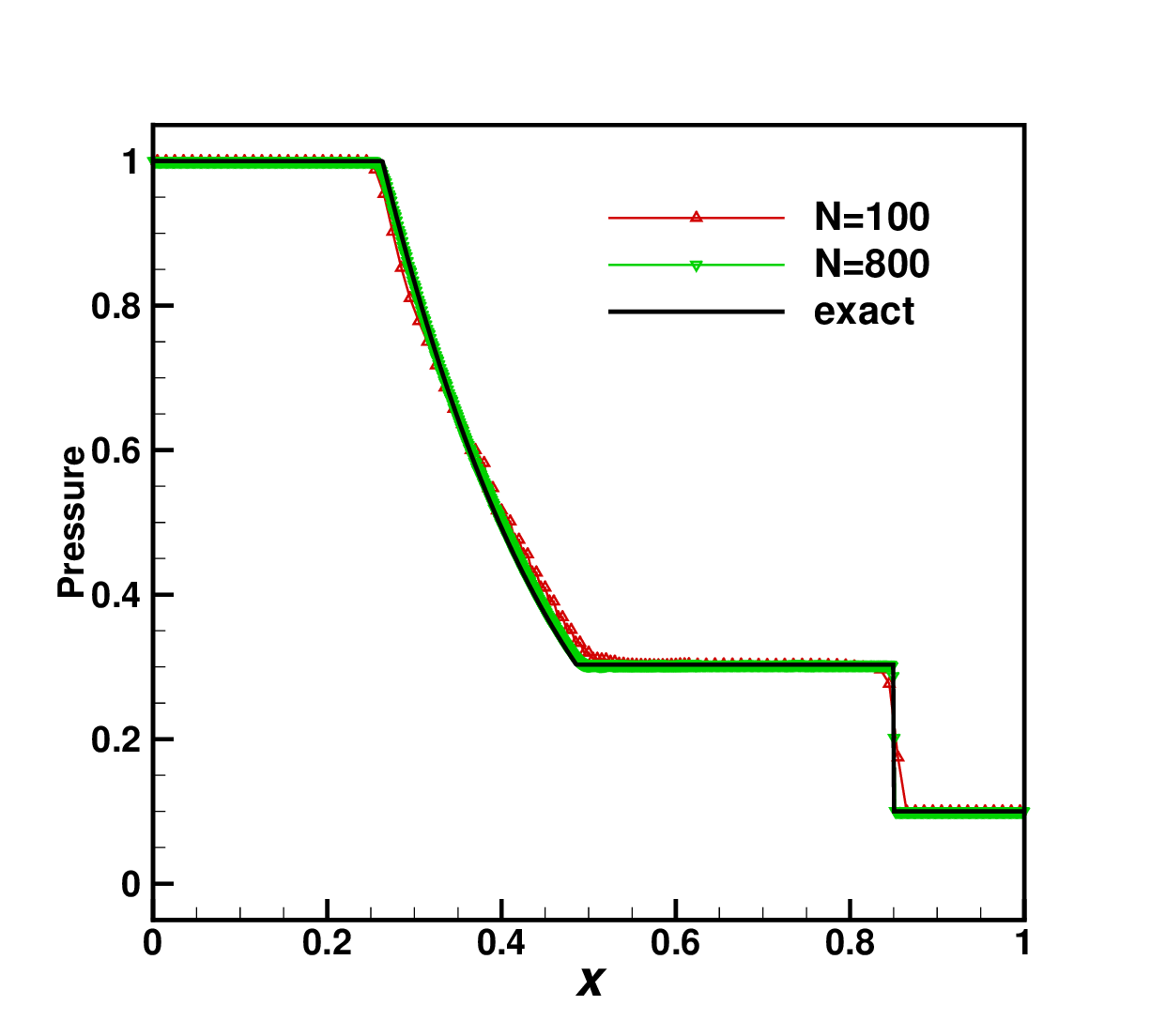}
\caption{ The numerical solutions from the proposed Cut-DG scheme for Sod shock tube problem on cut meshes. Top: $P^2$ approximations. Bottom:  $P^3$ approximations.}\label{fig:euler:sod}
\end{figure}

The second Riemann problem is a double rarefaction wave  in the domain $[-1,1]$ as in \cite{linde1997robust}. The initial condition is
\begin{align*}
(\rho_L, u_L, p_L) = (7,-1,0.2), \ x<0, \quad (\rho_R, u_R, p_R) = (7,1,0.2), \ x>0,
\end{align*}
and outflow boundary conditions are used at both boundaries.   There is no shock in this problem. Thus, only the positivity-preserving limiter is used for this problem. Notice that the exact solution  contains almost a vacuum. In the middle and right panels  of Figure \ref{exe:RP1:1d} we show the results from our proposed Cut-DG method for Euler equations.   The termination time is $t= 0.6$. In the left panel of Figure \ref{exe:RP1:1d}, we give the approximations of  $\rho,u,p$ from the standard DG method on  uniform meshes with $800$ elements. We  observe that the results on  cut meshes are practically similar to those on uniform meshes. For this Riemann problem computations are unstable without the  positivity preserving limiter. 
\begin{figure}[!htbp]
\centering
\includegraphics[width=2.1in]{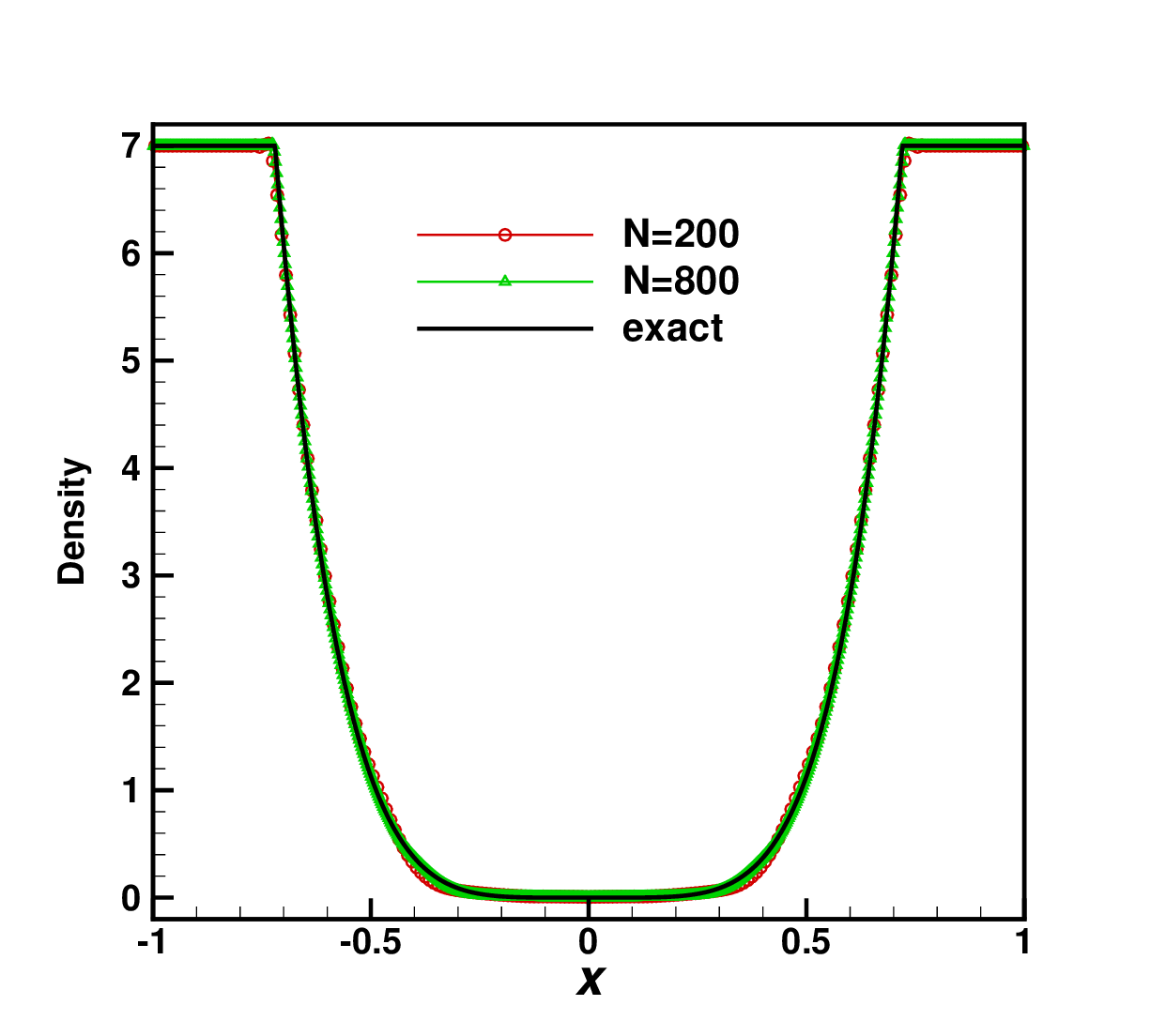}
\includegraphics[width=2.1in]{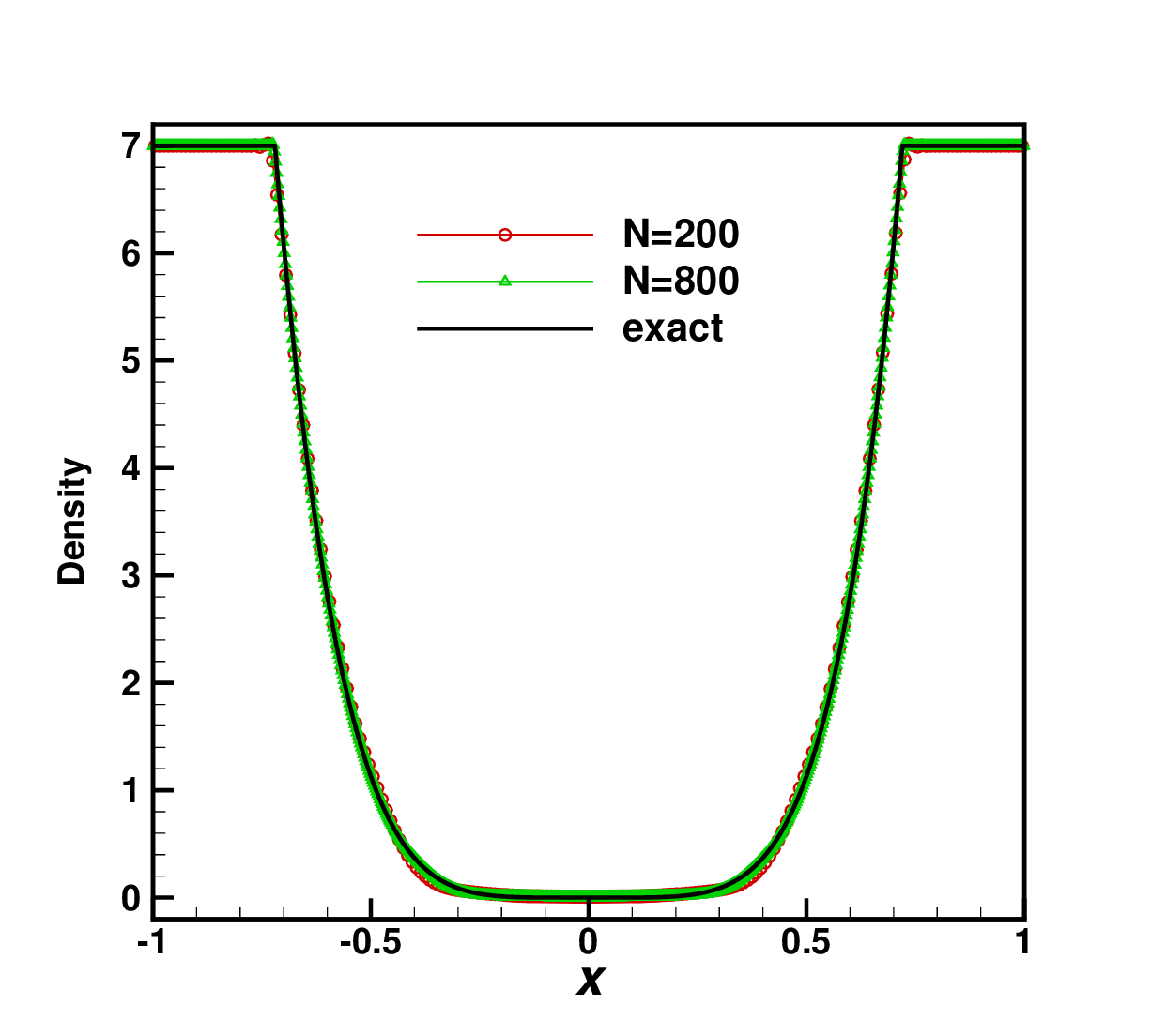}
\includegraphics[width=2.1in]{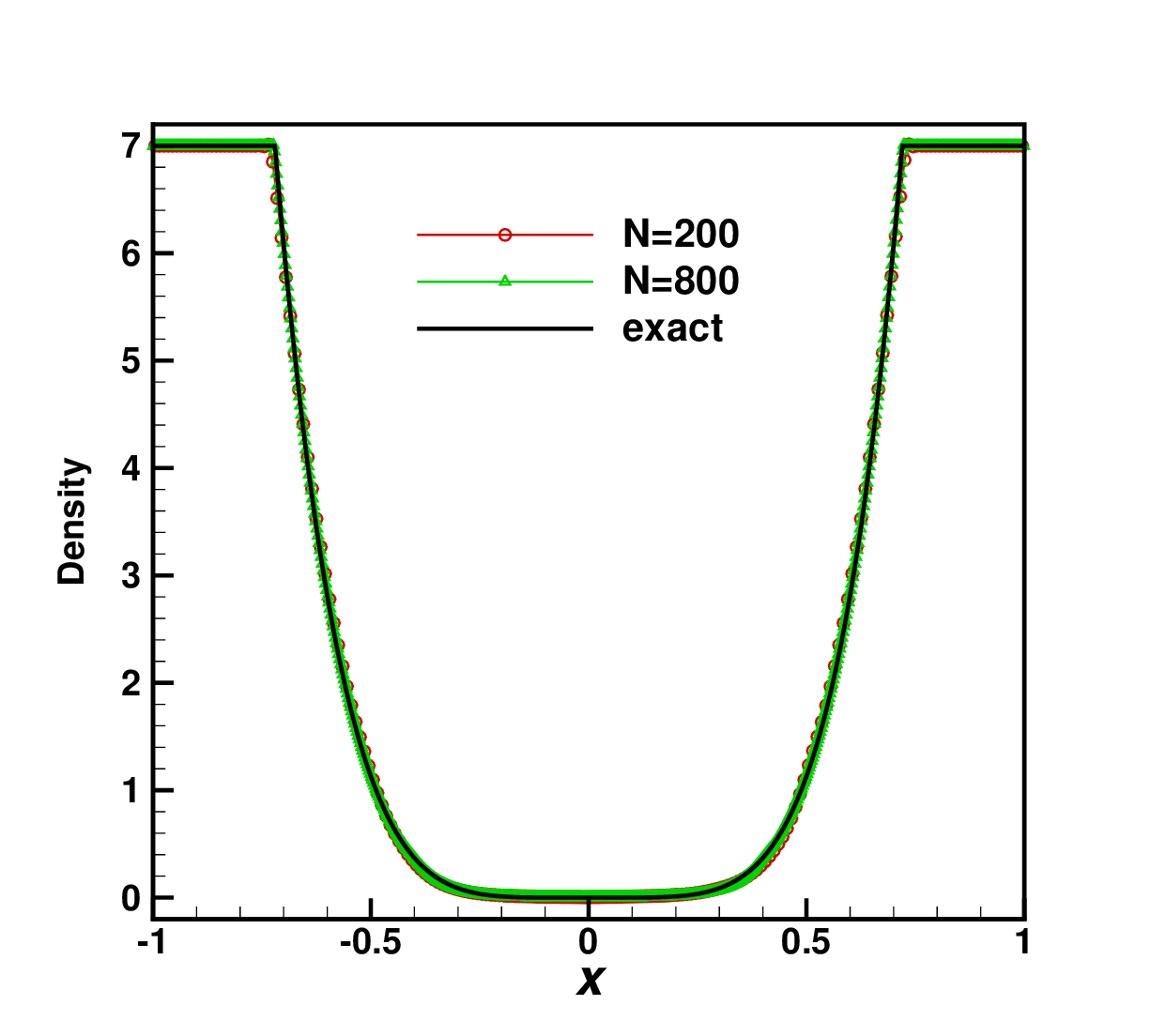}
\includegraphics[width=2.1in]{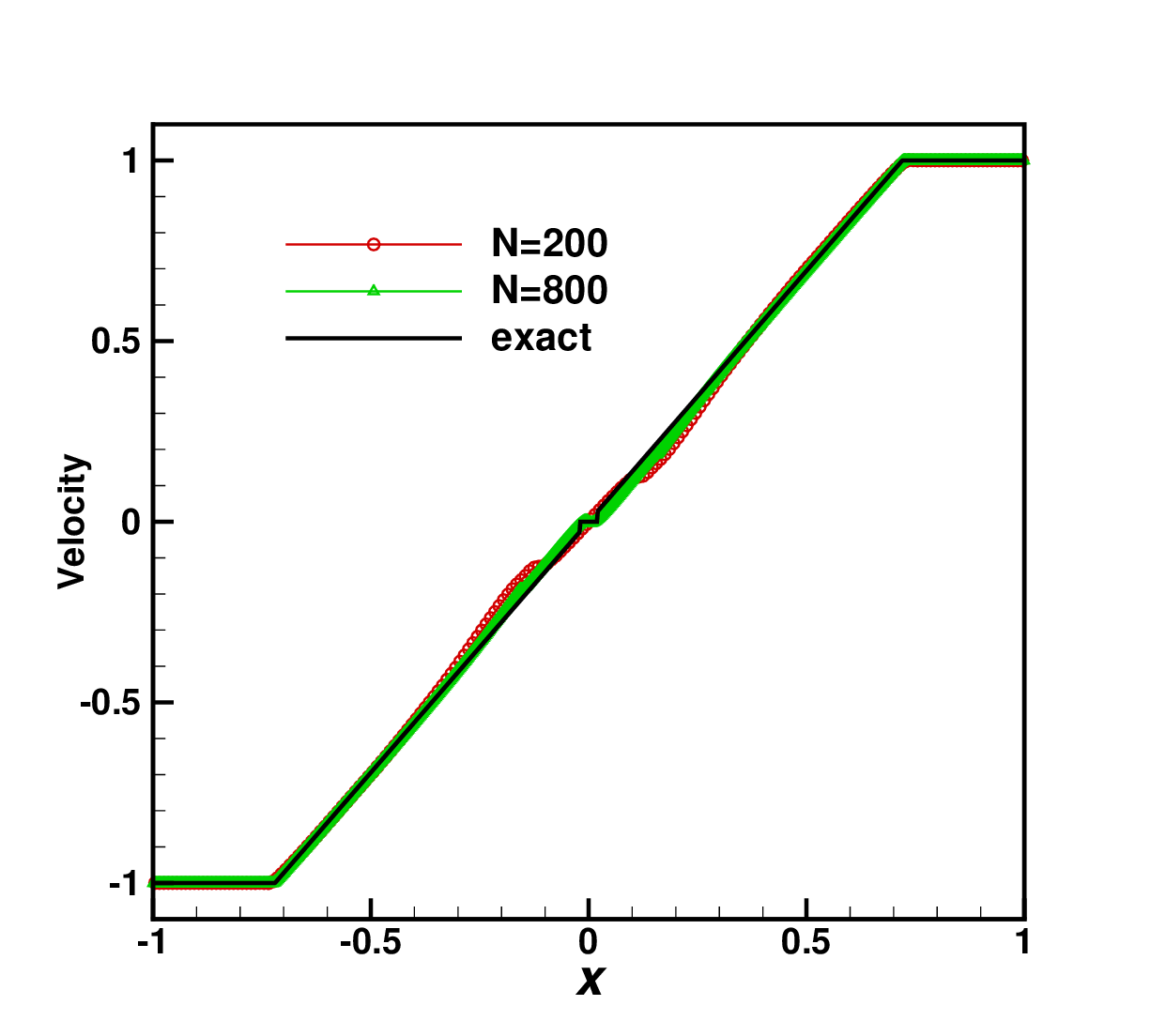}
\includegraphics[width=2.1in]{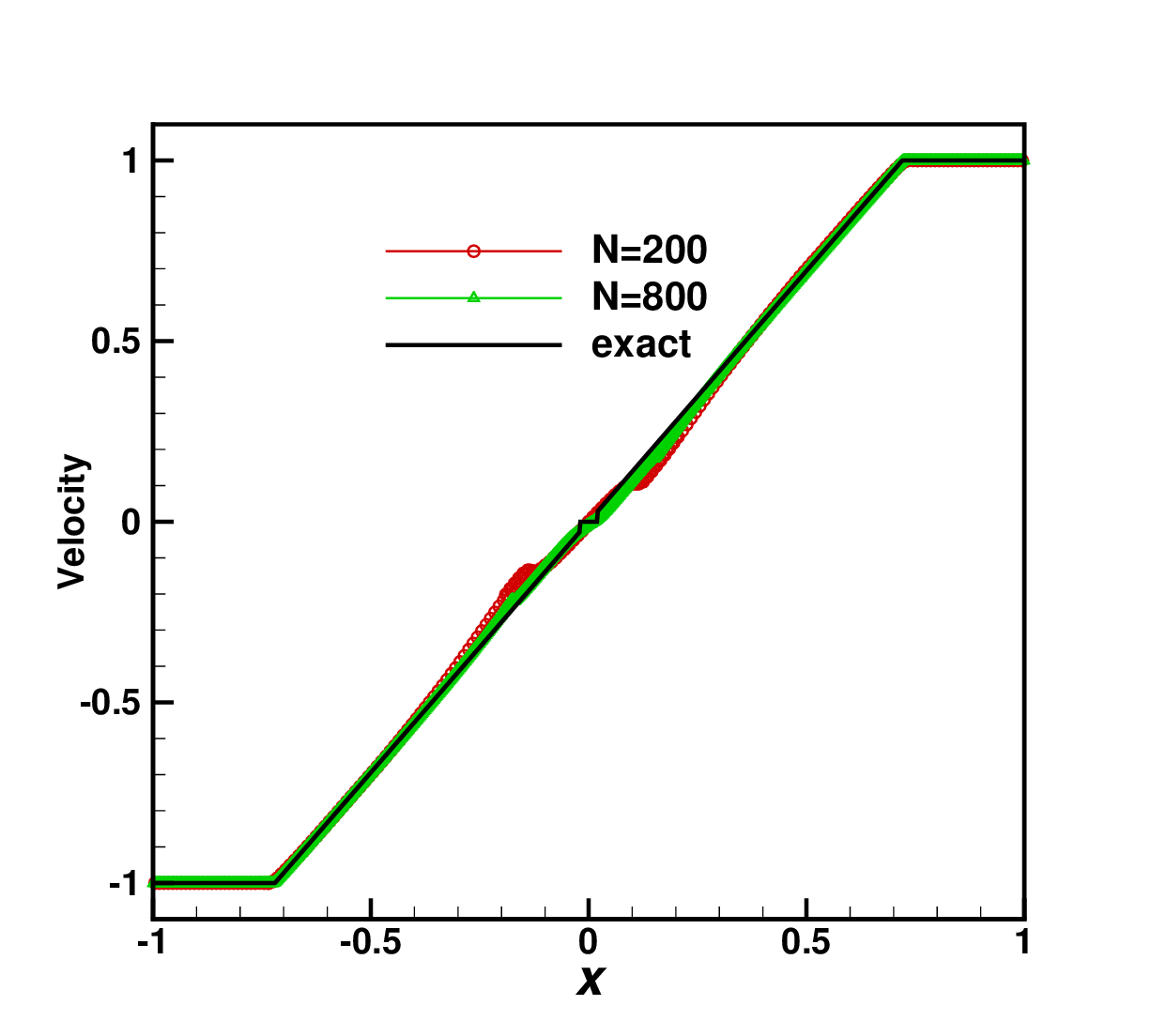}
\includegraphics[width=2.1in]{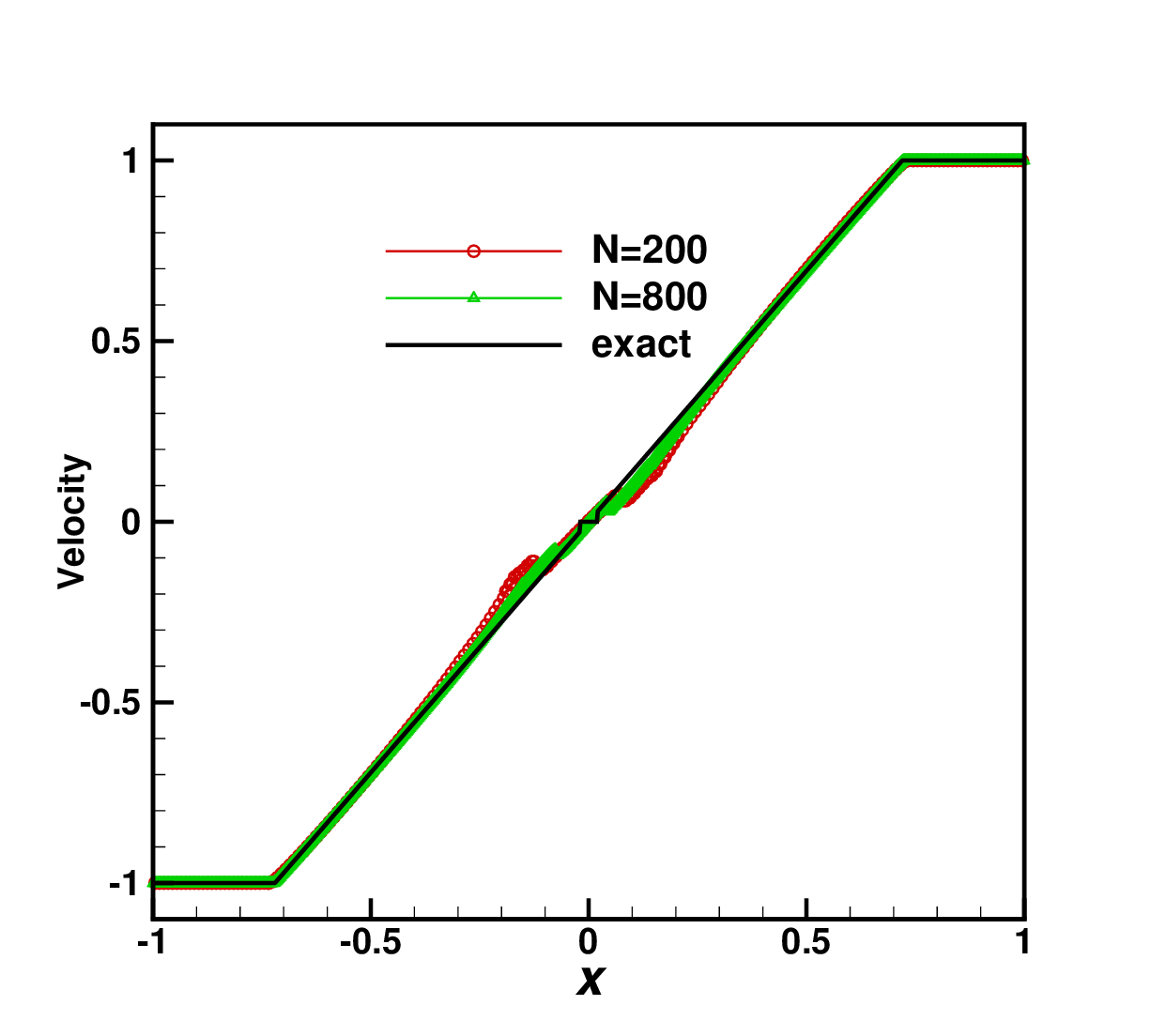}
\subfigure[DG, $P^2$]{
\includegraphics[width=2.05in]{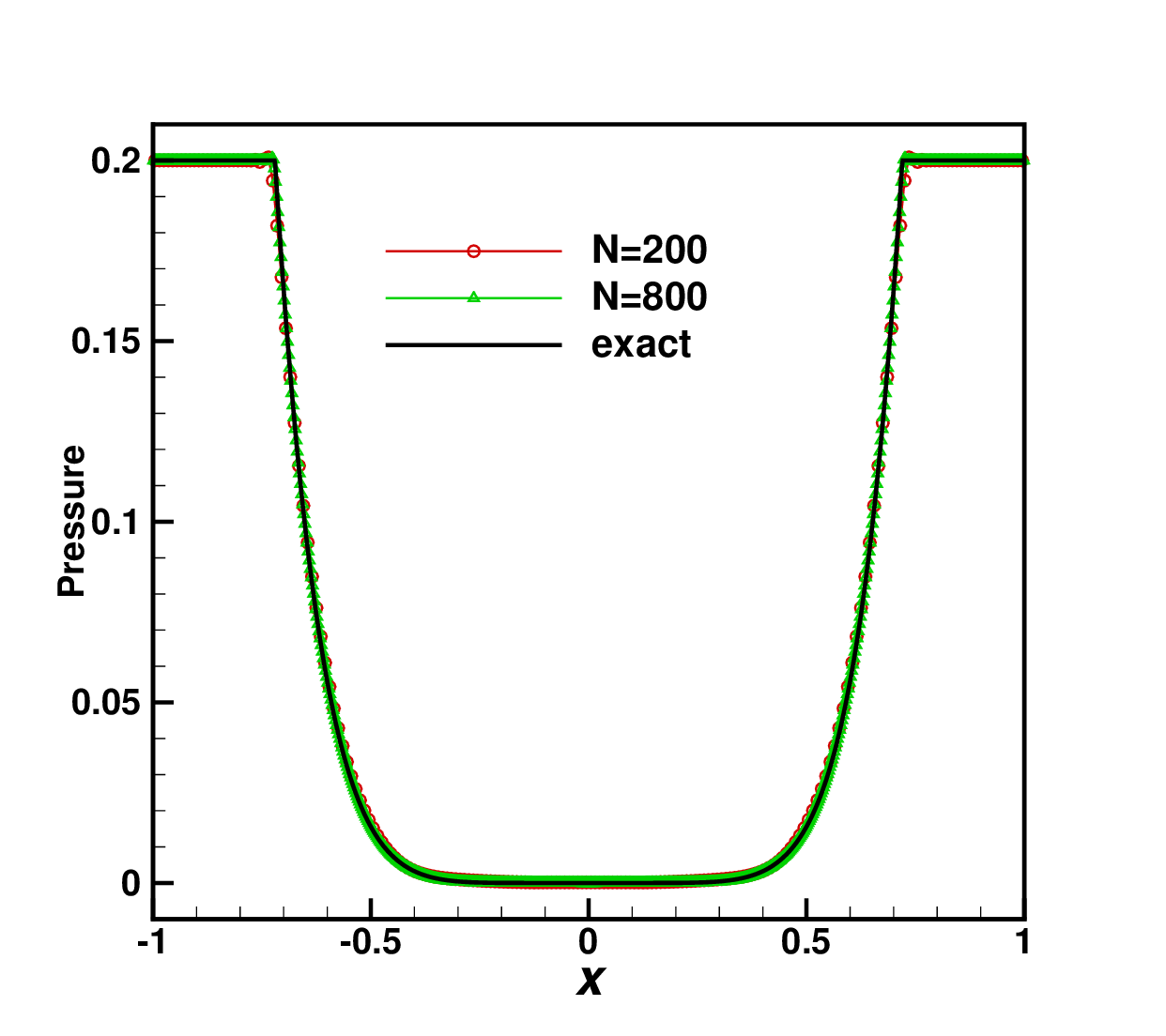}}
\subfigure[CutDG, $P^2$]{
\includegraphics[width=2.05in]{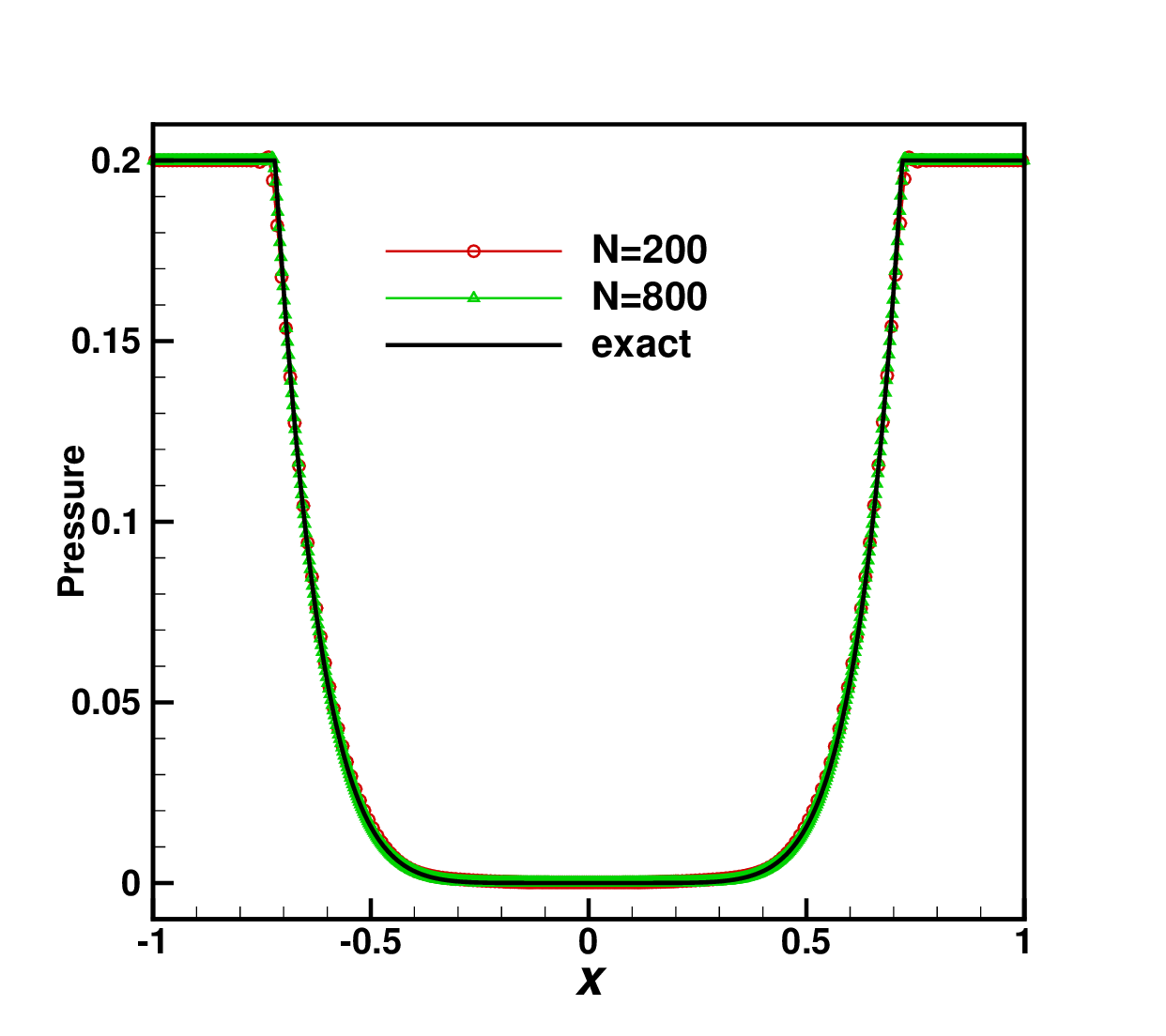}}
\subfigure[CutDG, $P^3$]{
\includegraphics[width=2.05in]{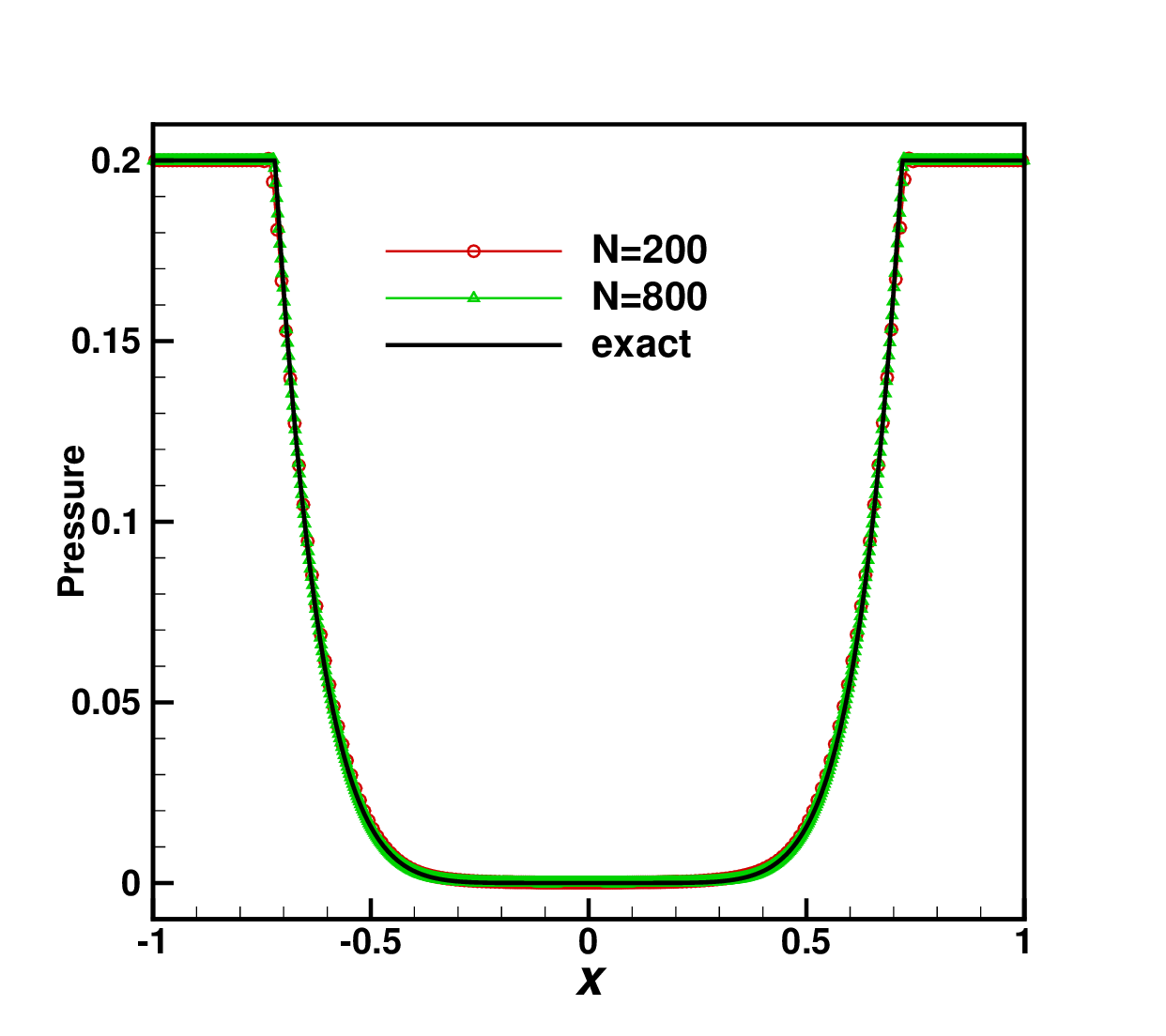}}
\caption{  DG and CutDG approximations  for the double rarefaction problem  on the uniform mesh (left) and cut (middle, right) meshes. Solid line: exact solution. Symbol: approximations.}\label{exe:RP1:1d}
\end{figure}

\subsubsection{Sedov blast wave problem}
We now test the Sedov blast wave problem, a typical low-density problem with strong shocks. The exact solution can be found in \cite{korobeinikov1991problems,sedov1993similarity}.  The computational domain is $[-2,2]$, and we use  $N=200,800$ elements on the background mesh.   We take the initial condition as
\begin{align*}
(\rho,u,E)=\left\{ 
    \begin{array}{lc}
         (1,0,3.2\cdot 10^6)& x \in [0,h],  \\
        (1,0,10^{-12})&\text{otherwise,}\\
    \end{array}
\right.
\end{align*}
where $h=4/N$. 
The elements in $[-0.5,0.5]$ are cut with $\alpha=0.01$. 
We simulate this problem up to time $t=0.001$ with $P^2$ approximations by our proposed scheme. We apply a  TVD limiter to the reconstructed approximation on the macro-elements and the positivity preserving limiter is applied to modify the limited approximation to make sure the solution has  positive density and pressure.  In the top panel of Figure \ref {exe:sedovblast:1d:macroonlybad}, we observe that our proposed Cut-DG method \eqref{scheme:cutDG2fully}-\eqref{eq:schemestep3}  together with the positivity preserving limiter in Section 4.3 can simulate this blast wave problem very well. To minimize errors and reduce computational costs, we also simulated this problem using the scheme \eqref{scheme:cutDG2fully}. The TVD limiter was applied to the unreconstructed solution on the macro-elements. Then, we reconstructed the solution on those macro-elements where the numerical solution had negative density or pressure, and a positivity-preserving limiter was applied to the reconstruction on the macro-elements. The results are shown in the bottom panel of Figure \ref{exe:sedovblast:1d:macroonlybad}, which demonstrate that the locally applied reconstruction can also simulate this problem effectively.
\begin{figure}[!htbp]
\centering
\includegraphics[width=2.0in]{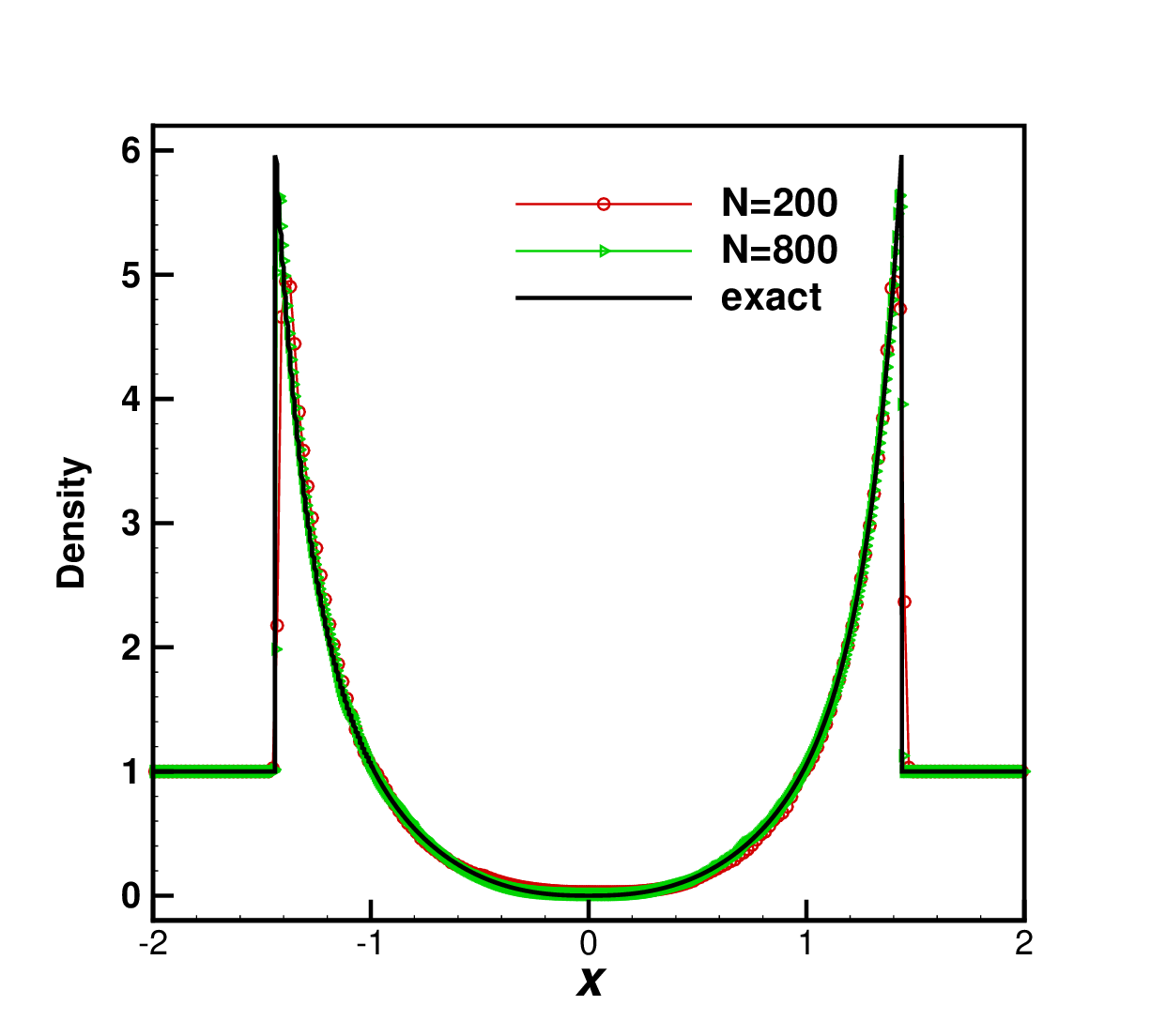}
\includegraphics[width=2.0in]{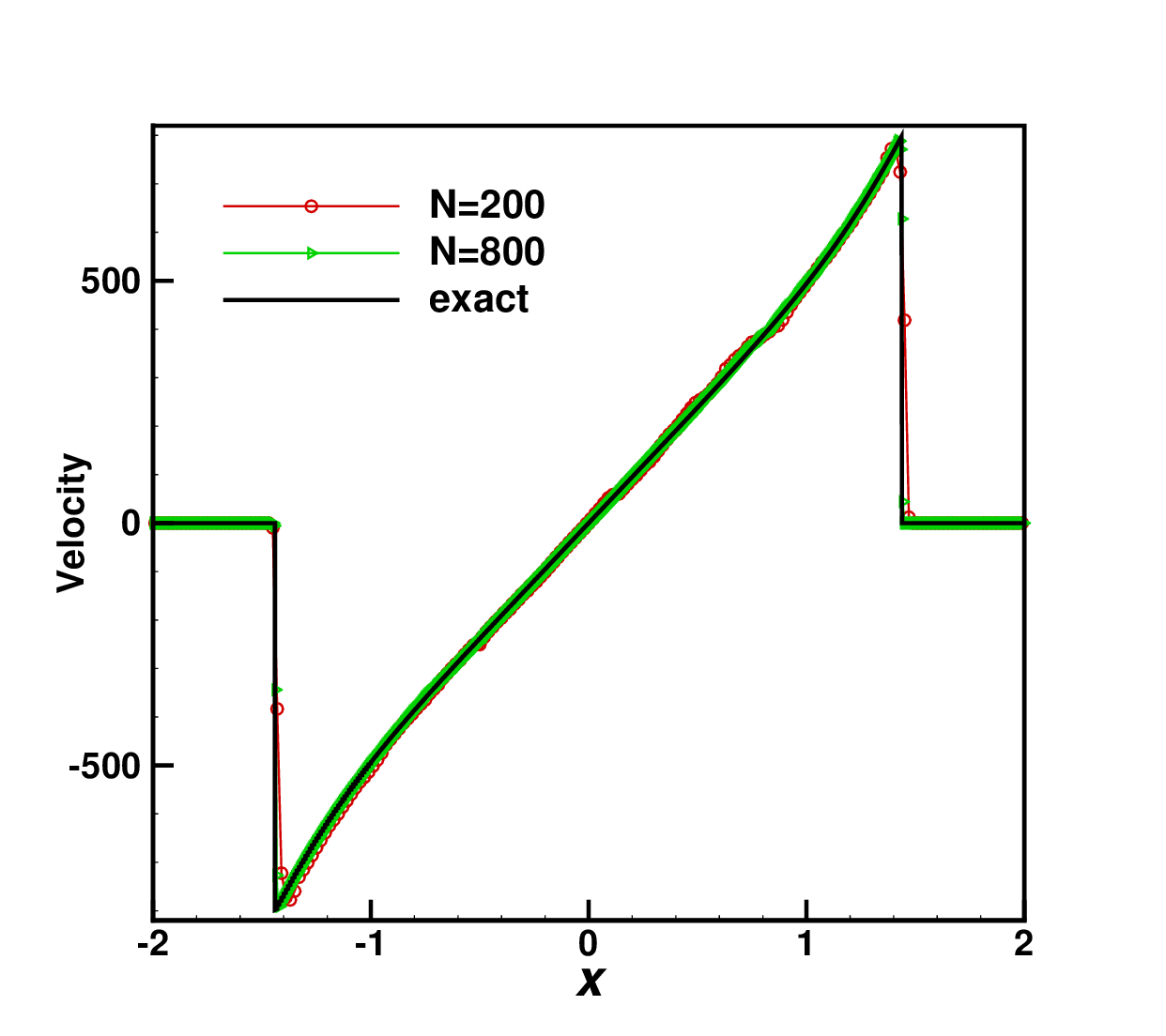}
\includegraphics[width=2.0in]{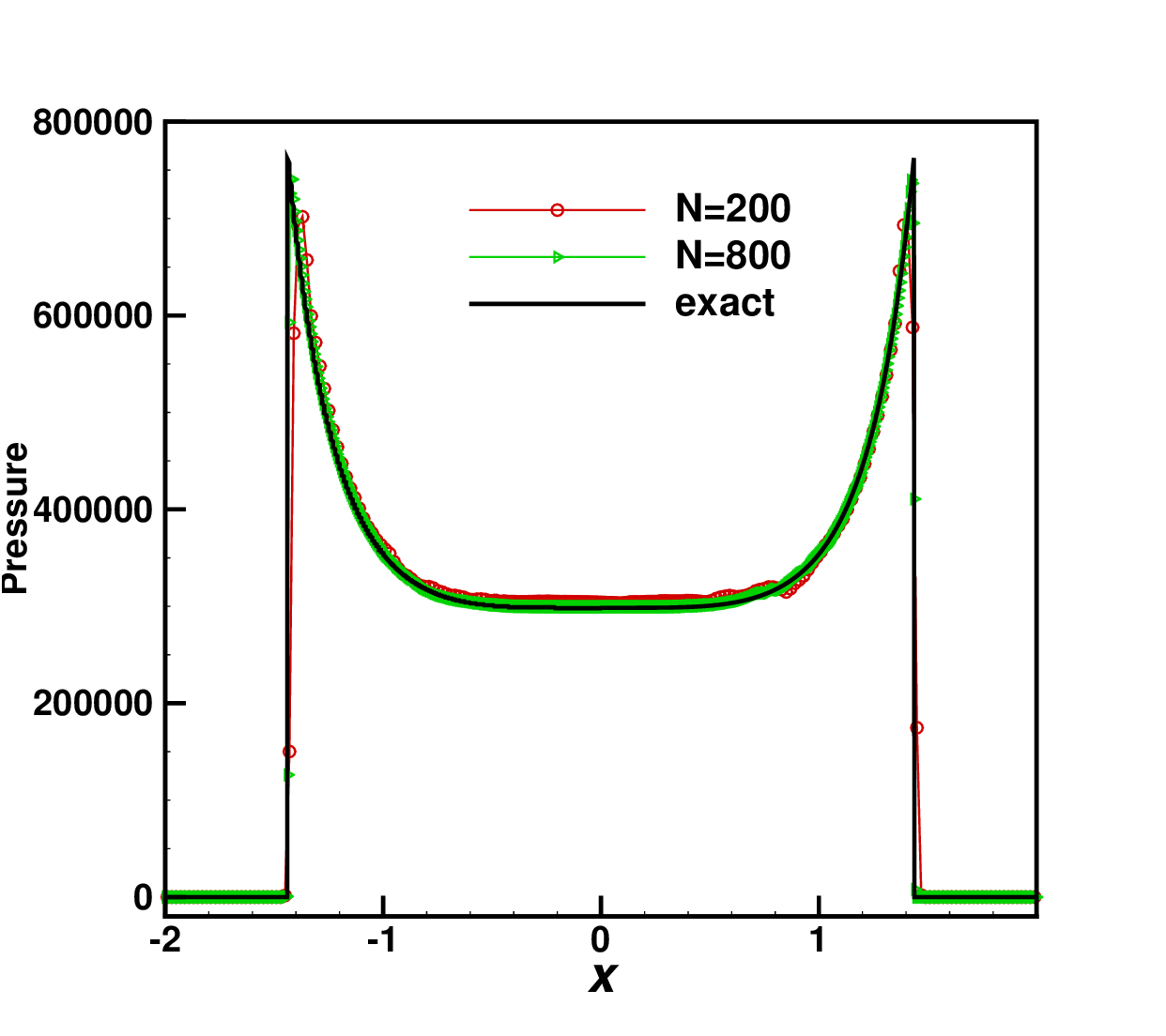}
\includegraphics[width=2.0in]{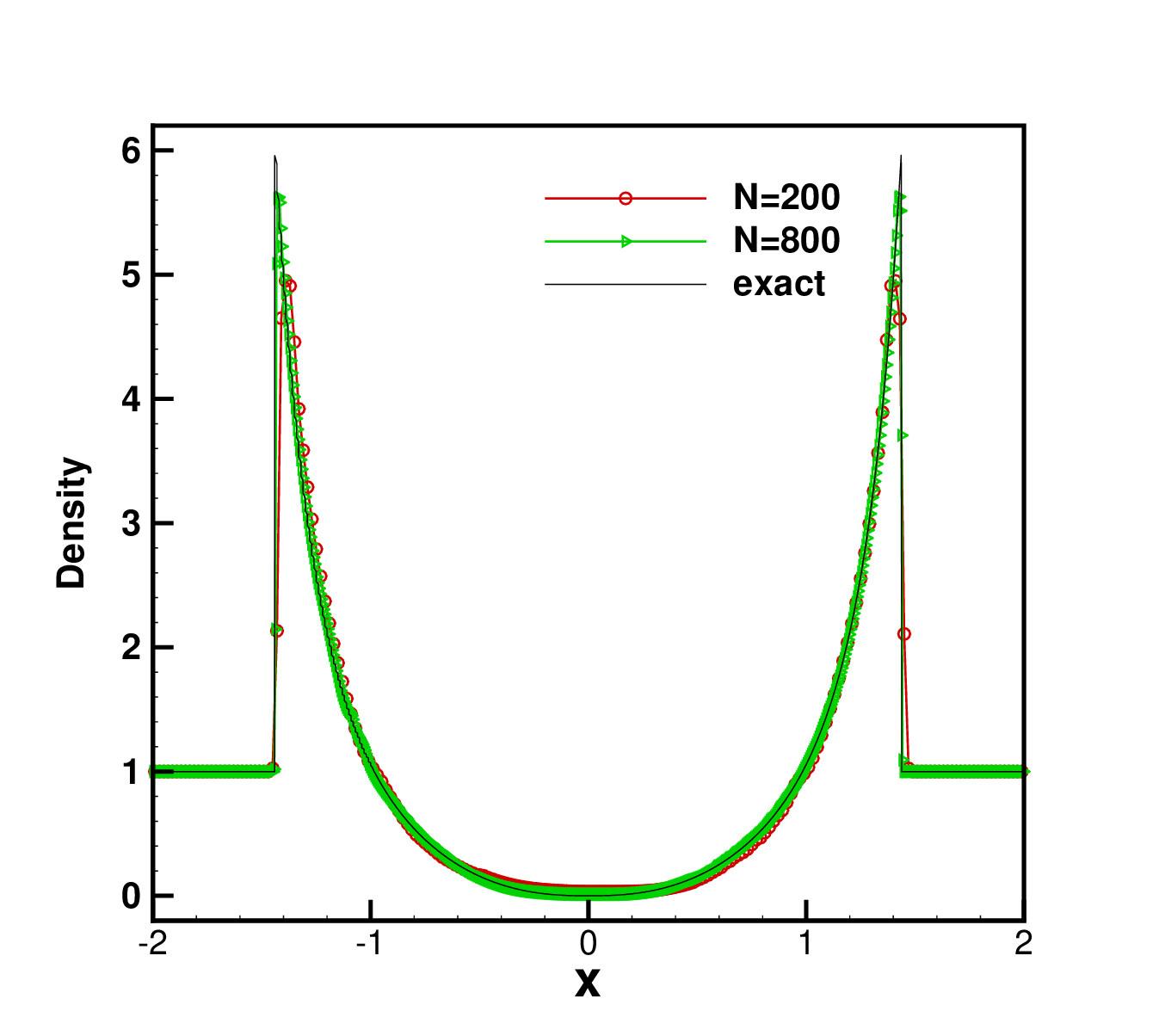}
\includegraphics[width=2.0in]{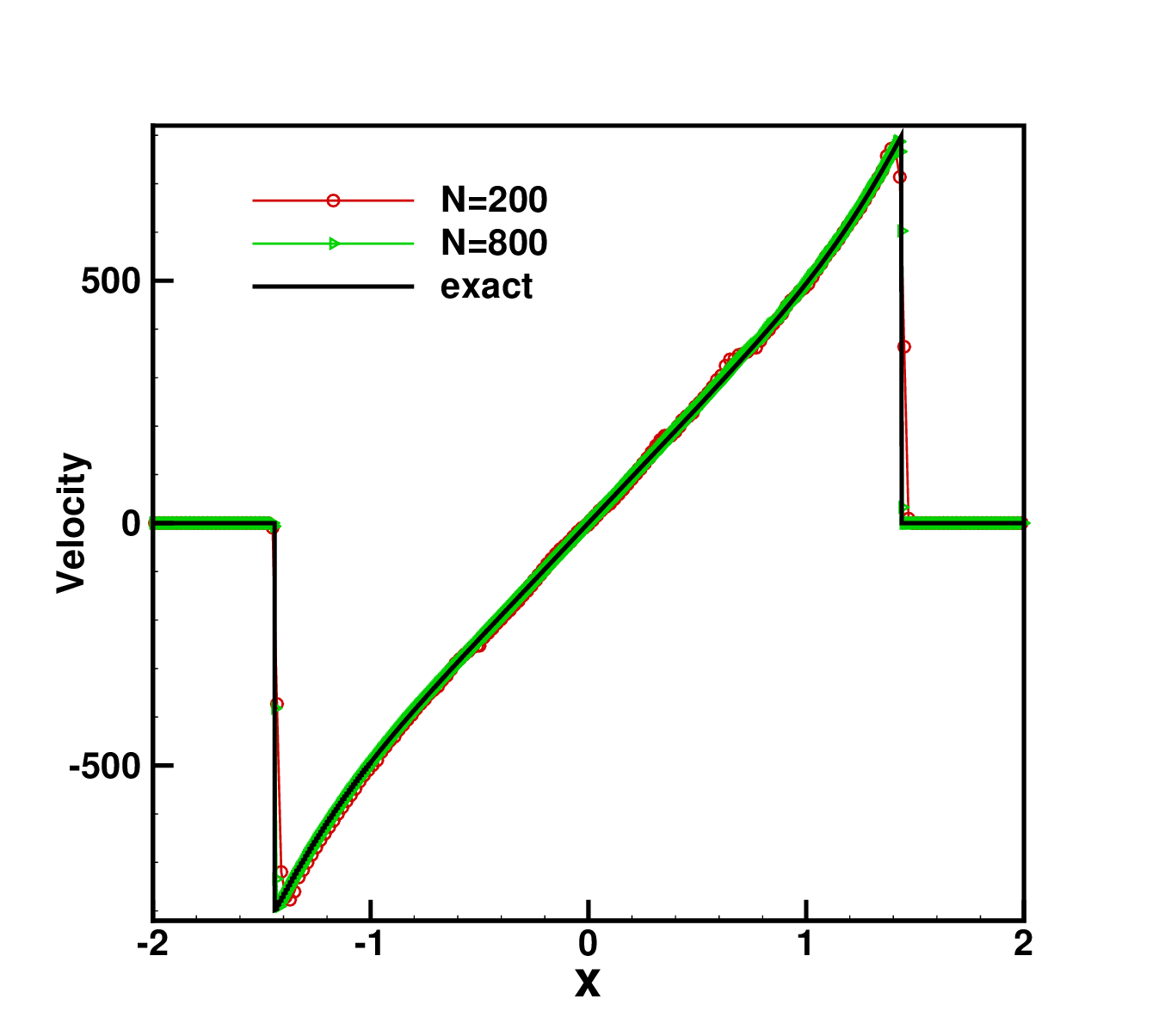}
\includegraphics[width=2.0in]{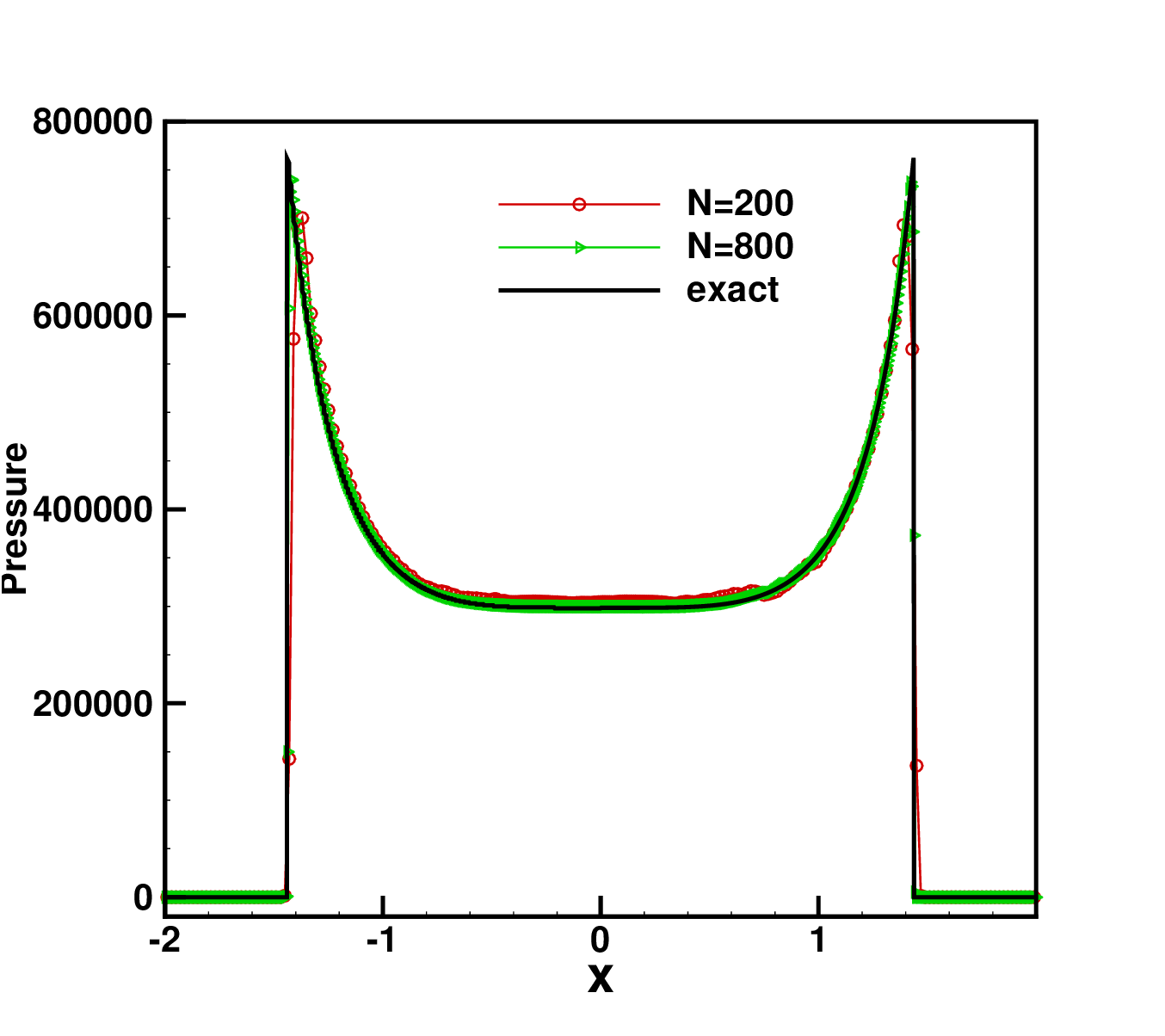}
\caption{The third order Cut-DG approximation of Sedov blast wave on the cut mesh.  The solid line is the exact solution and the symbols represent the Cut-DG solutions. Top: reconstruction on all macro-elements. Bottom: reconstruction only on macro-elements where the numerical solution has negative density or pressure.}\label{exe:sedovblast:1d:macroonlybad}
\end{figure} 

\subsubsection{The two blast wave problem}\label{sec:twoblastwave}
Finally, we consider the two blast wave problem for the Euler Equations with solid wall boundary conditions, see \cite{woodward1984numerical}. The solution contains strong shocks and rarefactions, which are reflected multiple times at the solid boundaries. There are a variety of interactions between shock waves and rarefaction waves, as well as between these waves and contact discontinuities. We consider the problem  on  the physical domain $[x_l,x_r]=[0,1]$, which is immersed in a background mesh. In this example there are only cut elements at the boundries.  We have constructed the mesh to be the same as in \cite{TAN20108144,YangInverseLB}, where the inverse Lax-Wendroff method is used.  
We take the computational domain as  $[x_L,x_R]$ with $x_L=x_l-(1-\alpha)h, x_R=x_r+(1-\alpha h)$ and a  uniform grid $x_L=x_{-\frac{1}{2}}<x_{\frac{1}{2}}<\cdots<x_{N-\frac{1}{2}}<x_{N+1/2}=x_R$.  As in \cite{YangInverseLB}, two unfitted small cells with size $\alpha h$ appear at the boundaries.  In our computations $\alpha=0.01$.    
The initial data at $t=0$ is  taken as 
\begin{align}
(\rho,u,p)= \begin{cases}
(1,0,10^3) & 0<x<0.1, \\ 
(1,0,10^{-2})& 0.1<x<0.9, \\
(1,0,10^2) & 0.9<x<1.
 \end{cases}
\end{align} 
At the boundaries, we use wall boundary conditions.   
We solve this problem by our proposed positivity-preserving scheme up to $t=0.038$. The numerical solution of density is shown in Figure \ref{fig:euler:twoblastwave}.  We use the solution from the third order standard DG method on uniform mesh  with $h = 1/3200$ as the reference solution.   We can see that our proposed scheme can capture the structure of the solution well and preserve the positivity of density and pressure. 
\begin{figure}[!htp]
\centering
\includegraphics[width=2.0in]{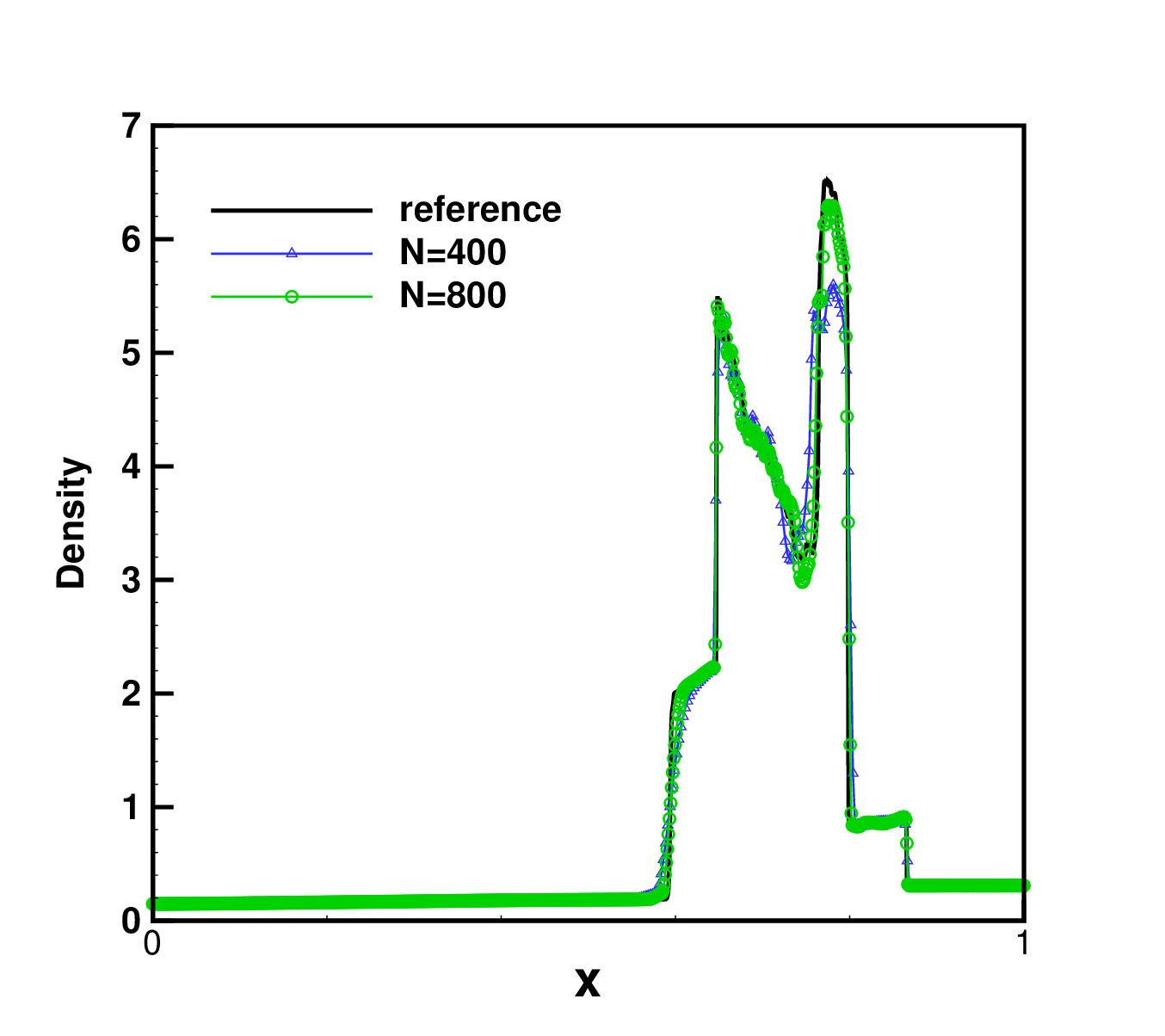}
\includegraphics[width=2.0in]{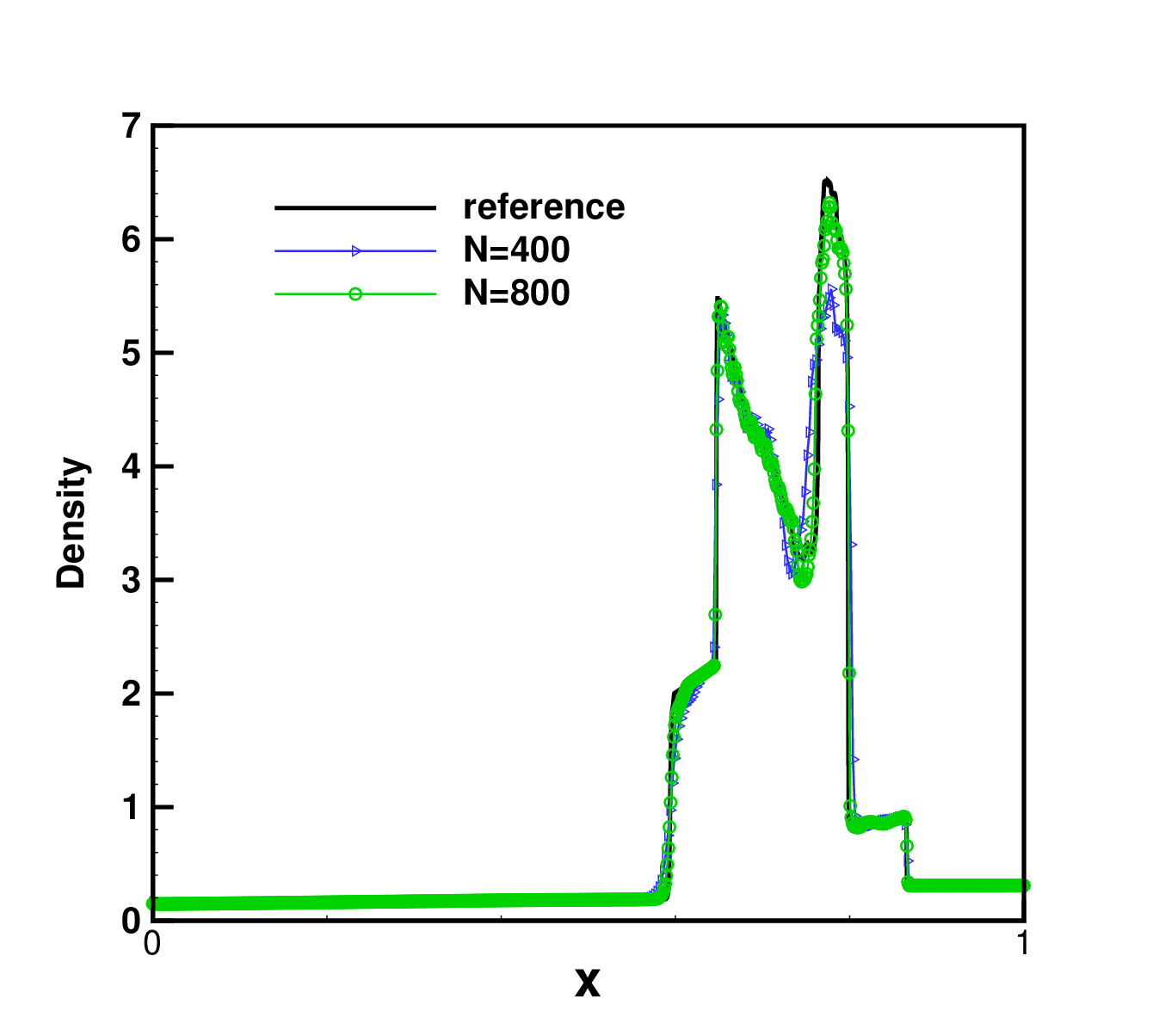}
\includegraphics[width=2.0in]{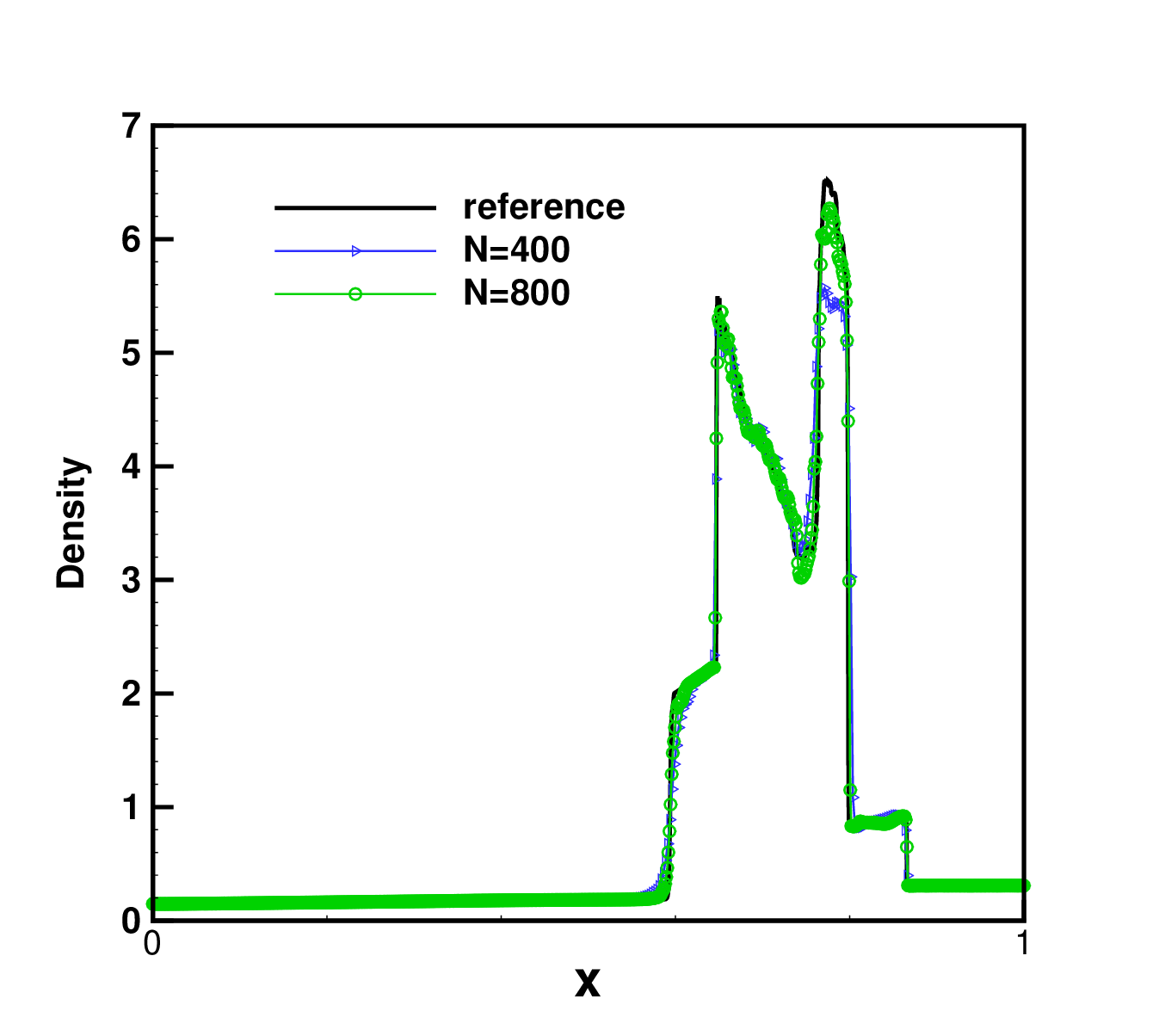}
\caption{Example 5.3: numerical approximations of density at $t=0.038$ for the two blast wave problem. Circle: numerical solution from our proposed scheme. Solid: reference solution from the standard DG method.  $P^1,P^2,P^3$ approximation from left to right. }\label{fig:euler:twoblastwave}
\end{figure}

\section{Conclusion}
In this paper, we present a family of bound preserving Cut-DG methods for hyperbolic conservation laws. The methods are based on explicit time-stepping with time-step restrictions on the same order as for the corresponding standard DG methods on un-cut meshes. The parameter $0<\delta \leq 1$ in \eqref{eq:largeel} regulates when cut-elements are stabilized, and influences the time-step restriction. By choosing $\delta=1$ we get exactly the same time-step restriction as in the standard un-cut case, but all cut elements are stabilized. 

A new reconstruction of the solution on macro-elements is introduced as a post-processing step in each inner stage of each time step. It ensures that our proposed piecewise constant scheme satisfies the maximum principle for scalar conservation laws and positivity of density and pressure for the system of Euler equations, respectively. The higher order versions of our proposed method include bound preserving limiting applied to the reconstruction on the macro-elements. To avoid oscillations when discontinuities are present, additional limiting is also included. Results from numerical computations demonstrate the good properties of our proposed methodology. The test cases include several challenging problems for the Euler equations, such as the Sedov blast wave and a double rarefaction problem. These examples include regions of low density and strong shocks. 

Extending the methodology to higher space dimensions is straightforward when no limiting is needed, and we have good preliminary results for smooth solutions. We also have good results when including bound preserving limiting for elements of polynomial orders zero and one. For polynomial orders higher than one, a good strategy for finding minimal/maximal values on irregular domains (which is needed in the bound preserving limiter) needs to be developed. We hope to present more complete results for problems in higher space dimensions in the future.

%\begin{acknowledgement}
%\noindent\textit{\textbf{Acknowledgements}} Pei Fu was partially supported by the Natural Science Foundation of Jiangsu Province, China under grants  BK20230869 and the National Natural Science Foundation of China under grants 12301470. Gunilla Kreiss was supported by the Swedish Research Council grant, VR2018-05279. Sara Zahedi was supported by the Wallenberg Academy Fellowship KAW 2019.0190.
% \end{acknowledgement}

\begin{appendix}
\section{Fully discrete algorithm for the Euler equations}
In this appendix, the fully discrete algorithm for the Euler equations is shown in Algorithm 2.
\begin{algorithm}[!hb]
\caption{Fully discrete bound preserving Cut-DG method for the Euler equations}\label{algorithm}
\KwData{Given initial condition $\bu_0(x)$ at $t=0$}
\KwResult{Bound preserving approximation at final time $T$}
Initialize $\tilde \bu_h^0$ by \eqref{eq:stabilize:L2} and $n=0$\;
$ {\bu}_h^{0}$ = Reconstruction and limiting$(\tilde {\bu}_h^{0})$\;
\While{$t\leq T$}{
Compute suitable time-step $\Delta t$ based on Theorem 4.1/4.2\;
\tcc {Runge-Kutta time discretization}
$\tilde {\bu}_h^{(1)}$ = Euler-Step (${\bu}_h^{n}$, $t_n$, $\Delta t$)\tcp*{first stage}
${\bu}_h^{(1)}$ = Reconstruction and limiting $(\tilde {\bu}_h^{(1)})$\;
${\bu}_h^{(2')}$ = Euler-Step(${\bu}_h^{(1)}$, $t_n+\Delta t$, $\Delta t$)\tcp* {second stage}
Set: $\tilde {\bu}_h^{(2)}=\frac{3}{4}{\bu}_h^{n}+\frac{1}{4}{\bu}_h^{(2')}$\;
${\bu}_h^{(2)}$ = Reconstruction and limiting $(\tilde {\bu}_h^{(2)})$\;
 $\tilde {\bu}_h^{(3')}$ = Euler-Step(${\bu}_h^{(2)}$, $t_n+\Delta t/2$, $\Delta t$)\tcp* {third stage}
Set: $\tilde {\bu}_h^{(3)}=\frac{1}{3}{\bu}_h^{n}+\frac{2}{3}{\bu}_h^{(3')}$\;
${\bu}_h^{(3)}$ = Reconstruction and limiting $(\tilde  {\bu}_h^{(3)})$\;
Set: ${\bu}_h^{n+1}= {\bu}_h^{(3)}$\;
$n\leftarrow n+1$\;
$t\leftarrow t+\Delta t$\;
}
% Write Function with word ``Function''
\SetKwFunction{proc}{Euler-step}
  \setcounter{AlgoLine}{0}
  \SetKwProg{eulerstep}{Function}{}{}
  \eulerstep{\proc{${\bu}_h$, $t$, $\Delta t$}}{
Using \eqref{scheme:cutDG2fully} with ${\bu}_h, t, \Delta t$ component-wise to get $\tilde{\bu}_h=(\tilde {\bu}_{h,1}, \tilde{\bu}_{h,2}, \cdots, \tilde{\bu}_{h,N_\Omega})\in V_h^p$ \;
 \KwRet $\tilde {\bu}_h=(\tilde {\bu}_{h,1}, \tilde{\bu}_{h,2}, \cdots, \tilde{\bu}_{h,N_\Omega})$\;
  }
  \SetKwFunction{procc}{Reconstruction and limiting}
  \setcounter{AlgoLine}{0}
  \SetKwProg{reconlimit}{Function}{}{}
  \reconlimit{\procc{$\tilde {\bu}_h$}}{
\For{$i=1$ \KwTo $N_\Omega$}{
Using \eqref{eq:scalar:reconstruction} with $\tilde {\bu}_h$ component-wise to reconstruct $\bu_{h,i}^{\mathcal{M}}, i=1,\cdots,N_\Omega$\;
\lIf{$\bu_{h,i}^{\mathcal{M}}|_{\Omega_i}\notin G$}{
modify $\bu_{h,i}^{\mathcal{M}}$ by  \eqref{euler:limit:rho}-\eqref{euler:limiter:p}}
$u_{h,i}\leftarrow u_{h,i}^{\mathcal{M}}$\;
}
 \KwRet ${\bu}_h=(u_{h,1}, {\bu}_{h,2}, \cdots, {\bu}_{h,N_\Omega})\in V_h^p$\;
  }
\end{algorithm}
\end{appendix}

\renewcommand\refname{Reference}
%\small
%\bibliographystyle{siamplain}
\bibliographystyle{abbrv}
\bibliography{CutDG}

\end{document}